\documentclass[11pt]{article}

\usepackage{amsmath}
\usepackage[margin=1in]{geometry}
\usepackage{amsmath,appendix}
\usepackage{appendix}
\usepackage{amsthm}
\usepackage{bbm}
\usepackage{amssymb}
\usepackage{setspace}
\usepackage{color}
\usepackage{float}
\usepackage{dsfont}
\usepackage{enumerate}
\usepackage{hyperref}
\usepackage{graphicx}
\usepackage{algpseudocode,algorithm,algorithmicx}
\usepackage{tikz}
\usepackage{subcaption} 
\usepackage{wrapfig}
\usepackage{bbm}
\usepackage{thmtools}
\usepackage{mathabx}
\usepackage[capitalize]{cleveref}

\newtheorem{theorem}{Theorem}
\newtheorem{proposition}[theorem]{Proposition}
\newtheorem{lemma}[theorem]{Lemma}

\theoremstyle{definition}
\newtheorem{definition}{Definition}[section]
\newtheorem{assumption}{Assumption}
\newtheorem{remark}{Remark}[section]
\numberwithin{theorem}{section} 
\numberwithin{lemma}{section} 

\numberwithin{equation}{section}

\usepackage{setspace}

\def\C{\mathbb{C}}
\def\E{\mathbb{E}}
\def\N{\mathbb{N}}
\def\P{\mathbb{P}}
\def\R{\mathbb{R}}
\def\S{\mathbb{S}}
\def\T{\mathbb{T}}
\def\Z{\mathbb{Z}}

\renewcommand{\supset}{\supseteq}
\renewcommand{\subset}{\subseteq}
\renewcommand{\hat}{\widehat}
\renewcommand{\tilde}{\widetilde}
\renewcommand{\epsilon}{\varepsilon}

\def\dist{{\rm dist}}
\def\trace{{\rm tr}}

\newcommand{\var}{{\rm var}}

\renewcommand{\Re}{{\rm Re}}
\renewcommand{\Im}{{\rm Im}}

\newcommand{\mesh}{{\rm mesh}}
\newcommand{\commentout}[1]{}
\newcommand{\bbone}{\mathbbm{1}}

\newcommand{\domain}{\Omega}
\newcommand{\sampset}{X}
\newcommand{\bfa}{a}

\newcommand{\calS}{\mathcal{S}}

\newcommand{\calN}{\mathcal{N}}
\newcommand{\calP}{\mathcal{P}}

\newcommand{\calY}{\mathcal{Y}}

\newcommand{\diag}{{\rm diag}}

\newcommand{\spann}{{\rm span}}

\def\wherespace{\quad\text{where}\quad}

\def\forallspace{\quad\text{for all}\quad}
\def\andspace{\quad\text{and}\quad}

\newcommand{\norm}[1]{{\left\|#1\right\|}}
\newcommand{\round}[1]{{\left(#1\right)}}
\newcommand{\inner}[1]{{\left\langle#1\right\rangle}}

\usepackage{colordvi}
\usepackage{color}

\numberwithin{equation}{section}

\title{Multidimensional Gradient-MUSIC: A Global Nonconvex Optimization Framework for Optimal Resolution}
\author{Albert Fannjiang\footnote{University of California, Davis. Email: cafannjiang@ucdavis.edu} \and Weilin Li\footnote{City University of New York, City College. Email: wli6@ccny.cuny.edu}}

\begin{document}
	
	\maketitle
\begin{abstract}
We develop a multidimensional version of \emph{Gradient-MUSIC} for estimating the frequencies of a nonharmonic signal from noisy samples. The guiding principle is that frequency recovery should be based only on the signal subspace determined by the data. From this viewpoint, the MUSIC functional is an economical nonconvex objective encoding the relevant information, and the problem becomes one of understanding the geometry of its perturbed landscape.

Our main contribution is a general structural theory showing that, under explicit conditions on the measurement kernel and the perturbation of the signal subspace, the perturbed MUSIC function is an admissible optimization landscape: suitable initial points can be found efficiently by coarse thresholding, gradient descent converges to the relevant local minima, and these minima obey quantitative error bounds. Thus the theory is not merely existential; it provides a constructive global optimization framework for multidimensional optimal resolution.

We verify the abstract conditions in detail for two canonical sampling geometries: discrete samples on a cube and continuous samples on a ball. In both cases we obtain uniform, nonasymptotic recovery guarantees under deterministic as well as stochastic noise. In particular, for lattice samples in a cube of side length $4m$, if the true frequencies are separated by at least $\beta_d/m$ and the noise has $\ell^\infty$ norm at most $\epsilon$, then Gradient-MUSIC recovers the frequencies with error at most
\[
C_d \frac{\epsilon}{m},
\]
where $C_d, \beta_d>0$ depend only on the dimension. This scaling is minimax optimal in $m$ and $\epsilon$. Under stationary Gaussian noise, the error improves to
\[
C_d\frac{\sigma\sqrt{\log(m)}}{m^{1+d/2}}.
\]
This is the \emph{noisy super-resolution scaling}: rather than referring to resolution below a classical diffraction threshold, it quantifies the enhanced decay of localization error with sampling diameter in the presence of random noise.
\end{abstract}

	\medskip
	\noindent
	{\bf 2020 MSC:} 42A05, 42A10, 90C26, 94A12
	
	\medskip
	\noindent
	{\bf Keywords:} Spectral estimation, MUSIC, nonconvex optimization, global convergence, gradient descent, geometric analysis, minimax rates, optimality, super-resolution

	\setcounter{tocdepth}{1}
	\tableofcontents

		\section{Introduction}

Spectral estimation is the inverse problem of recovering latent frequencies from high-dimensional, and typically noisy, observations of superposed non-harmonic Fourier modes. Although the forward model is linear in the amplitudes, the recovery of the frequencies themselves is intrinsically nonlinear. In this sense, spectral estimation belongs to the broader class of nonlinear inverse problems in which the physically meaningful parameters are not directly accessible through Fourier inversion, but must instead be inferred from the geometry of the measurement data.

More specifically, let the noiseless measurement be described by the function $y\colon \R^d\to\C$,
\begin{equation}
\label{eq:y}
y(x)=\sum_{\ell=1}^s a_\ell e^{2\pi i \theta_\ell\cdot x},
\qquad 
\text{for some }
\vartheta:=\{\theta_\ell\}_{\ell=1}^s\subset \domain
\andspace
\bfa:=\{a_\ell\}_{\ell=1}^s\subset \C\setminus\{0\},
\end{equation}
where $d\ge 1$ is the ambient dimension and $\domain$ is either the torus $\T^d=(\R/\Z)^d$ or a ball in $\R^d$. The unknown parameters are the number of sources $s$, their locations $\vartheta$, and their amplitudes $\bfa$.

In practice, one does not observe $y$ on all of $\R^d$, but only on a prescribed sampling set $\sampset_\star\subset \R^d$. Thus the noiseless data are given by the restriction $y|_{\sampset_\star}$. In addition, the measurements are corrupted by noise. Writing $\eta\colon \R^d\to\C$ for the noise and $\tilde y=y+\eta$ for the corrupted signal, the observed data take the form
\begin{equation}
\label{eq:ytilde}
\tilde y|_{\sampset_\star}.
\end{equation}
The multidimensional spectral estimation problem is to recover, or approximate, the unknown parameters $(s,\vartheta,\bfa)$ from the data \eqref{eq:ytilde}. Since only the relative size of signal and noise matters, we may normalize and assume without loss of generality that
\[
\min_\ell |a_\ell|=1.
\]

Models of the form \eqref{eq:y}--\eqref{eq:ytilde} arise throughout signal processing, imaging, inverse scattering and,  more recently, quantum computing \cite{helsen2022general, LinTong2022HeisenbergLimited, Somma2019QuantumEigenvalue,  stroeks2022spectral}. When $\domain$ parametrizes frequency, the $\theta_\ell$ represent latent spectral components, as in multivariate spectral analysis. When $\domain$ represents physical space or angular variables, the $\theta_\ell$ may encode source locations or directions of arrival, as in radar, array imaging, and inverse scattering. In all such settings, the parameters are generally \emph{off-grid}: the set $\vartheta$ is an arbitrary discrete subset of $\domain$, not assumed a priori to lie on a prescribed lattice. This point is essential. Grid-based sparse methods replace the continuum by a finite dictionary, but this introduces basis mismatch whenever the true parameters do not coincide with grid points \cite{fannjiang2012coherence}. The resulting discretization error is not merely numerical; it changes the inverse problem itself.

The starting point of the present work is the observation that, although the map \eqref{eq:y} is nonlinear in the parameters $(s,\vartheta,\bfa)$, the information relevant for recovering $\vartheta$ is carried by a much simpler object: the signal subspace. Indeed, if one defines
\[
\phi_\omega(x)=e^{2\pi i \omega\cdot x},
\]
then in the noiseless setting the measurement model is generated by the span of the steering vectors $\{\phi_{\theta_\ell}\}_{\ell=1}^s$. Let
\[
U:=\spann\{\phi_{\theta_1},\dots,\phi_{\theta_s}\},
\]
and let $P_U$ denote the orthogonal projection onto $U$. A natural surrogate for parameter recovery is then provided by the functional
\begin{equation}
\label{eq:q}
q(\omega)=\norm{(I-P_U)\phi_\omega}^2.
\end{equation}
We refer to any function of the form \eqref{eq:q} as a MUSIC function, after the classical MUltiple SIgnal Classification method of Schmidt \cite{schmidt1986multiple}. In the noiseless case, $q(\theta_\ell)=0$ for each $\ell$, so the ground truth $\vartheta$ is encoded by the set of global minima of the MUSIC landscape.

In practice, however, the subspace $U$ is not known exactly and must be estimated from the noisy data. If $\tilde U$ denotes a suitable estimator of $U$, then one is led to the perturbed MUSIC functional
\begin{equation}
\label{eq:qtilde}
\tilde q(\omega)=\norm{(I-P_{\tilde U})\phi_\omega}^2,
\end{equation}
whose smallest $s$ minima $\tilde\vartheta:=\{\tilde\theta_\ell\}_{\ell=1}^s$ are taken as estimates of the unknown parameters. This formulation achieves a striking reduction in complexity. Rather than optimizing over the full parameter space $(\vartheta,\bfa)\in \R^{sd}\times \C^s$, one optimizes a scalar function on $\R^d$, independently of the number of sources. The reduction is made possible by separating the reconstruction into two steps: first estimate the signal subspace, then recover the parameters by probing that subspace with the steering family $\phi_\omega$.

This dimensional reduction is one of the chief attractions of subspace methods, but it also creates the central difficulty addressed in this paper. The function $\tilde q$ is typically nonconvex, and one seeks not one minimizer but the smallest $s$ local minima. The classical MUSIC algorithm bypasses this difficulty by evaluating $\tilde q$ on a fine grid in $\domain$ and identifying approximate minima by exhaustive search \cite{schmidt1986multiple,stoica1989music,liao2016music}. Such a procedure is computationally tied to the desired accuracy and becomes rapidly prohibitive in dimension $d\ge 2$: a grid of mesh width $\delta$ contains $O(\delta^{-d})$ points. This curse of dimensionality has long limited the use of MUSIC in multidimensional problems and motivates the need for a genuinely optimization-based treatment.

At first sight, however, this prescription is far from self-evident. Why should $q_{\tilde U}$ even possess at least $s$ relevant local minima? Why should these minima lie near the true configuration $\vartheta$? And, from a computational point of view, how can they be found efficiently, given that $q_{\tilde U}$ is nonconvex? Classical MUSIC addresses this difficulty through an expensive brute-force search over a fine grid, precisely because it lacks a principled global initialization strategy.

The present paper develops such a treatment. Our goal is to formulate a global optimization framework for MUSIC that is both computationally efficient and mathematically transparent. The guiding principle is that one should work with the minimal information structurally required for reconstruction, namely, the signal subspace. From this perspective, the MUSIC function is not merely a classical heuristic; it is an economical nonconvex functional encoding the geometric information relevant to the inverse problem. The question is then whether this landscape possesses sufficient global regularity to permit reliable optimization.

This question should be viewed against the general background of ill-posed inverse problems. In the sense of Hadamard, a problem is well-posed if a solution exists, is unique, and depends continuously on the data. Spectral estimation typically violates these properties in several ways. The forward map may fail to be injective under insufficient sampling; noise may move the data away from the range of the ideal model; and even when the parameters remain identifiable, reconstruction may be unstable. From an optimization viewpoint, one faces an additional obstacle: a nonconvex objective may have many local minima unrelated to the ground truth. For MUSIC, this issue is especially acute because one must locate multiple minima globally and distinguish the relevant ones from spurious critical points.

To make this difficulty concrete, one may recall that closeness of objective functions does not by itself imply closeness of their minimizers. For example, the two functions
\[
f(t)=\epsilon^2|t|,
\qquad
g(t)=\epsilon^2|t-\epsilon^{-1}|,
\]
satisfy $\sup_{t\in\R}|f(t)-g(t)|\le \epsilon$, yet their minimizers are separated by distance $\epsilon^{-1}$. Thus even a small perturbation of the objective may drastically alter the location of the minimum. 
The source of this instability is the flatness of the landscape near the minimum which results in sensitivity to perturbations. This example shows that pointwise or uniform control of the objective is insufficient for stable recovery: one also needs geometric information about the basin of attraction, such as curvature or a quantitative rate of increase away from the minimum.
 In inverse problems, stability of the reconstruction therefore requires much more than pointwise closeness of data or functionals; it requires structural control of the optimization landscape itself.

A large body of literature approaches spectral estimation through other optimization principles. Classical methods such as nonlinear least squares and maximum likelihood estimation, when the noise model is known, optimize objective functions on the full high-dimensional parameter space $(\vartheta,\bfa)\in \R^{sd}\times \C^s$; see, for example, \cite{stoica2002maximum,traonmilin2020basins}. These formulations are likewise nonconvex and usually require carefully chosen initializations in a space whose dimension grows with the number of sources. At the other extreme, continuous compressed sensing and atomic norm methods embed the problem into an infinite-dimensional convex program over purely atomic measures \cite{candes2014towards}. In higher dimensions, however, numerical implementation typically requires either discretization in the primal domain or semidefinite representations in the dual domain \cite{de2016exact,poon2019multidimensional,poon2023geometry}, both of which are computationally intensive.

The MUSIC framework occupies a different position. Once the signal subspace has been estimated, the recovery problem is reduced to a function on $\R^d$, regardless of the number of sources. In return, one must understand the geometry of a nonconvex landscape and devise a principled global search strategy. A basic requirement for constructing the signal subspace is that there exist sets $X,X'\subset \R^d$ such that $X+X'\subset X_\star$, where $X+X'$ denotes the Minkowski sum. Under such a condition, one can assemble structured data matrices whose range encodes the desired subspace. The fundamental challenge is then no longer dimensionality, but landscape geometry: why should one expect the MUSIC function to be globally navigable by efficient local methods?

\subsection{Informal sketch of our contributions}
Our answer, building on our earlier one-dimensional work \cite{fannjiang2025optimality}, is that the MUSIC landscape possesses a global structure that is both analyzable and exploitable. The main theorem (\cref{thm:maingeometric}) of the paper identifies a class of \emph{admissible landscapes} for which simultaneous multiple initialization followed by gradient descent recovers the correct configuration. 
The geometry of an admissible landscape is highly structured, see for example, \cref{fig:example}. Regions where the MUSIC function is large are precisely those most susceptible to perturbation by noise, whereas regions where it is small---most importantly, the basins of attraction surrounding the relevant minima---remain comparatively stable. In particular, the size of these favorable regions does not collapse as the noise level increases within the admissible regime. This separation between unstable and stable parts of the landscape is the mechanism that makes global initialization by coarse thresholding, followed by local descent, both reliable and computationally effective. In this way, the theorem connects three issues that are often studied separately: the perturbation of subspace information by noise, the geometry of the resulting nonconvex objective, and the design of a computationally efficient recovery algorithm.

\begin{figure}[ht]
\centering
\begin{subfigure}{0.45\textwidth}
\includegraphics[width=\textwidth]{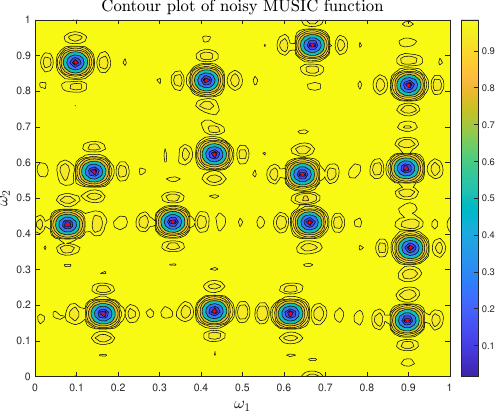}
\end{subfigure}
\begin{subfigure}{0.45\textwidth}
\includegraphics[width=\textwidth]{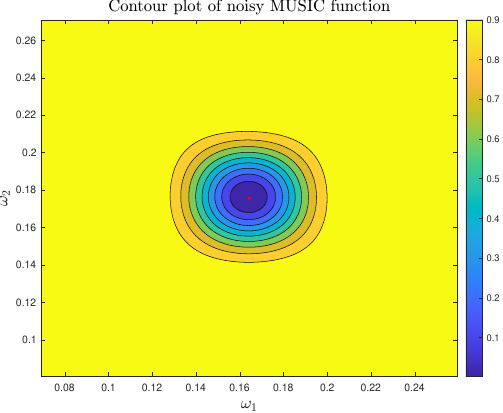}
\end{subfigure}
\caption{Contour plot of a perturbed MUSIC function $\tilde q$. The red dots indicate the true parameters $\{\theta_\ell\}_{\ell=1}^{16}$, assumed separated by at least $1/8$, while the amplitudes $\{a_\ell\}_{\ell=1}^{16}$ are chosen independently from $\{\pm1\}$. Samples are collected on the discrete square $Q_{10}\cap\Z^2$ and corrupted by i.i.d.\ Gaussian noise with mean zero and variance one. The smallest local minima of the MUSIC function closely track the true parameter configuration.}
\label{fig:example}
\end{figure}

A second main theme of the paper is optimal noise stability. Even if one can compute the relevant minima of $\tilde q$, one must still ask whether these minima provide good approximations of the true parameters. This leads naturally to a minimax viewpoint. Given an estimator $\hat\vartheta=\{\hat\theta_\ell\}_{\ell=1}^s$ for $\vartheta=\{\theta_\ell\}_{\ell=1}^s$, we measure the reconstruction error by
\[
\min_{\psi \text{ permutation}}
\max_{\ell=1,\dots,s}
|\theta_\ell-\hat\theta_{\psi(\ell)}|.
\]
For notational convenience, one may re-index $\hat\vartheta$ so that the optimal permutation is the identity. The natural class of configurations is determined by a lower bound
\begin{equation}
\label{eq:separation}
\Delta(\vartheta)\gtrsim_d \big(\hbox{Sampling Scale}\big)^{-1}. 
\end{equation}
where 
\begin{equation}
\label{eq:minsep}
\Delta
:=\Delta(\vartheta)
:=\min_{j\neq k} |\theta_j-\theta_k|
\end{equation}
is the minimum separation among the point objects.
Traditional criteria such as the half-width or FWHM rule \cite{born_wolf_1999} provide an ad hoc notion of distinguishability, but they do not explicitly quantify stability under measurement perturbation. Donoho's formulation of super-resolution \cite{donoho1992superresolution} shifts the emphasis to inequalities of the form
\[
\hbox{Error} \le \hbox{Constant}\times \hbox{Noise Level},
\]
thereby making the scaling of the stability constant itself part of the problem. The results of the present paper should be understood in precisely this spirit. 
A principal objective of this paper is to determine how the optimal reconstruction error scales with the sampling diameter, the separation $\Delta$, and the noise level, under both adversarial and stochastic noise models.

To state our results concisely, we use the notation $x\lesssim y$ to mean that $x\le Cy$ for an absolute constant $C$, and similarly for $x\gtrsim y$. We write $x\asymp y$ when both inequalities hold. If the implicit constant depends only on the dimension $d$, we write $x\lesssim_d y$, $x\gtrsim_d y$, and $x\asymp_d y$.

The conclusions reached in this work may be summarized informally as follows. First, the MUSIC landscape admits a structural description in terms of subspace perturbation, and this description yields a constructive global optimization procedure (\cref{thm:maingeometric}). Second, for natural multidimensional sampling geometries such as discrete cubes (\cref{thm:maincube}) and continuous balls (\cref{thm:mainball}), Gradient-MUSIC achieves the worst-case error scaling
\[
\hbox{(Location) Error}\lesssim
\big(\hbox{Sampling {Scale}}\big)^{-1}\times\hbox{Noise Level}
\]
under adversarial noise, and the improved scaling
\[
\hbox{(Location) Error}\lesssim
\big(\hbox{Sampling {Scale}}\big)^{-1-d/2}\times \hbox{Noise Level}
\]
up to logarithmic factors under Gaussian stationary noise.  
This is the multidimensional analogue of the {\em noisy super-resolution scaling}, identified for one-dimensional ESPRIT in \cite{ding2024esprit}. Up to logarithmic factors, this matches the optimal error scaling by the Cram\'er--Rao lower bound in \cite{stoica1989music} when $d=1$, and strongly suggests Gradient-MUSIC is optimal for Gaussian stationary noise for $d\geq 2$. The improvement reflects an averaging effect unavailable in the adversarial setting, and shows that the familiar reciprocal-aperture scaling is only the deterministic part of a richer statistical phenomenon. 

 This notion of noisy super-resolution scaling concerns not distinguishability itself, but \emph{stability in accuracy under noise} and is different from the conventional meaning of super-resolution as breaking the diffraction limit. 
The classical diffraction limit is a deterministic geometric concept: it describes the smallest scale at which nearby objects remain distinguishable, and typically leads to a resolution scale proportional to the reciprocal of the aperture or sampling diameter. In that sense, super-resolution means the ability to recover features below this classical resolution scale. In other words, conventional diffraction-limited super-resolution is about surpassing a geometric
resolution threshold, whereas noisy super-resolution scaling is about surpassing the noise-limited
accuracy one would predict from the deterministic aperture law alone. The former is a statement
about what scales are identifiable; the latter is a statement about how the estimation error decays
with aperture in a stochastic environment.

Third, these results are uniform and nonasymptotic, holding whenever the minimum separation is at least a sufficiently large constant multiple of the reciprocal sampling diameter. In the two examples and for deterministic noise, Gradient-MUSIC attains the optimal minimax scaling law in multiple dimensions, extending the one-dimensional optimality theory established in \cite{fannjiang2025optimality}. Definition of minimax optimal and lower bounds are deferred to \cref{sec:minimax}. For the stochastic case, in view of  Cram\'er-Rao lower bounds for $d=1$ \cite{stoica1989music}, our results strongly suggest that Gradient-MUSIC is also statistically optimal for Gaussian noise in higher dimensions. 

The broader message is that multidimensional super-resolution can be understood geometrically through the signal subspace. Once this geometry is encoded by the MUSIC function, one obtains not only a practical reconstruction principle, but also a unified framework in which stability theory, nonconvex optimization, and resolution analysis become different aspects of the same underlying structure.

\subsection{Related work in the literature}
\label{sec:relatedwork}

We conclude this section by situating the present work within the broader literature on spectral estimation, with particular emphasis on the multidimensional setting. The most relevant lines of work for comparison are classical MUSIC and related subspace methods, optimization-based approaches, and projection-based methods. A recurring theme is that the multidimensional problem is not merely a straightforward extension of the one-dimensional theory: both the geometry of the parameter space and the structure of the sampling set introduce phenomena that do not appear in the single-variable setting.

\paragraph{MUSIC and subspace methods.}
Classical asymptotic analyses of MUSIC, such as \cite{stoica1989music}, control the discrepancy between $\vartheta$ and the perturbed configuration $\tilde\vartheta$ in terms of complicated expressions that depend implicitly on both the model parameters and the realization of the noise. By contrast, our results provide explicit error bounds in terms of the fundamental quantities of the problem, such as the sampling diameter, the separation scale, and the noise level. Just as importantly, the classical theory does not address whether $\tilde\vartheta$ can in fact be computed numerically: it does not prove that the elements of $\tilde\vartheta$ are precisely the smallest $s$ local minima of the perturbed MUSIC function, nor does it provide a constructive procedure for finding them.

Several works, including \cite{liao2016music,liao2015multi,li2021stable}, study perturbation and stability properties of the MUSIC function itself. Such results are certainly informative, but by themselves they do not imply stability of the underlying parameter set $\vartheta$. Controlling the perturbation of a function is not the same as controlling the perturbation of its relevant minima, and this is precisely why a more geometric analysis of the landscape is needed. In dimension $d=1$, these issues were addressed in our recent work \cite{fannjiang2025optimality}; the present paper extends that viewpoint to the multidimensional setting.

The multidimensional literature is considerably more limited. In particular, \cite{liao2015multi} derives pointwise perturbation bounds for the MUSIC function under the additional assumption that the projections of $\vartheta$ onto the coordinate axes are well separated. Such an assumption is not intrinsic to the problem, and is not needed in our analysis: \cref{thm:maincube,thm:mainball} require only that $\vartheta$ itself be separated in $\R^d$. More fundamentally, as noted above, stability of the MUSIC function alone does not yield stability of the recovered configuration without further geometric arguments.

\paragraph{Optimization-based methods.}
A substantial body of work approaches spectral estimation through optimization. On the convex side, several multidimensional formulations based on continuous compressed sensing or semidefinite relaxations are developed in \cite{de2016exact,poon2019multidimensional,poon2023geometry}. These methods are powerful and mathematically elegant, but they may produce false positives, in the sense of outputting a point $\varphi$ not close to any true parameter $\theta_\ell$, or may place multiple recovered points near the same $\theta_\ell$. Even setting aside these issues, the available quantitative guarantees do not appear to reach the optimal minimax scaling in the noisy multidimensional regime. For example, the theory in \cite{poon2023geometry} implies that for samples on $\sampset_\star=\{-m,\dots,m\}^d$ and sufficiently small deterministic noise in $\ell^2$, the recovery accuracy is of order
\[
C_s \frac{\sqrt{\|\eta\|_{\ell^2(\sampset_\star)}}}{m},
\]
whereas the optimal minimax rate is of order
\[
C\frac{\|\eta\|_{\ell^2(\sampset_\star)}}{m^{3/2}}
\]
in the one-dimensional benchmark case.

A more classical nonconvex route is to estimate the full parameter vector $(\vartheta,\bfa)\in\R^{sd}\times \C^s$ via nonlinear least squares or maximum likelihood estimation; see, for example, \cite{stoica2002maximum,traonmilin2020basins}. These are optimization problems in very high dimension, especially when the number of sources $s$ is large. From this viewpoint, the main obstacle is the curse of dimensionality: even if the global minimum is statistically optimal, its basin of attraction is expected to occupy only a tiny fraction of parameter space. This has been quantified in the one-dimensional setting under random Fourier measurements \cite{traonmilin2020basins}, where the basin radius is on the order of the minimum separation $\Delta$, so that the basin volume decays exponentially in $s$. One of the main conceptual differences of the present work is that MUSIC reduces the optimization from the full parameter space to a function on $\R^d$, at the price of having to understand a nonconvex landscape with multiple relevant minima.

\paragraph{Projection-based methods.}
A natural strategy in dimensions $d\ge 2$ is to reduce the problem to a family of lower-dimensional ones by projection. For example, in dimension $d=2$, if $u\in\R^2$ is a unit vector, then
\[
y(tu)=\sum_{\ell=1}^s a_\ell e^{2\pi i (\theta_\ell\cdot u)t},
\qquad |t|\le m,
\]
is a one-dimensional signal of the same general form as \eqref{eq:y}, with projected frequencies $\{\theta_\ell\cdot u\}_{\ell=1}^s$. One may then attempt to recover the projected frequencies along sufficiently many directions $u$ and reconstruct $\theta_\ell$ from these projected data.

Projection methods are often attractive from the standpoint of sample complexity and dimensional scaling, since they reduce a multidimensional problem to a collection of one-dimensional ones. Their main weakness, however, is instability. Even when the original configuration $\vartheta$ is well separated in $\R^d$, its projection onto a given direction may have arbitrarily small separation, thereby producing an ill-conditioned one-dimensional inverse problem regardless of the reconstruction method used \cite{li2021stable,batenkov2021super}. Some works, such as \cite{cuyt2020sparse,mhaskar2025robust}, avoid this issue by assuming that suitable projection directions exist and are known in advance.

Recent papers \cite{chen2021algorithmic,jin2023super} do not assume access to such good directions. A comparison between these two approaches is given in \cite{jin2023super}. The algorithm in \cite{chen2021algorithmic}, formulated for $d=2$, assumes a separation condition $\Delta\ge c\,m^{-1}$ for an explicit constant $c$, but requires time and sample complexity at least
\[
C\left(\frac{s\sigma}{\Delta}\right)^{s^2},
\]
and also assumes that $\|\eta\|_{L^\infty(B_m)}$ is sufficiently small compared with an inverse polynomial in $s$. By contrast, \cref{thm:maincube} uses $O_d(m^d)$ samples and permits a much weaker noise condition; see \eqref{eq:etacondcube}.

The randomized algorithm of \cite{jin2023super} shows that, when
\[
m\Delta \gtrsim_d \log(s/\delta),
\]
the frequencies can be estimated with accuracy
\[
C_d m^{-1}\sqrt{m^{-d}\|\eta\|_{L^2([0,m]^d)}^2+\delta |a|_1}
\]
with probability $1-o(s)$. In particular, this yields the rate
\[
C_d\,m^{-d/2-1}\|\eta\|_{L^2([0,m]^d)}
\]
under the stronger separation requirement
\[
m\Delta \gtrsim_d \log\!\Bigl(|a|_1 s m^d \|\eta\|_{L^2([0,m]^d)}^{-2}\Bigr).
\]
Gradient-MUSIC attains the same scaling; see \cref{thm:mainball} (with the corresponding stochastic-noise estimate), but under the milder condition
\[
m\Delta\gtrsim_d 1.
\]
The trade-off is that the randomized projection-based algorithm is faster.

ESPRIT is another widely used subspace method. In one dimension \cite{kailath1989esprit}, it exploits the fact that translation becomes modulation under the Fourier transform, thereby recasting spectral estimation as an eigenvalue problem. Multidimensional extensions construct a family of matrices corresponding to several shift directions spanning $\R^d$. In the noiseless case these matrices are simultaneously diagonalizable; in the noisy case, however, poor choice of shift directions can again lead to instability. Various strategies have been proposed to mitigate this issue; see \cite{abed1998least,haardt2002simultaneous,sahnoun2017multidimensional,andersson2018esprit}. Nevertheless, many fundamental questions about the stability and optimality of multidimensional ESPRIT remain open.

\paragraph{Other methods.}
Prony's method in one dimension \cite{prony1795essai,katz2024accuracy} recovers the parameters by locating the roots of an associated trigonometric polynomial. In its classical form, it uses only the minimal number of samples and therefore cannot benefit from oversampling when $m\gg s$. Two multivariate extensions may be found in \cite{kunis2016prony,sauer2017prony}, although we are not aware of corresponding sharp stability results under noise. In one dimension, root-MUSIC can also be formulated through the roots of a univariate polynomial \cite{rao2002performance}; we are not aware of a satisfactory multidimensional analogue. There are also recent approaches based on rational approximation and interpolation; see \cite{wilber2022data,derevianko2025parameter}.

In summary, the distinctive feature of the present work is not merely that it studies MUSIC in multiple dimensions, but that it gives a constructive and quantitatively sharp theory linking subspace perturbation, nonconvex landscape geometry, and computational recovery. This combination appears to be largely absent from the existing multidimensional literature.
	
\subsection{Organization of the paper}

\cref{sec:setup} introduces the MUSIC function and associated measurement kernel. \cref{sec:structural} states the main structural theorem, outlines the Gradient-MUSIC algorithm, and deals with subspace estimation. \cref{sec:example1,sec:example2} verify the abstract assumptions of the structural theorem in order to prove the main results for the discrete cube and continuous ball, respectively. \cref{sec:minimax} concerns minimax lower bounds and optimality of Gradient-MUSIC. Ingredients used to prove the structural theorem are contained in \cref{sec:geometricanalysis}. Numerical simulations which verify the main results are in \cref{sec:numerics}. Discussion about the broader implications of this paper is in \cref{sec:conclusion}. The appendices contains auxiliary lemmas and technical proofs.

\subsection{Notation}
\label{sec:notation}

We collect here the notation used throughout the paper.

\paragraph{Asymptotic notation.}
Absolute constants are denoted by $C,c,C_1,c_1,\dots$ and may change from line to line. For nonnegative quantities $x,y$ and parameters $a,b$, we write
\[
x\lesssim_{a,b} y
\]
if $x\le C_{a,b}y$ for some constant $C_{a,b}>0$ depending only on $a,b$. When the constant is absolute, we write $x\lesssim y$. We write $x\asymp y$ if both $x\lesssim y$ and $y\lesssim x$, and similarly for $x\asymp_{a,b} y$.

\paragraph{Geometry.}
If $\domain\subset \R^d$, then $|\omega|$ denotes the Euclidean norm of $\omega\in\domain$. If $\domain\subset \T^d=(\R/\Z)^d$, identified with $[0,1)^d$, we set
\[
|\omega|:=\min_{n\in\Z^d} |\omega+n|.
\]
For $p\in[1,\infty]$, $|\omega|_p$ denotes the usual $\ell^p$ norm. We write $B_r(\omega)$ for the closed Euclidean ball of radius $r$ centered at $\omega$, and $Q_r(\omega)$ for the closed cube of side length $2r$ centered at $\omega$. We abbreviate
\[
B_r:=B_r(0),
\qquad
Q_r:=Q_r(0).
\]
For sets $A,B\subset\domain$, we write $\overline A$ for the closure, $A^\circ$ for the interior,
\[
-A:=\{-a:a\in A\},
\qquad
A+B:=\{a+b:a\in A,\ b\in B\},
\]
and
\[
\dist(A,B):=\inf\{|a-b|:a\in A,\ b\in B\}.
\]

\paragraph{Function spaces and operators.}
If $(X,\nu)$ is a measure space, then $\nu(X)$ denotes the total measure of $X$. For $p\in[1,\infty)$, $L^p_\nu(X)$ denotes the usual space of $\nu$-measurable complex-valued functions with finite $L^p$ norm, and $L^\infty_\nu(X)$ is defined in the usual essential-supremum sense. When $\nu$ is Lebesgue measure, we omit the subscript. When $X$ is discrete and $\nu$ is counting measure, we write $\ell^p(X)$ and $\|\cdot\|_{\ell^p}$.

For a vector-valued function $f=(f_1,\dots,f_n)$, we set
\[
|f(x)|:=\Big(\sum_{j=1}^n |f_j(x)|^2\Big)^{1/2},
\qquad
\|f\|_{L^p}:=\||f|\|_{L^p}.
\]
For a matrix $A$, $\|A\|$ denotes the spectral norm, $\|A\|_F$ the Frobenius norm, $\lambda_k(A)$ the $k$-th largest eigenvalue, $\sigma_k(A)$ the $k$-th largest singular value, and $\trace(A)$ the trace. We also write $\lambda_{\min}(A)$ and $\lambda_{\max}(A)$ and similarly for singular values. For a bounded operator $T$, $\|T\|$ denotes the operator norm. The identity matrix or operator is denoted by $I$.

If ${\mathcal H}$ is a Hilbert space, then $\inner{\cdot,\cdot}=\inner{\cdot,\cdot}_{\mathcal H}$ denotes its inner product, linear in the second argument. For $u,v\in\R^d$, we write $u\cdot v$ for the Euclidean inner product. For a bounded operator $T\colon {\mathcal H}\to{\mathcal H}$, $T^*$ denotes the adjoint.

\paragraph{Differentiation and Fourier transform.}
We write $\partial_j f$ for the partial derivative in the $j$-th coordinate, $\nabla f$ for the gradient, $\nabla^2 f$ for the Hessian, and $\Delta f$ for the Laplacian. For a multi-index $\alpha=(\alpha_1,\dots,\alpha_d)\in\N^d$,
\[
|\alpha|:=\alpha_1+\cdots+\alpha_d,
\qquad
\partial^\alpha:=\partial_1^{\alpha_1}\cdots \partial_d^{\alpha_d}.
\]
For $f\in L^1(\R^d)$, its Fourier transform is
\[
\hat f(\xi):=\int_{\R^d} f(x)e^{-2\pi i \xi\cdot x}\,dx.
\]
If $f\in L^1(Q_{r/2})$, we define its periodic Fourier coefficients by
\[
\hat f(k):=\int_{Q_{r/2}} f(x)e^{-2\pi i k\cdot x/r}\,dx,
\qquad k\in\Z^d.
\]

\paragraph{Probability.}
All random variables are complex-valued unless otherwise stated. We write
\[
z\sim \mathcal N(0,1)
\]
to mean that $\Re z$ and $\Im z$ are independent Gaussian random variables with mean zero and variance $1/2$. Following \cite{adler2007random}, a random field on $\R^d$ is a measurable map from a fixed probability space into the space of complex-valued functions on $\R^d$. It is Gaussian if every finite collection of its values is jointly Gaussian.

	\section{General framework and set-up}
	\label{sec:setup}
	
{Recall the definition of spectral estimation problem. 	}
	\begin{definition}[Spectral estimation problem]
		\label{def:spectral}
		Let $\Omega$ be either $\T^d$ or a bounded subset of $\R^d$. Define $y\colon \R^d\to\C$ by
		\begin{equation*}
			y(x)=\sum_{\ell=1}^s a_\ell e^{2\pi i \theta_\ell \cdot x}, \quad \text{for some}\quad  \vartheta:=\{\theta_\ell\}_{\ell=1}^s\subset \domain \andspace \bfa:=\{a_\ell\}_{\ell=1}^s\subset \C\setminus\{0\}. 
		\end{equation*}
		Let $\eta\colon \R^d\to\C$, $\tilde y= y+\eta$, and $X_\star\subset\R^d$. The (multidimensional) spectral estimation problem is to estimate the unknown parameters -- number of sources $s$, locations $\vartheta$, and amplitudes $\bfa$ -- given the measurement data 
		\begin{equation*}
			\tilde y|_{\sampset_\star}.
		\end{equation*}
		Since the stability of this problem only depends on $y$ versus $\eta$, by rescaling, we assume without loss of generality that $\min_\ell |a_\ell|= 1$.
	\end{definition}

We aim at a general  theory that applies to general sampling sets $\sampset_\star$, provided there exist sets $X$ and $X'$ such that
\begin{equation}
X+X'\subset \sampset_\star.\label{eq:sumset}
\end{equation}
and  we choose two positive Borel measures $\nu$ and $\nu'$ on $\R^d$, each with finite total variation, supported on sets $X$ and $X'$, respectively. 

\begin{table}[ht]
	\centering
	\begin{tabular}{|c|c|c|c|c|c|}
		\hline
		& $\sampset_\star$ & $X$ & $\nu$ & $X'$ & $\nu'$ \\
		\hline
		\cref{sec:exam1} & $Q_{2m}\cap \Z^d$ & $Q_m\cap \Z^d$ & counting measure on $X$ & $X$ & $\nu$ \\
		\hline
		\cref{sec:exam2} & $B_{2m}$ & $B_m$ & Lebesgue measure on $X$ & $X$ & $\nu$ \\
		\hline
	\end{tabular}
	\caption{Representative choices of $\sampset_\star$, $X$, and the measures $\nu$ and $\nu'$.}
	\label{table:Omega}
\end{table}

At this stage, the reader may largely forget about $\sampset'$ and the auxiliary measure $\nu'$. Their role is primarily in the construction of the subspace estimator $\tilde U$, see \cref{sec:subspace} for detailed discussion. The measure $\nu$ specifies the Hilbert-space geometry on $L^2_\nu(X)$ and can therefore be interpreted as a way of preprocessing or weighting the measurements, somewhat analogous to importance sampling in statistics. Different choices of $\nu$ lead to different metrics, and hence to different geometric properties of the resulting MUSIC function. In the present paper, however, we take $\nu$ as fixed and do not pursue the important question of how to optimize this choice; see \cref{sec:scale} for further discussion.

This level of generality is important in multidimensional problems, where many sampling geometries of practical interest fall outside the standard cube and ball settings.

For example, the annulus
\[
\sampset_\star=A_{r_1,r_2}:=\{x\in\R^d:r_1\le |x|\le r_2\},
\qquad 0<r_1<r_2,
\]
admits such a decomposition: one may take
\[
X=B_r,
\qquad
X'=A_{r_1+r,r_2-r},
\]
for any
\[
0<r<\frac{r_2-r_1}{2}.
\]
Then indeed $X+X'\subset \sampset_\star$. This illustrates how the abstract framework naturally accommodates sampling geometries beyond the basic model cases treated in this paper.

For each $\omega\in\domain$, define
\begin{equation}
\label{eq:phiomega}
\phi_\omega(x):=\frac{1}{\sqrt{\nu(X)}}\,e^{2\pi i \omega\cdot x},
\qquad x\in X,
\end{equation}
where $\nu(X)$ denote the total measure of $X$.
By construction,
\[
\|\phi_\omega\|_{L^2_\nu(X)}=1
\qquad \text{for all } \omega\in\domain.
\]

\begin{definition}[MUSIC function]
\label{def:musicfunction}
Let $W$ be a closed subspace of $L^2_\nu(X)$ and let
\[
P_W\colon L^2_\nu(X)\to L^2_\nu(X)
\]
be the orthogonal projection onto $W$.

The \emph{MUSIC function} associated with $W$ is the map
\[
q_W\colon \domain\to [0,1]
\]
defined by
\[
q_W(\omega)
:=
\inner{\phi_\omega,(I-P_W)\phi_\omega}_{L^2_\nu}
=
\int_X \overline{\phi_\omega}(I-P_W)\phi_\omega\,d\nu.
\]
\end{definition}

There are several equivalent expressions for $q_W$. Most notably,
\begin{equation}
\label{eq:qformulas}
q_W(\omega)
=
\|(I-P_W)\phi_\omega\|_{L^2_\nu}^2
=
1-\|P_W\phi_\omega\|_{L^2_\nu}^2.
\end{equation}
Thus $q_W(\omega)$ is small precisely when $\phi_\omega$ lies close to the subspace $W$. In particular, since $\phi_{\theta_\ell}\in U$ for each $\ell$, one has
\[
q_U(\theta_\ell)=0,
\qquad \ell=1,\dots,s.
\]
Hence, in the noiseless setting, the unknown configuration $\vartheta$ is encoded in the zero set of $q_U$.

\subsection{The measurement kernel}
\label{sec:assumptions}

	\begin{figure}[ht]
		\centering
		\includegraphics[width=0.45\textwidth]{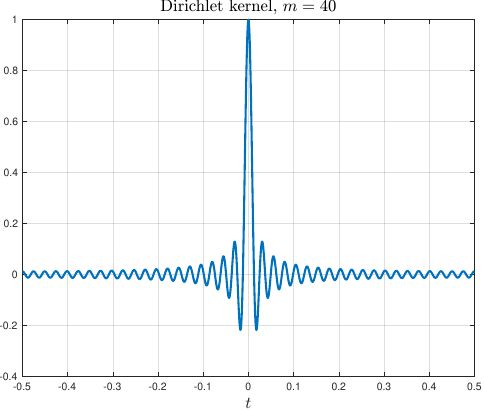}
		\includegraphics[width=0.45\textwidth]{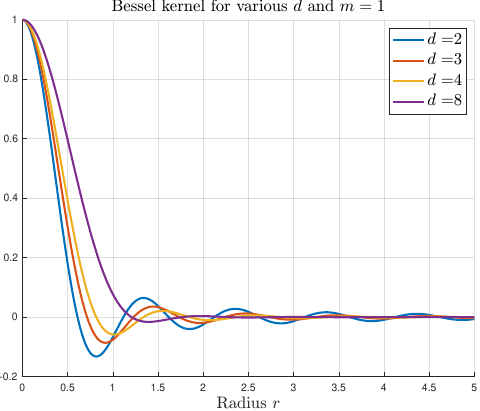}
		\caption{The kernel in Example 1 is a $d$-fold tensor product of a normalized Dirichlet kernel, shown on the left. The kernel in Example 2 is a radial function whose radial profile is a weighted Bessel function whose order depends on the dimension $d$, shown on the right.}
		\label{fig:kernels1}
	\end{figure}

The key observation behind the general theory is that the geometry of the MUSIC function can be described, quantitatively and with considerable precision, in terms of an associated measurement kernel $K$.  As a result, the theory reduces to verifying a collection of conditions on $K$ adapted to the sampling geometry under consideration. 

Since for $\omega,\omega'\in\domain$,
\[
\inner{\phi_{\omega'},\phi_\omega}_{L^2_\nu}
=
\frac{1}{\nu(X)}
\int_X e^{2\pi i (\omega-\omega')\cdot x}\,d\nu(x)
\]
Thus the interaction between two steering vectors depends only on the difference $\omega-\omega'$. This naturally leads to the following definition.

\begin{definition}[Measurement kernel]
\label{def:kernel}
The \emph{measurement kernel} associated with $(X,\nu)$ is the function
\[
K\colon \domain-\domain\to \C
\]
defined by
\begin{equation}
\label{eq:kerneldef}
K(\xi)
:=
\frac{1}{\nu(X)}
\int_X e^{2\pi i \xi\cdot x}\,d\nu(x).
\end{equation}
In particular,
\[
K(0)=1.
\]
\end{definition}

Slightly abusively, we shall regard $K$ as a function on the difference set $\domain-\domain$. See Figure~\ref{fig:kernels1} for plots of the kernels arising in the two model cases considered in \cref{sec:exam1,sec:exam2}.

The relevance of $K$ to the MUSIC landscape may be seen already from a simple calculation. 
Fix $\ell\in\{1,\dots,s\}$. Since $\phi_{\theta_\ell}\in U$, one has
\[
P_U\phi_{\theta_\ell}=\phi_{\theta_\ell}.
\]
Write
\[
U=\spann(\phi_{\theta_\ell})\oplus W,
\]
where $W\subset U$ is the $(s-1)$-dimensional subspace orthogonal to $\phi_{\theta_\ell}$. 
Then
\[
P_U=\phi_{\theta_\ell}\phi_{\theta_\ell}^*+P_W,
\]
and therefore
\begin{equation}
\label{eq:qcalculation}
\begin{aligned}
q_U(\omega)
&=1-\inner{\phi_\omega,P_U\phi_\omega}_{L^2_\nu} \\
&=1-\inner{\phi_\omega,\phi_{\theta_\ell}\phi_{\theta_\ell}^*\phi_\omega}_{L^2_\nu}
    -\inner{\phi_\omega,P_W\phi_\omega}_{L^2_\nu} \\
&=1-|K(\omega-\theta_\ell)|^2-\inner{\phi_\omega,P_W\phi_\omega}_{L^2_\nu}.
\end{aligned}
\end{equation}
Equivalently, using \eqref{eq:qformulas},
\[
q_U(\omega)
=
1-\|P_U\phi_\omega\|_{L^2_\nu}^2,
\]
and the decomposition above separates the contribution of the distinguished mode $\phi_{\theta_\ell}$ from the contribution of the remaining modes.

Formula \eqref{eq:qcalculation} already indicates the two scales of the problem. The term
\[
1-|K(\omega-\theta_\ell)|^2
\]
is a \emph{local} contribution: it depends only on the displacement of $\omega$ from the nearby parameter $\theta_\ell$. By contrast,
\[
\inner{\phi_\omega,P_W\phi_\omega}_{L^2_\nu}
\]
is a \emph{global} interaction term, reflecting the influence of all the other components of the configuration. Under suitable assumptions, when $\omega$ lies sufficiently close to $\theta_\ell$, this global term is small and the local behavior of $q_U$ is controlled by the kernel $K$. On the other hand, when $\omega$ lies far from the entire configuration $\vartheta$, a different argument is needed to show that $q_U(\omega)$ remains uniformly large. The full geometric analysis of both $q_U$ and $q_{\tilde U}$ is developed in \cref{sec:geometricanalysis}.

The first standing assumption on the kernel is a natural symmetry and moment condition.

\begin{assumption}[Symmetry and regularity]
\label{assump:symmetry}
Assume that $X=-X$, that
\[
\nu(A)=\nu(-A)
\qquad\text{for every $\nu$-measurable set } A\subset X,
\]
and that for every multi-index $\alpha\in\N^d$ with $|\alpha|\le 4$,
\begin{equation}
\label{eq:moment}
\int_X |x^\alpha|\,d\nu(x)<\infty.
\end{equation}
\end{assumption}

These assumptions are very mild. If $X$ is bounded, then the moment condition \eqref{eq:moment} is automatic. The symmetry condition is satisfied in many examples of interest, including:
\begin{enumerate}[(a)]
\item \emph{Structured discrete measures:} $\nu$ is counting measure on $\Z^d\cap X$, where $X$ is a cube, ball, annulus, or another centrally symmetric set.
\item \emph{Continuous measures:} $\nu$ is Lebesgue measure restricted to a cube, ball, annulus, or another centrally symmetric domain.
\item \emph{Random discrete measures:} $\nu=\sum_{x\in X}\delta_x$, where $X$ is a random centrally symmetric point cloud; for example, $X$ may consist of i.i.d.\ samples together with their reflections.
\item \emph{Singular measures:} $\nu=\sum_j \sigma_{S_j}$, where each $\sigma_{S_j}$ is surface measure on a set $S_j$ and the total measure remains centrally symmetric. Natural examples include collections of hyperplanes through the origin, as in tomography.
\end{enumerate}
Further examples will be discussed in \cref{sec:scale}.

Under \cref{assump:symmetry}, the kernel enjoys strong elementary properties. By the symmetry of $\nu$, we have
\[
K(\xi)=K(-\xi)=\overline{K(\xi)},
\]
so $K$ is real-valued and even. By dominated convergence, all partial derivatives up to order four exist, and if $X$ is compact then $K$ is in fact smooth of all orders. Moreover, all odd derivatives of $K$ vanish at the origin. In particular, for $j,k,\ell\in\{1,\dots,d\}$,
\begin{equation}
\label{eq:Kzeroformulas}
\partial_j K(0)=0,
\qquad
\partial_j\partial_k\partial_\ell K(0)=0.
\end{equation}
Thus the local shape of $K$ near the origin is determined first by its Hessian. For convenience we introduce the notation
\begin{equation}
\label{eq:hessiandef}
\Psi:=-\nabla^2 K(0).
\end{equation}
By Bochner's theorem, $\Psi$ is positive semidefinite. Since $\nabla K(0)=0$, it is already evident from \eqref{eq:qcalculation} that $\Psi$ governs the leading local geometry of the MUSIC function near each true parameter.

\begin{remark}
Although we do not assume it in general, in many examples of interest the matrix $\Psi$ is diagonal, and often a scalar multiple of the identity. This occurs, for instance, whenever $\nu$ enjoys the stronger coordinatewise reflection symmetry
\[
\nu(A)=\nu(R_kA),
\qquad k=1,\dots,d,
\]
where $R_k$ denotes reflection across the $k$-th coordinate hyperplane.
\end{remark}

\subsection{Example kernels in tomography}
\label{sec:scale}


As noted above, a natural measure in tomography is a sum of surface measures on hyperplanes passing through the origin. In a suitable continuum limit, with uniformly distributed orientations, the resulting effective sampling density behaves like $|x|^{-1}$ near the origin relative to Lebesgue measure. This motivates considering more general isotropic measures with radial densities comparable to $|x|^\alpha$ on $B_m$, where
\[
\alpha>-d.
\]

Let $\mu$ be such a measure. Then, up to normalization, the corresponding kernel has the form
\begin{align*}
K(\xi)
&\sim
\int_0^m \int_0^{2\pi}
r^{1+\alpha}\cos\bigl(|\xi|r\cos\theta\bigr)\,d\theta\,dr,
\qquad d=2,
\\
K(\xi)
&\sim
\int_0^m \int_0^\pi
r^{2+\alpha}\cos\bigl(|\xi|r\cos\theta\bigr)\sin\theta\,d\theta\,dr
=
\frac{2}{|\xi|}
\int_0^m \sin(|\xi|r)\,r^{1+\alpha}\,dr,
\qquad d=3.
\end{align*}
Straightforward Taylor expansion near $\xi=0$ yields
\begin{align*}
K(\xi)
&=
1-\frac{2+\alpha}{4(4+\alpha)}|\xi|^2m^2
+\frac{2+\alpha}{64(6+\alpha)}|\xi|^4m^4+\cdots,
\qquad d=2,
\\
K(\xi)
&=
1-\frac{3+\alpha}{6(5+\alpha)}|\xi|^2m^2
+\frac{3+\alpha}{120(7+\alpha)}|\xi|^4m^4+\cdots,
\qquad d=3.
\end{align*}
Consequently,
\begin{equation}
\label{eq:isotropichessian}
\Psi=C_{\alpha,d}\,m^2 I.
\end{equation}
Thus, at the local level relevant to the behavior near each true parameter, the kernel associated with any isotropic density of the form $|x|^\alpha$ behaves similarly to the kernel for Lebesgue measure on a ball. In particular, the second row of \cref{table:kernel2} remains the natural model for the local geometry.

A practical obstacle, however, is that explicit global estimates for $K_\vartheta$, $E_0(\vartheta)$, and $E_1(\vartheta)$ are generally not available for such weighted measures. One way around this is to exploit the freedom in the choice of $\nu$. Namely, if the original sampling density is comparable to $|x|^\alpha$, one may choose $\nu$ with compensating density comparable to $|x|^{-\alpha}$ so that the processed data correspond instead to the uniform measure on the ball. This transforms the kernel into the one studied in \cref{sec:exam2}, at the cost of altering the noise model. In particular, stationary noise in the original coordinates becomes nonstationary after this reweighting, typically with radial growth or decay. This provides another natural motivation for allowing nonstationary noise in the framework developed above.

\subsection{Global parameters of the measurement kernel}
Local control alone, however, is not enough. Since our goal is a globally convergent optimization theory, we must also quantify the nonlocal behavior of $K$. One important object is the \emph{kernel matrix}
\begin{equation}
\label{eq:kernelmatrix}
K_\vartheta
:=
\bigl[K(\theta_j-\theta_k)\bigr]_{j,k=1}^s,
\end{equation}
which is positive semidefinite by Bochner's theorem. Its extreme eigenvalues,
\[
\lambda_{\min}(K_\vartheta)
\qquad\text{and}\qquad
\lambda_{\max}(K_\vartheta),
\]
measure how coherent or incoherent the clean measurements are. Roughly speaking, if $K_\vartheta$ is well conditioned, then the individual components of the signal remain distinguishable at the subspace level.

A second family of nonlocal quantities consists of the energy terms
\begin{align}
E_0(\vartheta)
&:=
\sup_{\omega\in\Omega}
\sum_{\ell:\,|\omega-\theta_\ell|\ge \Delta/2}
|K(\omega-\theta_\ell)|^2,
\label{eq:E0def}
\\
E_1(\vartheta)
&:=
\sup_{\omega\in\Omega}
\sum_{\ell:\,|\omega-\theta_\ell|\ge \Delta/2}
|\nabla K(\omega-\theta_\ell)|^2,
\label{eq:E1def}
\end{align}
where
\[
\Delta:=\Delta(\vartheta)=\min_{j\neq k}|\theta_j-\theta_k|
\]
is the minimum separation. Intuitively, these quantities measure the cumulative influence of all sources that are not locally closest to $\omega$. If $K$ and $\nabla K$ decay sufficiently fast away from the origin, then $E_0(\vartheta)$ and $E_1(\vartheta)$ become small once the configuration is well separated.

These quantities arise naturally in the analysis of \eqref{eq:qcalculation}. The matrix $K_\vartheta$ is used to approximate the projection $P_W$ and quantify how the different modes interact, while $E_0(\vartheta)$ controls the nonlocal term in \eqref{eq:qcalculation} itself. An analogous argument applied to the Hessian of $q_U$ produces the corresponding role of $E_1(\vartheta)$. The details are given in \cref{sec:geometricanalysis}.

For the two canonical examples of this paper, the resulting quantities admit explicit bounds.

\begin{table}[ht]
\centering
\begin{tabular}{|c|c|c|c|c|c|}
\hline
 & $\Psi$ & $\lambda_{\min}(K_\vartheta)$ & $\lambda_{\max}(K_\vartheta)$ & $E_0(\vartheta)$ & $E_1(\vartheta)$ \\
\hline
\cref{sec:exam1}
& $\frac{4\pi^2}{3}m(m+1)I$
& $\ge 1-C_d/\beta$
& $\le 1+C_d/\beta$
& $\lesssim_d \beta^{-2}$
& $\lesssim_d \beta^{-2}m^2$
\\
\hline
\cref{sec:exam2}
& $\frac{4\pi^2}{d+2}m^2 I$
& $\ge 1-C_d/\beta$
& $\le 1+C_d/\beta$
& $\lesssim_d \beta^{-d-1}$
& $\lesssim_d \beta^{-d-1}m^2$
\\
\hline
\end{tabular}
\caption{Quantitative behavior of the kernel quantities in the two model cases, under the separation condition $\Delta(\vartheta)\ge \beta/m$ for sufficiently large $\beta=\beta(d)$.}
\label{table:kernel2}
\end{table}

  	\section{Main structural theorem and algorithm}
\label{sec:structural}

In general, greedy local methods such as gradient descent need not converge to the minima one wishes to recover. They may become trapped at stationary points, or converge to undesirable local minima, unless the objective function enjoys additional structure or the initialization is already sufficiently accurate. In many inverse problems, especially in high dimension, the real difficulty is therefore not the local convergence mechanism itself, but rather the construction of reliable initialization.

{Since the MUSIC function is generally nonconvex and one seeks multiple minima rather than a single global minimizer, the practical success of the method depends on whether one can locate suitable initial points efficiently. Gradient-MUSIC addresses this by thresholding the MUSIC function on a coarse grid and then using local descent only after the relevant regions of the landscape have been identified. The theoretical question is therefore not merely whether gradient descent converges locally, but whether the landscape admits a global organization that makes this two-stage strategy reliable. One of the central results of the paper is that, under natural assumptions, it does.
}

The following definition isolates a class of landscapes for which this difficulty can be overcome. Informally, an admissible landscape is one in which each true parameter is accompanied by a nearby local minimum with a basin of attraction of controlled size, while all points sufficiently far from the true configuration lie above a definite threshold. Such a structure makes it possible to combine coarse global search with local descent in a provably effective way.

For a given local optimization scheme, we say that a set $S\subset \Omega$ is a \emph{basin of attraction} of $\omega\in\Omega$ if every trajectory of the scheme initialized at a point $\omega_0\in S$ converges to $\omega$.

\begin{definition}[Admissible optimization landscape]
\label{def:desirable}
Let $f\colon \domain\to [0,\infty)$, let $\{\theta_\ell\}_{\ell=1}^s\subset \domain$, and let $\epsilon,\rho,\alpha_0,\alpha_1$ satisfy
\[
0<\epsilon<\rho,
\qquad
0\le \alpha_0<\alpha_1.
\]
We say that $f$ is an \emph{admissible optimization landscape} for the given configuration $\{\theta_\ell\}_{\ell=1}^s$, with parameters $(\epsilon,\rho,\alpha_0,\alpha_1)$, if the following hold:
\begin{enumerate}
\item
For each $\ell\in\{1,\dots,s\}$, the function $f$ has a local minimum $\tilde\theta_\ell\in B_\epsilon(\theta_\ell)$, and the ball $B_\rho(\theta_\ell)$ is contained in the basin of attraction of $\tilde\theta_\ell$.

\item
The relevant minima are uniformly shallow in value and all points outside the attraction regions are uniformly higher:
\[
\max_{\ell=1,\dots,s} f(\tilde\theta_\ell)\le \alpha_0,
\qquad
f(\omega)>\alpha_1
\quad
\text{for all }
\omega\notin \bigcup_{\ell=1}^s B_\rho(\theta_\ell).
\]
\end{enumerate}
\end{definition}

Let us briefly explain why this is the right notion. The first requirement says that each $\tilde\theta_\ell$ is an $\epsilon$-accurate approximation of $\theta_\ell$, and that it can be found by a local method initialized anywhere in the larger ball $B_\rho(\theta_\ell)$. Thus $\rho$ measures the effective width of the attraction region and may be much larger than the target accuracy $\epsilon$. The second requirement separates the desirable wells from the rest of the landscape by a definite gap $\alpha_1-\alpha_0$. This gap is precisely what makes coarse thresholding possible.

The next lemma formalizes this observation.

\begin{lemma}[Thresholding yields valid initialization]
\label{lem:threshold}
Suppose that $f\colon \Omega\to [0,\infty)$ is an admissible optimization landscape for $\{\theta_\ell\}_{\ell=1}^s$ with parameters $(\epsilon,\rho,\alpha_0,\alpha_1)$. Let $G\subset \domain$ be a finite set satisfying
\[
\mesh(G)
\le
\min\left\{
\rho,\,
\frac{\alpha_1-\alpha_0}{\|\nabla f\|_{L^\infty(\domain)}}
\right\}.
\]
Then for every $\ell\in\{1,\dots,s\}$,
\[
B_\rho(\theta_\ell)\cap \{\omega\in G:f(\omega)\le \alpha_1\}\neq \emptyset.
\]
\end{lemma}

\begin{proof}
Fix $\ell\in\{1,\dots,s\}$. By the definition of mesh norm, there exists $\omega_\ell\in G$ such that
\[
|\omega_\ell-\theta_\ell|\le \mesh(G)\le \rho.
\]
Hence $\omega_\ell\in B_\rho(\theta_\ell)$. By the mean value theorem,
\[
|f(\omega_\ell)-f(\theta_\ell)|
\le
\|\nabla f\|_{L^\infty(\domain)}\,|\omega_\ell-\theta_\ell|
\le
\|\nabla f\|_{L^\infty(\domain)}\,\mesh(G).
\]
Since $f$ is admissible, $f(\theta_\ell)\le \alpha_0$, and therefore
\[
f(\omega_\ell)
\le
\alpha_0+\|\nabla f\|_{L^\infty(\domain)}\,\mesh(G)
\le
\alpha_1.
\]
Thus $\omega_\ell\in B_\rho(\theta_\ell)\cap\{\omega\in G:f(\omega)\le \alpha_1\}$, as claimed.
\end{proof}

The significance of \cref{lem:threshold} is that the mesh norm of the initialization set $G$ need not depend on the final accuracy parameter $\epsilon$. This is crucial: in the small-noise regime one expects $\epsilon$ to be very small, but the lemma shows that one may still initialize from a much coarser set, provided the basin width $\rho$ and threshold gap $\alpha_1-\alpha_0$ are sufficiently large. Consequently, the thresholded set
\[
\{\omega\in G:f(\omega)\le \alpha_1\}
\]
contains, near each true parameter, at least one point belonging to the appropriate basin of attraction. In practice these points appear in well-separated clusters, and one may simply choose one representative per cluster as an initialization for local descent.

Of course, for our purposes it is not enough merely to know that an admissible landscape exists. The actual values of $(\epsilon,\rho,\alpha_0,\alpha_1)$ matter. One wants $\epsilon$ to be small, since it controls the final reconstruction error, and ideally it should decrease as the sampling aperture grows or the noise decreases. At the same time, one wants both $\rho$ and the separation gap $\alpha_1-\alpha_0$ to be as large as possible, since these determine how coarse the initialization grid may be and hence how efficient the thresholding step becomes.

The following theorem is the main abstract result of the paper. It shows that, under explicit conditions on the measurement kernel and the subspace perturbation, the perturbed MUSIC function is an admissible landscape for gradient descent. The concrete theorems for the cube and ball settings, \cref{thm:maincube,thm:mainball}, are obtained by verifying these abstract conditions in the corresponding model geometries.

\begin{theorem}[The perturbed MUSIC function is an admissible landscape]
\label{thm:maingeometric}
Suppose \cref{assump:symmetry} holds. Let $\delta_1,\delta_2>0$ satisfy
\[
\delta_1+\delta_2<2.
\]
Let $\vartheta\subset \domain$ be finite, let
\[
0<\tau\le {\Delta(\vartheta)}/{2},
\]
and let $\tilde U$ be a subspace of $L^2_\nu(X)$. Assume that
	\begin{align}
    		\sqrt{2 \tau \|\nabla K\|_{L^\infty(B_\tau)}} + \frac{\|\nabla K\|_{L^\infty(B_\tau)}^2}{\lambda_{\min}(K_\vartheta) \sqrt{\Delta^2 K(0)} } + \frac{E_1(\vartheta)}{\lambda_{\min}(K_\vartheta) \sqrt{\Delta^2 K(0)}}
    		\leq \delta_1 \frac{\lambda_d(\Psi)}{\sqrt{\Delta^2 K(0)}}, \label{eq:betacondition1} \\ 
    		E_0(\vartheta)+\lambda_{\max}(K_\vartheta) \max\left\{ \Big| \frac 1 {\lambda_{\max}(K_\vartheta)}-1\Big|,\, \Big|\frac 1 {\lambda_{\min}(K_\vartheta)} - 1\Big| \right\} 
    		\leq \frac 38 \left( 1 - \|K\|_{L^\infty(B_\tau^c)}^2 \right),
    		\label{eq:betacondition2} \\
    		\norm{P_U-P_{\tilde U}}
    		\leq \delta_2 \min \left\{ \frac 1 4 \left( 1- \|K\|_{L^\infty(B_\tau^c)}^2 \right), \,   \frac{\lambda_d(\Psi)}{2 \sqrt{ \Delta^2 K(0) + |\trace(\Psi)|^2}}, \, \frac{\tau}{2}\frac{\lambda_d(\Psi)} {\sqrt{\trace(\Psi)}} \right\},	\label{eq:noisecondition}
    	\end{align} 
where $E_0(\vartheta)$ and $E_1(\vartheta)$ are defined in \eqref{eq:E0def} and \eqref{eq:E1def}, respectively.

Then the perturbed MUSIC function $
\tilde q:=q_{\tilde U}$
is an admissible optimization landscape for gradient descent and the target configuration $\vartheta$, with parameters
\begin{equation}
\label{eq:mainparameters}
\left(
\frac{2\sqrt{\trace(\Psi)}}{\delta_2\lambda_d(\Psi)}
\|P_U-P_{\tilde U}\|,
\,
\tau,
\,
\|P_U-P_{\tilde U}\|^2,
\,
\left(\frac58-\frac{\delta_2}{4}\right)
\Bigl(1-\|K\|_{L^\infty(B_\tau^c)}^2\Bigr)
\right).
\end{equation}
In addition, for every $\omega\in\Omega$ satisfying
\[
\dist(\omega,\vartheta)\le \tau,
\]
one has
\begin{align}
\|\nabla \tilde q(\omega)\|
&\le
2\sqrt{\trace(\Psi)},
\label{eq:qtilde1}
\\
(2-\delta_1-\delta_2)\lambda_d(\Psi)
\le
\lambda_d(\nabla^2 \tilde q(\omega))
&\le
\lambda_1(\nabla^2 \tilde q(\omega))
\le
(2+\delta_1+\delta_2)\lambda_1(\Psi).
\label{eq:qtilde2}
\end{align}
\end{theorem}

The assumptions \eqref{eq:betacondition1}--\eqref{eq:noisecondition} are all dimensionless and have distinct roles.

\begin{enumerate}[1.]
\item
The first condition, \eqref{eq:betacondition1}, controls the local Hessian of $\tilde q$. It may be viewed as a quantitative approximation requirement in which the local behavior of the kernel, represented by $\|\nabla K\|_{L^\infty(B_\tau)}$, is balanced against the global interaction terms $E_1(\vartheta)$ and $\lambda_{\min}(K_\vartheta)$.

\item
The second condition, \eqref{eq:betacondition2}, controls the size of $\tilde q(\omega)$ away from the target configuration. It involves only genuinely global quantities, namely $E_0(\vartheta)$, the conditioning of $K_\vartheta$, and the tail bound $\|K\|_{L^\infty(B_\tau^c)}$.

\item
The third condition, \eqref{eq:noisecondition}, is the subspace perturbation requirement. It states that $\tilde U$ must approximate $U$ sufficiently well in projection norm.
\end{enumerate}

We now turn to the conclusions. The first conclusion, \eqref{eq:mainparameters}, gives explicit values of the landscape parameters. In particular, the first entry yields an upper bound for the location error between the true configuration and the relevant minima of $\tilde q$. A noteworthy feature is that this error can be substantially smaller than the raw subspace perturbation $\|P_U-P_{\tilde U}\|$, thanks to the factor
\[
\frac{\sqrt{\trace(\Psi)}}{\lambda_d(\Psi)}.
\]
For the model kernels in \cref{table:kernel2}, this improvement is already visible at the level of scaling. The remaining entries identify the basin radius $\tau$, the depth of the relevant minima, and the threshold level separating the attraction regions from the rest of the landscape.

The derivative bounds \eqref{eq:qtilde1} and \eqref{eq:qtilde2} provide the quantitative control needed for gradient descent. Most importantly, \eqref{eq:qtilde2} shows that $\tilde q$ is uniformly and strictly convex in each ball $B_\tau(\theta)$, $\theta\in\vartheta$. This strict convexity is the local mechanism underlying linear convergence of gradient descent once initialization enters the appropriate attraction region.

It is worth emphasizing that \cref{thm:maingeometric} is {\em not} uniform in $\vartheta$ yet: it is formulated for a fixed configuration, and its assumptions are expressed in terms of the associated kernel quantities such as $K_\vartheta$, $E_0(\vartheta)$, and $E_1(\vartheta)$. 

In summary, \cref{thm:maingeometric} is the structural bridge between subspace perturbation and computational recovery. It shows that once the measurement kernel has the appropriate local and global properties, and once the estimated subspace is sufficiently accurate, the perturbed MUSIC function automatically acquires the geometry needed for a practical global optimization method: coarse thresholding supplies reliable initialization, and local descent completes the recovery.

\subsection{The algorithm}
\label{sec:examples}

The computational philosophy of Gradient-MUSIC is simple. Rather than attempting to solve the full spectral estimation problem directly in the high-dimensional parameter space $(\vartheta,\bfa)$, one first extracts the signal subspace from the data and then recovers the parameters by optimizing the associated MUSIC landscape in the physical parameter domain $\domain\subset\R^d$. This separates the problem into two conceptually distinct tasks: subspace estimation and landscape navigation. The first task compresses the data into the geometrically relevant information; the second exploits the fact that, under suitable conditions, the smallest minima of the MUSIC function remain both identifiable and computationally accessible. The role of the present section is to describe this algorithmic framework and to indicate, through canonical examples, how the general theory applies to concrete sampling geometries.

\begin{algorithm}[ht]
\caption{Gradient-MUSIC -- coarse thresholding followed by local descent}
\label{alg:gradientMUSIC}
\begin{algorithmic}
\Require Projection operator $P_{\tilde U}$ onto an $s$-dimensional subspace $\tilde U$.
\State \textbf{Parameters:} a sufficiently dense grid $G\subset \domain$, a threshold level $\alpha_1$, a step size $h$, and a number of iterations $n$.
\begin{enumerate}
\item \textbf{Initialization by thresholding.}
Evaluate the MUSIC function $\tilde q$ associated with $\tilde U$ on $G$ and define
\[
G_1=\{\omega\in G:\tilde q(\omega)\le \alpha_1\}.
\]
Identify the connected clusters of $G_1$ and select one representative from each cluster, denoted by $\theta_{1,0},\dots,\theta_{s,0}$.

\item \textbf{Local refinement by gradient descent.}
For each $\ell=1,\dots,s$, perform $n$ iterations of gradient descent initialized at $\theta_{\ell,0}$:
\[
\theta_{\ell,k+1}
=
\theta_{\ell,k}-h\,\nabla \tilde q(\theta_{\ell,k}),
\qquad k=0,\dots,n-1.
\]
\end{enumerate}
\Ensure Estimated frequencies $\theta^\sharp_\ell=\theta_{\ell,n}$, $\ell=1,\dots,s$.
\end{algorithmic}
\end{algorithm}
The input is a projection operator $P_{\tilde U}$ onto a subspace $\tilde U$ that approximates the ground-truth signal subspace $U$ associated with the unknown configuration $\vartheta$. From $\tilde U$ one forms the perturbed MUSIC function $\tilde q$, whose smallest minima serve as proxies for the true parameters.

The first stage of Gradient-MUSIC requires evaluating $\tilde q$ on a grid $G$. To keep the computational complexity low, this grid should be chosen as coarse as possible while still resolving the relevant geometric features of the landscape. The appropriate measure of coarseness is the \emph{mesh norm},
\begin{equation}
\label{eq:meshG}
\mesh(G)
:=
\max_{\omega\in\domain}\min_{\omega_0\in G} |\omega-\omega_0|.
\end{equation}
{A key result of our subsequent analysis in \cref{sec:example1} and \cref{sec:example2} is that the grid needs only satisfy}
\begin{equation}
\label{eq:meshsize}
\mesh(G)\asymp \frac{1}{dm}
\end{equation}
and is independent of the noise level. The contour plot in \cref{fig:example} illustrates this phenomenon and shows why coarse thresholding of the MUSIC function is sufficient to identify reliable initial regions for local descent.
This verification is a substantial technical step and is carried out in \cref{sec:example1,sec:prooflemmas1} for  discrete cubes and in \cref{sec:example2,sec:prooflemmas2} for continuous balls.

This observation is central to the computational advantage of Gradient-MUSIC over classical MUSIC, whose exhaustive grid search is tied directly to the target accuracy; see \cref{sec:complexity1}. 

\subsection{Subspace estimation}
\label{sec:subspace}

The groundwork  for applying the structural theorem and algorithm is the approximation of the projection operator onto the signal subspace $U=\spann\{e^{2\pi i x\cdot \theta_\ell}\}_{\ell=1}^s.$
This step is both natural and unavoidable: if the parameters $\{\theta_\ell\}_{\ell=1}^s$ were known, then the associated subspace would be known as well. In this sense, subspace estimation is the minimal linear-algebraic task underlying parameter recovery.

A standard way to estimate this subspace is to lift the measurement data $\tilde y$ on $\sampset_\star$ to a structured Hankel matrix or operator. 
 Under this assumption we define the Hankel-type operator
\[
H(\tilde y)\colon L^2_{\nu'}(X')\to L^2_\nu(X)
\]
formally by
\begin{equation}
\label{eq:hankel}
H(\tilde y)f(x)=\int_{X'} \tilde y(x+x')f(x')\,d\nu'(x').
\end{equation}
Since $x\in X$ and $x'\in X'$ imply $x+x'\in \sampset_\star$, the operator $H(\tilde y)$ depends only on the observed data $\tilde y|_{\sampset_\star}$. When $\tilde y(x+x')\in L^2_{\nu\times \nu'}(X\times X')$ (i.e., we interpret $\tilde y$ as a function on $X\times X'$), the operator $H(\tilde y)$ is Hilbert--Schmidt and consequently has point spectrum with zero as the only accumulation point. 
To well define
$$
\tilde U\subset L^2_\nu(X)
$$
as the leading $s$-dimensional left singular space of $H(\tilde y)$ we need a singular value gap, 
\begin{equation}
	\label{eq:siggap}
	\sigma_s(H(\tilde y))> \sigma_{s+1}(H(\tilde y)),
\end{equation}
see  \cref{lem:abstractP} below. 

When $U$ and $\tilde U$ are finite-dimensional,$\|P_U-P_{\tilde U}\|$ 
 is exactly the sine of the largest canonical angle between the two subspaces. Thus the problem of subspace estimation is naturally quantified in terms of principal-angle perturbation.

This perspective is foundational in modern perturbation theory. In the case where $H(\tilde y)$ can be identified with a matrix (i.e., both $\nu$ and $\nu'$ are purely atomic with finitely many dirac masses), a central tool is the Wedin sine-theta theorem  \cite{stewart1990matrix}. We will, however, adopt the approach of a  general theory for perturbation of linear operators in Kato \cite[Theorem 3.18 on page 214]{kato2013perturbation}.

To this end, define the operators
\begin{align}
	T_\vartheta\colon \C^s\to L^2_\nu(X), &\wherespace T_\vartheta v(x) := \sum_{\ell=1}^s v_\ell e^{2\pi i \theta_\ell\cdot x} \in L^2_\nu(X), \label{eq:synthesis1} \\
	T_\vartheta'\colon \C^s\to L^2_{\nu'}(X'), &\wherespace (T_\vartheta') v(x) := \sum_{\ell=1}^s v_\ell e^{2\pi i \theta_\ell\cdot x} \in L^2_{\nu'}(X'). \label{eq:synthesis2}
\end{align}
The operator $T_\vartheta$ has the adjoint 
$$
T_\vartheta^*
\colon L^2_\nu(\sampset)\to \C^s, \wherespace (T_\vartheta^*f)_\ell = \int_\sampset f(x) e^{-2\pi i \theta_\ell \cdot x}\, d\nu(x).
$$
The kernel matrix $K_\vartheta$ in \eqref{eq:kernelmatrix} and $T_\vartheta$ are related by the formula
\begin{equation}
	K_\vartheta = \frac 1 {\nu(X)}  T_\vartheta^* T_\vartheta.
	\label{eq:relationshipKT}
\end{equation}

In the language of harmonic analysis, in particular frame theory, both $T_\vartheta$ and $T_\vartheta'$ can be interpreted as ``synthesis" operators. We prove the following lemma  in \cref{proof:abstractP}.
\begin{lemma}
	\label{lem:abstractP}
	For $y$ defined in \eqref{eq:y}, the operator $H(y)$ has rank $s$ and 
	$$
	\sigma_s(H(y))
	\geq \sigma_s(T_{\vartheta}) \sigma_s(T_{\vartheta}').
	$$ 
	Suppose 
	\begin{equation}\label{eq:perturb}
	4\|H(\eta)\|\leq \sigma_s(H(y)) \quad\hbox{and}\quad \tilde y(x+x')\in L^2_{\nu\times \nu'}(X\times X'). 
	\end{equation} Then 
	$$
	\sigma_s(H(\tilde y))> \sigma_{s+1}(H(\tilde y)),
	$$ 
	the projection $P_{\tilde U}$ onto the $s$-dimensional leading left singular space of $H(\tilde y)$ is well-defined, and 
	$$
	\norm{P_U-P_{\tilde U}}
	\leq \frac {64}{\pi} \frac{\sigma_1(H(y))}{\sigma_s^2(H(y))} \, \|H(\eta)\|
	\leq \frac {64 |a|_\infty}{\pi} \frac{\sigma_1(T_\vartheta)\sigma_1(T_\vartheta')}{\sigma_s^2(T_\vartheta)\sigma_s^2(T_\vartheta')} \, \|H(\eta)\|.
	$$
	If additionally $\nu$ and $\nu'$ are purely atomic with a finite number of atoms (i.e., $H(y)$ and $H(\tilde y)$ are matrices), then the previous inequality can be improved to
	$$
	\norm{P_U-P_{\tilde U}}
	\leq \frac 43 \frac{1}{\sigma_s(H(y))} \|H(\eta)\|
	\leq \frac 43 \frac{1}{\sigma_s(T_{\vartheta}) \sigma_s(T_{\vartheta}')} \|H(\eta)\|.
	$$
\end{lemma}

In addition to initiating the Gradient-MUSIC algorithm, \cref{lem:abstractP} is an important tool  for showing that Gradient-MUSIC achieves the minimax optimality under adversarial noise and noisy super-resolution scaling under random noise by deriving useful estimates of $\sigma_s(H(y))$ and $\|H(\eta)\|$.

\section{Model case 1: Discrete samples in a cube}
\label{sec:example1}
\label{sec:exam1}

In this section, we illustrate the structural theorem by specializing  to the case where $\domain=\T^d$ is the $d$-dimensional torus identified with either $[-1/2,1/2)^d$ or $[0,1)^d$ with periodic boundary, and 
	the sampling set 
\[
	\sampset_\star := \{-2m,\dots,2m\}^d = Q_{2m}\cap \Z^d,\quad m \in \N.
	\]
	This is a canonical sampling geometry for multidimensional spectral estimation. 
The natural $\ell^2$ metric on $\T^d$ is defined as 
	$$
	|\omega-\omega'|:=\min_{n\in \Z^d} |\omega-\omega'+n|.
	$$
	Recall that the minimum separation $\Delta(\vartheta)$ in \eqref{eq:minsep} is defined using this metric. 
	\cref{assump:symmetry} automatically holds since $Q_{2m}\cap\Z^d = -(Q_{2m}\cap\Z^d)$ and $\nu$ is an atomic measure with finitely many Dirac masses. 
		
	In this case, we have $\phi_\omega\colon Q_m\cap\Z^d\to \C$ where
	$$
	\phi_\omega(x)
	:=\frac 1 {(2m+1)^{d/2}} \, e^{2\pi i \omega\cdot x}.
	$$
	Next we compute the measurement kernel $K$ associated to this setup. For convenience, we define the single variable Dirichlet kernel $d_m\colon \T\to \R$ as
	\begin{equation}
		\label{eq:dirichlet}
		d_m(t)
		:=\frac 1 {2m+1} \sum_{k=-m}^m e^{2\pi i  kt}
		= \frac 1 {2m+1}\frac{\sin ((2m+1)\pi t)}{\sin(\pi t)}.
	\end{equation}
	Note that $d_m$ has been normalized so that $d_m(0)=1$. To differentiate between notation for the general setting and this particular case, instead of writing $K$, we use the notation $D_m$ for the {\it (square) Dirichlet kernel}, 
	\begin{equation}
		\label{eq:squareDirichlet}
		D_m(\xi)
		:=\frac 1 {(2m+1)^d} \sum_{x\in Q_m\cap \Z^d} e^{2\pi i x\cdot \xi}
		= d_m(\xi_1)\cdots d_m(\xi_d).
	\end{equation}
	Notice that $D_m$ is a tensor product of single variable Dirichlet kernels $d_m$ and is infinitely differentiable.  Another calculation yields
	\begin{equation}\label{eq:hessian1}
		\Psi
		:=-\nabla^2 D_m(0)
		=\frac {4\pi^2}3 m(m+1) I. 
	\end{equation}

We briefly explain how the approximating subspace $\tilde U$ is constructed in this case. 
 For this model case, we define the effective sampling set and measure, respectively, by
	$$
	\sampset :=Q_m\cap\Z^d
	= \{-m,\dots,m\}^d, \andspace
	\nu=\sum_{x\in \sampset} \delta_x, 
	$$
	and fix an arbitrary ordering of the elements of $X$, and denote by $x^{(j)}\in X$ the $j$-th point in this ordering. Associated with the noisy data $\tilde y$ is the \emph{multilevel Hankel matrix}
\[
H(\tilde y)_{j,k}:=\tilde y\bigl(x^{(j)}+x^{(k)}\bigr),
\qquad j,k.
\]

Since $X+X=\sampset_\star$, the matrix $H(\tilde y)$ depends only on the observed samples of $\tilde y$ on $\sampset_\star$. Viewed as an operator on $\ell^2(X)$, it has a well-defined leading $s$-dimensional left singular space, and throughout this example we take $\tilde U$ to be precisely this space. This is the standard choice in subspace methods, although other constructions are possible.
{This is only one of many choices for an approximate projection operator. For instance, \cite{haardt2008higher} shows that a higher-order SVD of a Hankel tensor works well in practice. }

The next theorem gives the corresponding performance guarantees for Gradient-MUSIC.

\begin{theorem}[Gradient-MUSIC for discrete samples in a cube]
\label{thm:maincube}
Suppose we are given noisy samples of $\tilde y=y+\eta$ on
\[
\sampset_\star=\{-2m,\dots,2m\}^d,
\]
with $d\ge 2$ and $m\ge 1$, and suppose
\[
\vartheta=\{\theta_\ell\}_{\ell=1}^s\subset \T^d,
\qquad s\ge 1,
\]
satisfies the separation condition
\[
\Delta(\vartheta)\gtrsim_d m^{-1}.
\]
Choose a grid $G\subset \T^d$ satisfying
\[
\mesh(G)\asymp_d m^{-1},
\qquad
\alpha_1\asymp_d 1,
\qquad
h\asymp m^{-2}.
\]
Then the following hold.

\begin{enumerate}[(a)]
\item \textbf{Adversarial noise.}
Suppose that for some $p\in[1,\infty]$,
\begin{equation}
\label{eq:etacondcube}
\|\eta\|_{\ell^p(\sampset_\star)}
\lesssim
\frac{m^{d/p}}{|a|_\infty}.
\end{equation}
Then $\tilde U$ is well-defined as the leading $s$-dimensional left singular space of $H(\tilde y)$. 

Let $\{\theta_\ell^\sharp\}_{\ell=1}^s$ denote the $n$-th iterate of Gradient-MUSIC, where
\begin{equation}
\label{eq:ngraditercube1}
n\asymp
\log\!\left(
\frac{m^{d/p}}{\|\eta\|_{\ell^p(\sampset_\star)}}
\right).
\end{equation}
Then
\begin{equation}
\label{eq:cubecasea}
\max_{\ell=1,\dots,s}
\bigl|\theta_\ell-\theta_\ell^\sharp\bigr|
\lesssim_d
\frac{\|\eta\|_{\ell^p(\sampset_\star)}}{m^{1+d/p}}.
\end{equation}

\item \textbf{Random noise.}
Suppose $\eta=\sigma g$, where the values of $g$ on $\sampset_\star$ are i.i.d. standard Normal random variables $\mathcal N(0,1)$, and $\sigma\colon\R^d\to[0,\infty)$ satisfies
\begin{equation}
\label{eq:sigmasizecondcube}
\|\sigma\|_{\ell^2(\sampset_\star)}
\lesssim_d
\frac{m^d}{|a|_\infty\sqrt{\log m}}.
\end{equation}
Then $\tilde U$ is well-defined as the leading $s$-dimensional left singular space of $H(\tilde y)$. 

Let $\{\theta_\ell^\sharp\}_{\ell=1}^s$ denote the $n$-th iterate of Gradient-MUSIC, where
\begin{equation}
n\asymp
\log\!\left(
\frac{m^d}{\|\sigma\|_{\ell^2(\sampset_\star)}\sqrt{\log m}}
\right).
\end{equation}
Then, with probability at least $1-m^{-d}$,
\begin{equation}
\label{eq:cubecaseb}
\max_{\ell=1,\dots,s}
\bigl|\theta_\ell-\theta_\ell^\sharp\bigr|
\lesssim_d
\frac{\|\sigma\|_{\ell^2(\sampset_\star)}\sqrt{\log m}}{m^{d+1}}.
\end{equation}
\end{enumerate}
\end{theorem}

Part~(a) of \cref{thm:maincube} concerns deterministic, or adversarial, perturbations. The size condition \eqref{eq:etacondcube} is essentially necessary, and estimate \eqref{eq:cubecasea} matches the minimax-optimal rate in its dependence on $\|\eta\|_{\ell^p(\sampset_\star)}$ and $m$; see \cref{thm:minimaxp}. The conclusion admits an interpretation that is both important and somewhat surprising. For instance, when $p=\infty$, the quantity $\|\eta\|_{\ell^\infty(\sampset_\star)}$ need not decrease with $m$, while the signal itself satisfies
\[
\|y\|_{L^\infty(\R^d)}\le |a|_1.
\]
Thus the pointwise noise level may remain uniformly comparable to the signal strength as the sampling cube grows. Nevertheless, the reconstruction error in \eqref{eq:cubecasea} still decays like $m^{-1}$ (or faster for smaller $p$). In other words, the location error tends to zero even when the local signal-to-noise ratio does not improve. This quantifies a basic benefit of oversampling: increasing the sampling set yields more information about the configuration even when individual measurements do not become more accurate.

Part~(b) concerns stochastic noise and exhibits a genuinely stronger scaling law. If the noise values on $\sampset_\star$ are i.i.d.\ $\mathcal N(0,\sigma^2)$, so that $\sigma$ is constant, then \eqref{eq:cubecaseb} reduces to
\[
\frac{\sigma\sqrt{\log m}}{m^{d/2+1}},
\]
which agrees with the Cram\'er--Rao lower bound in its dependence on $\sigma$ and $m$, up to the logarithmic factor; see \cite{stoica1989music}. The hypothesis \eqref{eq:sigmasizecondcube} is also flexible enough to permit spatially varying noise levels. For example, it allows radial growth of the form
\[
\sigma(x)=(1+|x|)^r,
\qquad r\in(0,d/2),
\]
in which case the bound becomes
\[
\frac{\sqrt{\log m}}{m^{d+1}}\|\sigma\|_{\ell^2(\sampset_\star)}
\asymp
\frac{\sqrt{\log m}}{m^{d/2-r+1}}.
\]
Thus even when the noise variance increases with distance from the origin, and eventually dominates the signal magnitude in the sense that $\|y\|_{L^\infty(\R^d)}\le |a|_1$ remains bounded while $\sigma(x)$ grows, the reconstruction error may still converge to zero as $m\to\infty$. This illustrates a striking feature of the noisy super-resolution scaling: enlarging the sampling aperture can continue to improve recovery even when the additional samples appear, locally, to be increasingly noise-dominated.

	\subsection{Approximating the projection operator}
	\label{sec:hankeldiscrete}
	
	Here, we will specialize the general outline of \cref{sec:subspace} and control the various quantities that appear in \cref{lem:abstractP}. 

	Define the Fourier matrix,
		\begin{equation}
			\label{eq:fouriermatrixcube}
			T_\vartheta:=\left[ e^{2\pi i x \cdot \theta_\ell}\right]_{x\in Q_m\cap\Z^d, \, \ell \in \{1\dots,s\}}.
		\end{equation}
		This is a (multivariate) Fourier matrix whose nodes are $\{\theta_\ell\}_{\ell=1}^s\subset\T^d$, and its rows are indexed by $Q_m\cap\Z^d$. This matrix is precisely the operator defined in \eqref{eq:synthesis1}. 		
	For a function $y$ satisfying \cref{def:spectral} we have the factorization,
	\begin{equation}\label{eq:hankelfactorization}
		H(y) = T_\vartheta \, \diag(\bfa) \, T_\vartheta^T.
	\end{equation}
	This is the multidimensional analogue of the usual Fourier matrix factorization of a Hankel matrix. 
	Due to the matrix factorization of $H(y)$ in  \eqref{eq:hankelfactorization} and that $\min_\ell |a_\ell|= 1$ (see \cref{def:spectral}), we have
	\begin{equation}
		\label{eq:sigsH}
		\sigma_s(H(y))
		\geq \sigma_{\min}^2(T_\vartheta).
	\end{equation}
	
Note that
	$$
	H(\tilde y)=H(y)+H(\eta). 
	$$	To control $\|H(\eta)\|$, we interpret $H(\eta)$ as a convolution operator on $\Z^d$. By Young's inequality and manipulations, 
	\begin{equation} \label{eq:projdiscrete1}
		\|H(\eta)\|
		\leq (4m+1)^{d/p'} \|\eta\|_{\ell^p(Q_{2m}\cap\Z^d)}.
	\end{equation} 
	We omit the details since its proof is standard and for $d=1$ can be found in \cite[Lemma 5.7]{fannjiang2025optimality}. While this bound holds for arbitrary $\eta$, we expect it to be loose if the values of $\eta$ on $Q_{2m}\cap\Z^d$ are independent random variables, since the $\ell^p$ norm does not take advantage of any possible cancellations. To remedy this, we have the following result, which is proved in \cref{proof:projdiscrete2}. 	
	
	\begin{lemma}
		\label{lem:projdiscrete2}
		Fix any $\sigma\colon \R^d\to [0,\infty)$. Suppose $\eta(x)=\sigma(x)g(x)$ for each $x\in Q_{2m}\cap \Z^d$ where $g(x)\sim\calN(0,1)$ and $g(x),g(x')$ are independent whenever $x\ne x'$. For any $\delta\in (0,1)$, with probability at least $1-\delta$,
		\begin{equation*}
			\|H(\eta)\|
			\leq \|\sigma\|_{\ell^2(Q_{2m}\cap \Z^d)} \sqrt{ \log(1/\delta)+ d \log(m)}. 
		\end{equation*}
		In particular, fix any $\sigma>0$ and assume that $\sigma(x)=\sigma$ for all $x\in Q_{2m}\cap \Z^d$. For any $\delta\in (0,1)$, with probability at least $1-\delta$,
		$$
		\|H(\eta)\|
		\lesssim_d \sigma m^{d/2} \sqrt{ \log(1/\delta)+ \log(m)}.
		$$
	\end{lemma}
	    
    \begin{remark}
        \label{rem:sparsitydetection}
        Under the assumptions of \cref{thm:maincube}, for both noise models, the value of $s$ can be deduced by examining the ratio 
        $$
        \frac{\sigma_k(H(\tilde y))}{\sigma_1(H(\tilde y))}. 
        $$
        The basic rationale (see \cite[Lemma 5.6]{fannjiang2025optimality} for a proof) is that if 
        $$
        \|H(\eta)\|\leq c_d |a|_\infty^{-1} m^d
        $$
        for a sufficiently small $c_d$ (and this inequality holds for both noise models due to \eqref{eq:projdiscrete1} and \cref{lem:projdiscrete2}), then 
        $$
        \frac{\sigma_s(H(\tilde y))}{\sigma_1(H(\tilde y))} \gtrsim \frac 1{|a|_\infty}, \andspace
        \frac{\sigma_{s+1}(H(\tilde y))}{\sigma_1(H(\tilde y))}
        \lesssim \frac{\|H(\eta)\|}{\sigma_s(H(y))}
        \leq \frac{c_d}{|a|_\infty}.
        $$
        Hence, for $c_d$ sufficiently small, $s$ can be correctly computed from just the observation $\tilde y|_{X_\star}$.  
    \end{remark}

	\subsection{Uniform estimates of the kernel}
	\label{sec:verification1}
	
	We will compute the various quantities that appear in \cref{thm:maingeometric} uniformly over the class of $\vartheta$ satisfying a separation condition $\Delta(\vartheta)\geq \beta/m$ for a large enough $\beta$. 
	
	We will start with the easier to calculate quantities. The following technical lemma is shown in \cref{proof:dirichletconstants}.
	 
	\begin{lemma}\label{lem:dirichletconstants} 
		We have $\|\nabla D_m\|_{L^\infty(B_\tau)} \leq 4\pi^2 m^2 \tau$ and $\Delta^2 D_m(0)\leq 16\pi^4 d^2 m^2 (m+1)^2/5$. 
	\end{lemma}
	
   Next we calculate the extreme eigenvalues of the kernel matrix $K_\vartheta$. 
	
	Define the $\ell^\infty$ minimum separation by 
	$$
	\Delta_\infty(\vartheta)
	:=\min_{j\not=k} \min_{n\in \Z^d} |\theta_j-\theta_k+n|.
	$$
	The singular values of such matrices were studied in \cite{li2025nonharmonic}. 
	\begin{proposition}[\cite{li2025nonharmonic}, Theorem 2.3]\label{prop:lifourier}
		For any $m\geq 1$, $d\geq 2$, $\beta\geq \pi/\log(2)$, and $\vartheta=\{\theta_\ell\}_{\ell=1}^s\subset \T^d$ such that $\Delta(\vartheta) \geq \beta d/m$, 
		we have 
		$$
		(2m+1)^{d/2} \sqrt{1-(e^{1/(2\beta)}-1)}
		\leq \sigma_{\min}(T_\vartheta)
		\leq \sigma_{\max}(T_\vartheta)
		\leq (2m+1)^{d/2}\sqrt{1+(e^{1/(2\beta)}-1)}.
		$$
	\end{proposition}
	
	This proposition is stated using the $\ell^\infty$ norm on the separation of $\vartheta$ rather than the $\ell^2$ distance. The following lemma is a reformulation of \cref{prop:lifourier} for future ease of reference. 
	
	\begin{lemma}
		\label{lem:lifourier2}
		For any $d\geq 2$, $\vartheta\subset \T^d$, and $\beta\geq \pi d^{3/2} /\log(2)$ such that $\Delta(\vartheta)\geq \beta/m$, we have
		$$
		1-\frac{d^{3/2}}{\beta}
		\leq \lambda_{\min}(K_\vartheta)
		\leq \lambda_{\max}(K_\vartheta)
		\leq 1+ \frac{d^{3/2}}{\beta}.
		$$
	\end{lemma}
	
	\begin{proof}
		Notice that the kernel matrix \eqref{eq:kernelmatrix} is related to $T_\vartheta$ by the formula
		\begin{equation}\label{eq:KTrelation}
			K_\vartheta = \frac 1 {(2m+1)^d} \, T_\vartheta^* T_\vartheta.
		\end{equation}
		We have 
		$$
		\Delta_\infty (\vartheta)
		=\min_{j\ne k} |\theta_j-\theta_k|_{\infty}
		\geq \frac 1 {\sqrt d} \min_{j\ne k} |\theta_j-\theta_k| 
		= \frac 1 {\sqrt d} \, \Delta(\vartheta)
		\geq \frac{\beta}{\sqrt d m}. 
		$$
		We use \cref{prop:lifourier} (where $\beta/d^{3/2}$ plays the role of $\beta$) and use \eqref{eq:KTrelation} to see that
		$$
		1-\left(e^{d^{3/2}/(2\beta)}-1\right)
		\leq \lambda_{\min}(K_\vartheta)
		\leq \lambda_{\max}(K_\vartheta)
		\leq 1+\left(e^{d^{3/2}/(2\beta)}-1\right).
		$$
		To finish the proof, we use the inequality $e^t-1\leq 2t$ for all $t\in [0,1]$.
	\end{proof}

Next we estimate the behavior of $E_0(\vartheta)$ and $E_1(\vartheta)$. The main technical obstruction is that $D_m$ decays slowly along the coordinate axes. Indeed, 
$$
|d_m(t)|\lesssim \min \left\{1,\frac 1 {m|t|} \right\},
$$
so for any $\xi$ on the $k$-coordinate axis, 
$$
|D_m(\xi)|\lesssim \min \left\{1,\frac 1 {m|\xi_k|} \right\},
$$
regardless of dimension $d$. This decay is too weak to control $E_0(\vartheta)$ and $E_1(\vartheta)$ through a naive argument. Instead, we must resort to a more intricate strategy. The following lemma requires additional preparatory lemmas, which are stated and proved in \cref{proof:dirichletenergy2}.
	
	\begin{lemma}\label{lem:dirichletenergy2}
		For any $d\geq 2$, $\vartheta\subset \T^d$, and $\beta \geq 4\pi \sqrt d$ such that $\Delta:=\Delta(\vartheta)\geq \beta/m$, we have
		$$
		E_0(\vartheta)=\frac{16\pi^2 d^2}{\beta^2} \andspace
		E_1(\vartheta)=\frac{64\pi^2 d^3 m^2}{\beta^2}.
		$$
	\end{lemma}
	
	\begin{proof}
		Fix any $\vartheta$ satisfying the lemma's assumptions and $\omega\in\T^d$. Consider the set
		$$
		\Gamma:=\{\omega-\theta_\ell\colon |\omega-\theta_\ell|\geq \Delta/2\}.
		$$
		Since $\Gamma$ is a subset of $\vartheta$ shifted by $\omega$ which does not change its minimum separation, we have $\Delta(\Gamma)\geq \Delta(\vartheta)\geq \beta/m$ and consequently, $\Delta_\infty(\Gamma)\geq \beta/(\sqrt d m)$. By construction, $\emptyset=\Gamma\cap B_{\Delta/2}\supset \Gamma\cap B_{\beta/(2m)}\supset \Gamma \cap Q_{\beta/(2 \sqrt d m)}$. Then we apply \cref{lem:dirichletenergy} for $\Gamma$ to get
		\begin{align*}
			\sum_{\ell\colon |\omega-\theta_\ell|\geq \Delta/2}
			|D_m(\omega-\theta_\ell)|^2
			&\leq \frac{16\pi^2 d^2}{\beta^2}, \\
			\sum_{\ell\colon |\omega-\theta_\ell|\geq \Delta/2} |\nabla D_m(\omega-\theta_\ell)|^2
			&\leq \frac{64\pi^4 d^3 m^2}{\beta^2}. 
		\end{align*}
		Both inequalities hold uniformly over $\omega\in \T^d$, which completes the proof. 
	\end{proof}

	At this point, we have obtained inequalities for the quantities that appear in \cref{thm:maingeometric} which are uniform in $\vartheta$, depending only on $\beta$, the constant that appears in the separation condition $\Delta(\vartheta)\geq \beta/m$. The behavior of these quantities are summarized in \cref{table:kernel2}. This is how we strengthen \cref{thm:maingeometric} (a nonuniform result) to a class of parameters.

Now we have all the ingredients for verifying the assumptions of Theorem \ref{thm:maingeometric} and carry this out to completion in \cref{proof:maincube}. 
	
	\section{Model case 2:  Continuous samples in a ball} 
	\label{sec:example2}
\label{sec:exam2}

We next consider the continuous sampling geometry
\[
\sampset_\star=B_{2m}\subset \R^d,
\]
the ball of radius $2m$ centered at the origin. This model is rotationally invariant and is especially natural in imaging and tomography. 

	This represents a continuous sampling setup, where $\tilde y(x)$ is known at each $x\in B_{2m}$ and
	we define the effective sampling set
		$$ 
		 \sampset:= B_m  \andspace 
\nu:=|\cdot| \quad \text{is Lebesgue measure restricted to $B_m$}.
		$$	
	The symmetry and moment conditions in \cref{assump:symmetry} clearly hold. 
	
	For this case, we have $\phi_\omega\colon B_m\to\C$, where  
	$$
	\phi_\omega(x) = \frac 1{\sqrt{|B_m|}}e^{2\pi i \omega\cdot x}.
	$$
	To distinguish between the general notation and this special setting, we use $W_m$ to denote the measurement kernel $K$. Letting $J_\alpha$ be the Bessel function of the first kind with parameter $\alpha$, where we follow the convention in \cite[equation (8), papge 40]{watson1922treatise}, we have the formula,
	\begin{equation}
		\label{eq:Besselkernel}
		W_m(\xi)
		:= \frac{1}{|B_m|} \int_{B_m} e^{2\pi i  \xi \cdot x}\, dx
		= \frac{1}{|B_1|} \frac{ J_{d/2}(2\pi m|\xi|)}{m^{d/2}|\xi|^{d/2} }. 
	\end{equation}
	A calculation shows that
	\begin{equation}
		\label{eq:hessian2}
		\Psi:=-\nabla^2 W_m(0)= \frac {1}{d+2} 4\pi^2 m^2 I.
	\end{equation}

{For the perturbation $\eta$, we allow either deterministic noise or a class of Gaussian random fields with mild regularity, described below.

\begin{definition}[Gaussian random field with spectral representation]
\label{def:etastochastic}
Let $W$ be a standard Brownian field on $\R^d$, and let $\gamma\colon \R^d\to \C$ be bounded and compactly supported, normalized so that
\[
\|\gamma\|_{L^2(\R^d)}=1.
\]
We say that $g\colon \R^d\to \C$ is a \emph{Gaussian random field with spectral representation} $\gamma$ if
\[
g(x)=\int_{\R^d} \gamma(t)e^{2\pi i t\cdot x}\,dW(t).
\]
Then $g$ is stationary, mean zero, and has unit variance:
\[
\E |g(x)|^2=\|\gamma\|_{L^2(\R^d)}^2=1.
\]
\end{definition}
}

The following theorem gives the corresponding guarantees for Gradient-MUSIC. 

\begin{theorem}[Gradient-MUSIC for continuous samples in a ball]
\label{thm:mainball}
Suppose we are given noisy samples of $\tilde y=y+\eta$ on
\[
\sampset_\star=B_{2m}\subset \R^d,
\]
with $d\ge 1$ and $m>0$. Let $\Omega\subset \R^d$ be bounded, and suppose
\[
\vartheta=\{\theta_\ell\}_{\ell=1}^s\subset \Omega,
\qquad s\ge 1,
\]
satisfies the separation condition
\[
\Delta(\vartheta)\gtrsim_d m^{-1}.
\]
Choose a grid $G\subset \Omega$ satisfying
\[
\mesh(G)\asymp_d m^{-1},
\qquad
\alpha_1\asymp_d 1,
\qquad
h\asymp m^{-2}.
\]
Then the following hold.

\begin{enumerate}[(a)]
\item \textbf{Adversarial noise.}
Suppose that for some $p\in[1,\infty]$,
\begin{equation}
\label{eq:etacondball}
\|\eta\|_{L^p(B_{2m})}
\lesssim_d \frac{m^{d/p}}{|\bfa|_\infty}.
\end{equation}
Then $\tilde U$ is well-defined as the leading $s$-dimensional left singular space of $H(\tilde y)$.

Let $\{\theta_\ell^\sharp\}_{\ell=1}^s$ denote the $n$-th iterate of Gradient-MUSIC, where
\begin{equation}
\label{eq:ngraditerball1}
n\asymp
\log\!\left(
\frac{m^{d/p}}{|\bfa|_\infty \|\eta\|_{L^p(B_{2m})}}
\right).
\end{equation}
Then
\begin{equation}
\label{eq:ballcasea}
\max_{\ell=1,\dots,s}
\bigl|\theta_\ell-\theta_\ell^\sharp\bigr|
\lesssim_d
\frac{|\bfa|_\infty \|\eta\|_{L^p(B_{2m})}}{m^{1+d/p}}.
\end{equation}

\item \textbf{Random noise.}
Suppose $\eta=\sigma g$, where $g$ is a Gaussian random field with spectral representation $\gamma$. Let $\chi_m$ be any function satisfying $\chi_m=1$ on $B_{2m}$, and let $\sigma\in L^2(\R^d)$ be nonnegative and satisfy
\begin{align}
\|\sigma\|_{L^2(B_{2m})}
&\lesssim_d
\frac{m^d}{|\bfa|_\infty\sqrt{\log m}},
\label{eq:sigmasizecondball}
\\
\int_{\R^d}
\bigl|\widehat{\sigma\chi_m}(\xi)\bigr|^2
\log(1+|\xi|)\,d\xi
&\lesssim_d
\log(m)\,\|\sigma\chi_m\|_{L^2(\R^d)}^2.
\label{eq:sigmasizecondball2}
\end{align}
Then $\tilde U$ is well-defined as the leading $s$-dimensional left singular space of $H(\tilde y)$.

Let $\{\theta_\ell^\sharp\}_{\ell=1}^s$ denote the $n$-th iterate of Gradient-MUSIC, where
\begin{equation}
\label{eq:ngraditerball2}
n\asymp
\log\!\left(
\frac{m^d}{|\bfa|_\infty \|\sigma\chi_m\|_{L^2(\R^d)}\sqrt{\log m}}
\right).
\end{equation}
Then, with probability at least $1-m^{-d}$,
\begin{equation}
\label{eq:ballcaseb}
\max_{\ell=1,\dots,s}
\bigl|\theta_\ell-\theta_\ell^\sharp\bigr|
\lesssim_{d,\gamma}
\frac{|\bfa|_\infty \|\sigma\chi_m\|_{L^2(\R^d)}\sqrt{\log m}}{m^{d+1}}.
\end{equation}
\end{enumerate}
\end{theorem}

Part~(a) of \cref{thm:mainball} concerns deterministic, or adversarial, perturbations. The size condition \eqref{eq:etacondball} is essentially necessary, and estimate \eqref{eq:ballcasea} matches the minimax-optimal rate in its dependence on $\|\eta\|_{L^p(B_{2m})}$ and $m$; see \cref{thm:minimaxp2}. 

Part~(b) concerns Gaussian random fields with spectral representation. The size condition \eqref{eq:sigmasizecondball} is mild and allows the noise amplitude to grow with $|x|$; for example, it permits
\[
\sigma(x)=(1+|x|)^r,
\qquad r\in(0,d/2).
\]
The second condition \eqref{eq:sigmasizecondball2} is a mild regularity assumption on $\sigma\chi_m$. By Parseval's identity, the same inequality without the logarithmic weight would hold automatically, so \eqref{eq:sigmasizecondball2} requires only a very small amount of additional smoothness. A standard example is the constant case $\sigma(x)\equiv \sigma$, for which \eqref{eq:ballcaseb} reduces to
\[
\frac{|\bfa|_\infty \sigma\sqrt{\log m}}{m^{d/2+1}}.
\]

This framework also captures more realistic imaging models. In tomography, for instance, one often samples along hyperplanes through the origin whose orientations are uniformly distributed, with measurements corrupted by stationary noise. Because this sampling scheme oversamples regions near the origin, the effective signal-to-noise ratio is spatially inhomogeneous. After suitable preprocessing { and re-weighting as discussed in \cref{sec:scale}}, the resulting problem is naturally modeled as uniform sampling on a ball with nonstationary noise, precisely of the type covered by \cref{thm:mainball}. 

	\subsection{Approximating the projection operator}
	\label{sec:hankelcontinuous}
	
	Here, we will specialize the general outline of \cref{sec:subspace} and control the various quantities that appear in \cref{lem:abstractP}. 
		
	We start with a few basic observations about the Hankel operator $H(h)$ defined in \cref{eq:hankel}. For $h\in L^1 (B_{2m})$ and $f\in L^2(B_m)$, we can extend $\bbone_{B_{2m}}h$ to $\R^d$ by zero padding, and get
	$$
	H(h)f=h*Rf
	$$
	where
	$$
	Rf(x):=f(-x),  \quad x\in B_m.
	$$ 
	Young's convolution inequality implies, 
	\begin{equation}\label{eq:hankelnoise2}
		\norm{H(h)}
		\leq \|h\|_{L^1(B_{2m})}. 
	\end{equation}

	For the deterministic case, inequality \eqref{eq:hankelnoise2} suffices. We expect it to be loose if $\eta$ is stochastic. We have the subsequent lemma for the class of Gaussian random fields in \cref{def:etastochastic}, which is proved in \cref{proof:contnoisestochastic}.
	
	\begin{lemma}
		\label{lem:contnoisestochastic}
		Let $g$ be a Gaussian random field with spectral representation $\gamma$, $\chi_m$ be any function such that $\chi_m=1$ on $B_{2m}$, and $\sigma\in L^2(\R^d)$ be a real nonnegative function such that 		
		\begin{equation}
			\label{eq:sigmacond0}
			\int_{\R^d} \left| \hat{\sigma \chi_m}(\xi) \right|^2 \log(1+|\xi|) \, d\xi
			\lesssim \log(m) \|\sigma \chi_m \|_{L^2(\R^d)}^2 .
		\end{equation}
		Let $\eta=\sigma g$. For any $\delta\in (0,1)$, with probability at least $1-\delta$, we have 
		$$
		\|H(\eta)\|
		\lesssim_{d, \gamma} \|\sigma\chi_m\|_{L^2(\R^d)} \sqrt{\log(1/\delta)+\log(m)}.
		$$
		In particular, for any $\sigma>0$, if $\sigma(x)=\sigma$ for all $x\in\R^d$, then for any $\delta\in (0,1)$, with probability at least $1-\delta$,
		$$
		\|H(\eta)\|
		\lesssim_{d,\gamma} \sigma m^{d/2} \sqrt{\log(1/\delta)+\log(m)}.
		$$ 
	\end{lemma}

	\begin{remark}
	    \label{rem:sparsitydection2}
        For the purposes of this paper, we ignore the numerical error required to compute $P_{\tilde U}$. This would involve computing the leading singular functions of a integral kernel operator. The numerical error of computing $P_{\tilde U}$ can be implicitly absorbed into the term $\|P_U-P_{\tilde U}\|$. If we redefine $P_{\tilde U}$ as the numerically computed projection, then $\|P_U-P_{\tilde U}\|$ contains both the approximation and numerical errors. Analogous to \cref{rem:sparsitydetection}, if 
        $$
        \|H(\eta)\|\leq c_d |a|_\infty^{-1} m^d
        $$ 
        for a sufficiently small $c_d>0$ (which holds in both noise models in \cref{thm:mainball} due to inequality \eqref{eq:hankelnoise2} and \cref{lem:contnoisestochastic}), then $s$ can be detected by thresholding the ratio $\sigma_k(H(\tilde y))/\sigma_1(H(\tilde y))$.
	\end{remark}

	\subsection{Uniform estimates of the kernel}
	\label{sec:verification2}
	
	The following lemma is shown in \cref{proof:besselconstants}.	
	\begin{lemma}\label{lem:besselconstants} 
		We have $\|\nabla W_m\|_{L^\infty(B_\tau)} \leq 4\pi^2 m^2 \tau/(d+2)$ and $\Delta^2 W_m(0)\leq 16\pi^4 m^4 d/(d+4)$. 
	\end{lemma}
	
	For the next set of quantities the extreme singular values of $K_{\vartheta}$ can be controlled using standard techniques. Recall the operator defined in  \eqref{eq:synthesis1}, which becomes $T_\vartheta\colon \C^s\to L^2(B_m)$ defined as
	\begin{equation}
		\label{eq:Fourierball}
		T_\vartheta v(x) := \sum_{\ell=1}^s v_\ell e^{2\pi i \theta_\ell\cdot x}. 
	\end{equation}
	The singular values of $T_\vartheta$ has been studied previously (though written in a different form) in \cite{holt1996beurling,gonccalves2018note,li2025nonharmonic}. Roughly speaking, these results say that $\Delta(\vartheta)\geq C_d/m$ for a big enough $C_d$ implies that the eigenvalues of $K_\vartheta$ are all proportional to one. The following result is essentially folklore, which combines results in \cite{holt1996beurling} with a Poisson summation argument, and is proved in \cref{proof:orthogonalityball}. 

	\begin{lemma}
		\label{lem:orthogonalityball}
		Let $d\geq 1$ and $m>0$. There are constants $C_d, C_d'$ which only depend on $d$ such that for all $\beta\geq C_d'$ and any $\vartheta\subset \T^d$ such that $\Delta(\vartheta)\geq \beta/m$, we have
		$$
		1-\frac{C_d}{\beta}
		\leq \lambda_{\min}(K_\vartheta)
		\leq \lambda_{\max}(K_\vartheta)
		\leq 1+ \frac{C_d}{\beta}.
		$$
	\end{lemma}	
	
	To control the energy quantities, we will derive some pointwise upper bounds for $|W_m|^2$ and $|\nabla W_m|^2$ which decay barely fast enough that a covering argument suffices to control the energy quantities $E_0$ and $E_1$. This is done separately in a lemma stated and proved in \cref{proof:besselenergy}. 
	
	\begin{lemma}
		\label{lem:besselenergy}
		Let $d\geq 1$ and $m>0$. For any $\vartheta\subset\Omega$ such that $\Delta:=\Delta(\vartheta)\geq \beta/m$, we have
		\begin{align*}
			E_0(\vartheta)
			\leq \frac{d\, 4^{d+1}}{\pi^2 |B_1|^2 \beta^{d+1}} \andspace 
			E_1(\vartheta)
			\leq \frac{d \, 4^{d+1} 4\pi^2 m^2 }{|B_1|^2 \beta^{d+1}}.
		\end{align*}
	\end{lemma}
	
	\begin{proof}
		Fix any $\vartheta\subset \Omega$ such that $\Delta(\vartheta)\geq \beta/m$. Fix any $\omega\in\Omega$ and consider the set
		$$
		\Gamma:=\{\omega-\theta_\ell\colon |\omega-\theta_\ell|\geq \Delta/2\}. 
		$$
		Since $\Gamma$ is a subset of a shift of $\vartheta$, we see that $\Delta(\Gamma)\geq \Delta(\vartheta)\geq \beta/m$. By construction, $\emptyset=\Gamma\cap (B_{\Delta/2})^\circ\supset \Gamma\cap (B_{\beta/(2m)})^\circ$. Now we are in position to apply \cref{lem:besselenergy2} which yields 
		\begin{align*}
			\sum_{\ell\colon |\omega-\theta_\ell|\geq \Delta/2} |W_m(\omega-\theta_\ell)|^2
			&\leq \frac{d\, 4^{d+1}}{\pi^2 |B_1|^2 \beta^{d+1}}, \\
			\sum_{\ell\colon |\omega-\theta_\ell|\geq \Delta/2} |\nabla W_m(\omega-\theta_\ell)|^2
			&\leq \frac{d \, 4^{d+1} 4\pi^2 m^2 }{|B_1|^2 \beta^{d+1}}.
		\end{align*}
		The right hand side is uniform over $\omega\in\Omega$ which completes the proof. 	
	\end{proof}
	
	At this point, we have obtained inequalities for the quantities that appear in \cref{thm:maingeometric} which are uniform in $\vartheta$, and only depend on $\beta$, the constant that appears in the separation condition $\Delta(\vartheta)\geq \beta/m$. The behavior of these quantities are summarized in \cref{table:kernel2}. This is how we strengthen \cref{thm:maingeometric} (a nonuniform result) to a class of parameters. 
	
The rest of technical details for verifying assumptions in \cref{thm:maingeometric} are given in \cref{sec:prooflemmas2}. 

	\section{Minimax error bounds}
	\label{sec:minimax}
	
We show in 	this section  that the error bounds in \cref{thm:maincube,thm:mainball} are minimax optimal in the sense defined below. Throughout we fix an integer $s\geq 1$ and let $m\geq s$. We define the {\it parameter space},
	\begin{equation}
		\label{eq:parameterspace}
		\calP:= \big\{(\vartheta,\bfa)\colon |\vartheta|=|\bfa|=s, \, \min_\ell |a_\ell|=1 \big\}\subset (\Omega\times \C)^s. 
	\end{equation}
	
	Let $\calN$ denote a set of functions on $X_\star$, called  the {\it noise class},  which becomes a $\ell^p$ or $L^p$ ball of prescribed radius for Model case 1 or 2, respectively. Let the {\it signal class} $\calS$ be the set of parameters such that $m\Delta\geq \beta$ for an appropriate $\beta$. 

With the notation
	\begin{equation}
		\label{eq:information}
		\tilde y:=\tilde y(\vartheta,\bfa,\eta) = (y+\eta)|_{X_\star}
	\end{equation}	
 we refer to 
 $$
	\calY:=\{\tilde y(\vartheta,\bfa,\eta)\colon (\vartheta,\bfa)\in\calS, \, \eta\in\calN\},
	$$
	as the {\it data space}.

	Any measurable function $\psi: \calY\to\calP$ defines a {\it method}. Let
	$$
	(\hat\vartheta,\hat\bfa)=\psi(\tilde y(\vartheta,\bfa,\eta))
	$$ 
	denote the output of $\psi$. From here onward, we simply write $(\hat\vartheta,\hat\bfa)$ instead of $\psi$. Just like the other parts of this paper, we have implicitly assumed that $\vartheta$ and $\hat\vartheta$ have been indexed to best match each other. 
	
	For any method $\psi$, with its output $(\hat\vartheta,\hat\bfa)$,  the worst case  errors are given by
	\begin{align*}
		\Theta(\psi,\calS,\calN)&:=\sup_{(\vartheta,\bfa)\in \calS} \, \sup_{\eta\in\calN} \,  \max_{j=1,\dots,s} |\theta_\ell-\hat \theta_\ell|, \\
		A(\psi,\calS,\calN)&:=\sup_{(\vartheta,\bfa)\in \calS} \, \sup_{\eta\in\calN} \,  \max_{j=1,\dots,s} |a_\ell-\hat a_\ell|.
	\end{align*}
	For given $\calS$ and $\calN$, the {\it minimax errors }  among all methods are given by
	\begin{align*}
		\Theta_\star(\calS,\calN):=\inf_{\psi} \Theta(\psi,\calS,\calN), \\
		A_\star(\calS,\calN):=\inf_{\psi} A(\psi,\calS,\calN). 
	\end{align*}
	Since we do not preclude the use of intractable algorithms or even non-computable functions, it is not apriori clear  the minimizers exist and, when they do, if they are computable or tractable.  
	
	We are interested in obtaining an explicit lower bounds for these quantities. There is a standard abstract argument (see \cite{batenkov2021super,li2021stable,fannjiang2025optimality}) which reduces the problem down to providing an explicit example. 
	
	\begin{lemma} \label{lem:minmaxhelp}
		Suppose there are distinct pairs $(\vartheta,\bfa),(\vartheta',\bfa')\in \calS$, and $\eta,\eta' \in \calN$ such that 
		\begin{equation}
			\label{eq:data}
			\tilde y(\vartheta,\bfa,\eta)
			=\tilde y(\vartheta',\bfa',\eta').
		\end{equation}
		Then 
		\begin{align*}
			\Theta_\star(\calS,\calN)
			\geq \frac 1 2 \max_{\ell=1,\dots,s} |\theta_\ell-\theta_\ell'|, \andspace 
			A_\star(\calS,\calN)
			\geq \frac 1 2 \max_{\ell=1,\dots,s} |a_\ell-a_\ell'|.
		\end{align*}
	\end{lemma}
	
	\begin{proof}
		By assumption \eqref{eq:data}, there are two sets of parameters $(\vartheta,\bfa)$ and $(\vartheta',\bfa')$ that generate the same noisy data $\tilde y$. Let $\psi$ be an arbitrary method with the output $(\hat\vartheta,\hat\bfa)$. Given data \eqref{eq:data}, by definition of $\Theta(\psi,\calS,\calN)$, we have
		\begin{align*}
			\Theta(\psi,\calS,\calN)
			\geq \max_{\ell=1,\dots,s} |\hat \theta_\ell-\theta_\ell| \andspace
			\Theta(\psi,\calS,\calN) \geq \max_{\ell=1,\dots,s} |\hat \theta_\ell-\theta_\ell'|. 
		\end{align*}
		(Here, $\hat\vartheta$ is sorted to best match $\vartheta$ in the first expression, while it is sorted by a possibly different permutation to best match $\vartheta'$ in the second expression.) By triangle inequality,
		\begin{align*}
			\Theta(\psi,\calS,\calN)
			\geq \max\left\{ \max_{\ell=1,\dots,s} |\hat \theta_\ell-\theta_\ell|, \, \max_{\ell=1,\dots,s} |\hat \theta_\ell-\theta_\ell'| \right\}
			\geq \frac 1 2 \max_{\ell=1,\dots,s} |\theta_\ell-\theta_\ell'|. 
		\end{align*}
		Since $\psi$ is arbitrary, taking the sup over all $\psi$ proves the first claim inequality for $\Theta_\star(\calS,\calN)$. Repeating the same argument for the amplitudes instead completes the proof. 
	\end{proof}
	
For  the two model cases in \cref{sec:example1,sec:example2}, let  $\calS$  have a $\beta$ and $\calN$ a $\ell^p$  or $L^p$ ball as specified  in \cref{thm:maincube} or \cref{thm:mainball}, respectively. 
	
	\begin{theorem}[Minimax lower bound for $X_\star = Q_{2m}\cap\Z^d$]
		\label{thm:minimaxp}
		Let $s\geq 2$, $\beta\geq 1$, $m\geq 2\beta s$, and $X_\star = Q_{2m}\cap\Z^d$. Consider the signal class
		$$
		\calS:=\big\{(\vartheta,\bfa) \colon |\vartheta|=|\bfa|=s, \, \min_\ell |a_\ell|=1, \, m\Delta(\vartheta)\geq \beta \big\}\subset (\T^d\times \C)^s. 
		$$
		For any $p\in [1,\infty]$ and $\epsilon\leq c_d m^{d/p}$ where $c_d$ only depends on $d$, consider the noise class
		$$
		\calN:= \{\eta\colon \|\eta\|_{\ell^p(X_\star)}\leq \epsilon\}.
		$$
		Then we have
		\begin{equation*}
			\Theta_\star(\calS,\calN) \gtrsim_d \frac{\epsilon}{ m^{1+d/p}} \andspace
			A_\star(\calS,\calN)  \gtrsim_d \frac{\epsilon}{m^{d/p}}.
		\end{equation*}
	\end{theorem}
	
	\begin{proof}
	Consider $(\vartheta,\bfa)\in \calS$ where $\theta_1=0$, $|\theta_j-0|\geq 2\beta/m$ for each $j\ne 1$ and $|\theta_j-\theta_k|\geq \beta/m$ for $j\ne k$. Let $a=(1,\dots,1)$ and
	$$
	\delta:=\frac {c_d\, \epsilon} {m^{1+d/p}}.
	$$
	
	Let $e_1$ denote the first canonical basis vector of $\R^d$. Consider the alternative parameters $(\vartheta',\bfa')$ where $\vartheta_1'=\delta e_1$, $\vartheta_j'=\vartheta_j$ for $j\ne 1$, $a'_j = 1+m \delta$, and $a_j'=a_j$ for $j\ne1$. By the assumption on $\epsilon$ and assuming $c_{d}$ is sufficiently small, we have 
	$$
	|\delta|\leq \frac 1 m. 
	$$
	Since $\beta \geq 1$ and $|\theta_j-0|\geq 2\beta/m$, this implies $|\theta_1'-0|\geq \beta/m$. Consequently, $(\vartheta',\bfa')\in \calS$ and the current ordering of $\vartheta$ and $\vartheta'$ minimizes their maximum distance to each other. Define 
	$$
	\eta:= y(\vartheta,a)-y(\vartheta',a'), 
	$$
	so that trivially, 
	$$
	\tilde y(\vartheta,a,0) = \tilde y(\vartheta',a',\eta). 
	$$
	We proceed to control the $\ell^p$ norm of $\eta$. Since $\vartheta$ and $\vartheta'$ are identical except for $\theta_1$ and $\theta_1'$ (and analogously for $a$ and $a'$), we see that 
	$$
	\eta(x) = a_1 e^{2\pi i\theta_1\cdot x} - a_1' e^{2\pi i \theta_1'\cdot x} = 1 - a_1' e^{2\pi i \delta x_1}.
	$$
	We proceed to estimate the $\ell^p$ norm of $\eta$. Using that $|1-e^{2\pi it}|\leq 2\pi t$ for $t\in \R$ and that $|x_1|\leq 2m$ for $x\in X_\star$, we get
	\begin{align*}
		|\eta(x)| 
		&= |1 - a_1'  e^{2\pi i \delta x_1}| \\
		&\leq |1-a_1'| + |a_1'| |1- e^{2\pi i \delta x_1} | \\
		&\leq m \delta + 2\pi (1+\delta) \delta |x_1| \\
		&\leq m \delta (1+4\pi + 4\pi \delta).
	\end{align*}
	By choice of $\delta$ and by making $c_d$ sufficiently small, we have
	$$
	\|\eta\|_{\ell^p(X_\star)}
	\leq (4m+1)^d \|\eta\|_{\ell^\infty(X_\star)}
	\leq (4m+1)^d  m \delta (1+4\pi + 4\pi \delta)
	\leq \epsilon,
	$$
	which proves that $\eta\in \calN$. By \cref{lem:minmaxhelp}, we have 
	\begin{align*}
		\Theta_\star(\calS,\calN)
		&\geq \frac 12 \max_{\ell=1,\dots,s}|\theta_\ell-\theta_\ell'|
		=\frac 12 |\theta_1-\theta_1'|
		=\frac 12\delta 
		=\frac{c_d \epsilon}{2m^{1+d/p}}, \\
		A_\star(\calS,\calN)
		&\geq \frac 12 \max_{\ell=1,\dots,s} |a_\ell-a_\ell'|
		=\frac 12 |a_1-a_1'|
		=\frac 12 m\delta
		=\frac{c_d \epsilon}{2m^{d/p}}. 
	\end{align*}
	This completes the proof.
	\end{proof}
	
	Comparing \cref{thm:maincube} and \cref{thm:minimaxp}, we see that Gradient-MUSIC is minimax optimal in $\|\eta\|_{\ell^p(X_\star)}$ and $m$ for all $p\in [1,\infty]$.
	
	\begin{theorem}[Minimax lower bound for $X_\star = B_{2m}$]
		\label{thm:minimaxp2}
		Let $s\geq 2$, $\beta\geq 1$, $m\geq 2\beta s$, and $X_\star = B_{2m}$. Consider the signal class
		$$
		\calS:=\big\{(\vartheta,\bfa) \colon |\vartheta|=|\bfa|=s, \, \min_\ell |a_\ell|=1, \, m\Delta(\vartheta)\geq \beta \big\}\subset (\Omega\times \C)^s. 
		$$
		For any $p\in [1,\infty]$ and $\epsilon\leq c_d m^{d/p}$ where $c_d$ only depends on $d$, consider the noise class
		$$
		\calN:= \{\eta\colon \|\eta\|_{L^p(B_{2m})}\leq \epsilon\}.
		$$
		Then we have
		\begin{equation*}
			\Theta_\star(\calS,\calN) \gtrsim_d \frac{\epsilon}{ m^{1+d/p}} \andspace
			A_\star(\calS,\calN)  \gtrsim_d \frac{\epsilon}{m^{d/p}}.
		\end{equation*}
	\end{theorem}
	
	\begin{proof}
		The proof is almost identical to the proof of \cref{thm:minimaxp}. Let $(\vartheta,a)$, $(\vartheta',a')$, $\delta$ and $\eta$ be the same quantities defined in that proof. Following the same steps shows that for all $x\in B_{2m}$, 
		\begin{align*}
			|\eta(x)| 
			&= |1 - a_1'  e^{2\pi i \delta x_1}| \\
			&\leq |1-a_1'| + |a_1'| |1- e^{2\pi i \delta x_1} | \\
			&\leq m \delta + 2\pi (1+\delta) \delta |x_1| \\
			&\leq m \delta (1+4\pi + 4\pi \delta).
		\end{align*}
		By choice of $\delta$ and by making $c_d$ sufficiently small, we have
		$$
		\|\eta\|_{L^p(B_{2m})}
		\leq |B_{2m}|^d \|\eta\|_{L^\infty(B_{2m})}
		\leq |B_1| (2m)^d  m \delta (1+4\pi + 4\pi \delta)
		\leq \epsilon,
		$$
		which proves that $\eta\in \calN$. Using \cref{lem:minmaxhelp} completes the proof.
	\end{proof}
	
	Comparing \cref{thm:mainball} and \cref{thm:minimaxp2}, we see that Gradient-MUSIC is minimax optimal in $\|\eta\|_{L^p(B_{2m})}$ and $m$ for all $p\in [1,\infty]$.

	\section{Geometric analysis: Proof of structural theorem}
	\label{sec:geometricanalysis}

   	The approach  proposed in this paper hinges on a  geometric analysis of the MUSIC function $q:=q_U$ associated with the true subspace $U:=U_\vartheta$ and its perturbed counterpart $\tilde q:=q_{\tilde U}$. The geometric analysis will show how $q$ and $\tilde q$ are influenced by properties of $K$, which itself depends on the sampling set $X_\star$.	
	
	The outline of our analysis is as follows. First we establish  some basic properties, formulas, and inequalities in \cref{sec:geometricformulas}. Afterwards, we start with the simpler task of analyzing the noiseless MUSIC function $q$. We will derive various quantitative estimates for $q$. \cref{lem:approx2} identifies an important parameter $\tau$ such that such that 
	$$
	\nabla^2 q (\omega)\approx -2\nabla^2 K(0)=2\Psi \forallspace \omega\in B_\tau(\theta_\ell).
	$$ 
	Hence, $q$ is strictly convex in $B_\tau(\theta_\ell)$ if $\Psi$ is strictly positive definite and so $\tau$ can be thought of as the proper scale on which properties of $q$ can be deduced from that of $K$. 
	
	For $\omega$ far away from all the $\{\theta_\ell\}_{\ell=1}^s$, we use a different argument. The proof of \cref{lem:approx} provides a global approximation result such that 
	$$
	1-q(\omega) \approx \sum_{\ell=1}^s |K(\omega -\theta_\ell)|^2 \forallspace \omega\in \Omega.
	$$ 
	Under our main assumptions, this is enough to deduce that 
	$$
	q(\omega)\geq C_\vartheta \forallspace \omega\not\in \bigcup_{\ell=1}^s B_\tau(\theta_\ell)
	$$ 
	for a $C_\vartheta>0$ which can be further estimated.

	We are of course more interested in a perturbed MUSIC function $\tilde q$ under an assumption that $\norm{P_U-P_{\tilde U}}$ is small enough. Several properties of $q$ carry over to $\tilde q$ through \cref{lem:bernstein}. We call these Bernstein-type inequalities since they control 
	$$
	\|\partial^\alpha q - \partial^\alpha \tilde q\|_{L^2_\nu}
	$$
	in terms of $K$ and $\|P_U-P_{\tilde U}\|$. 
	
	The main technical difficultly is showing that $\tilde q$ has $s$ local minima $\tilde \vartheta$ that are within $\epsilon$ to $\vartheta$, where $\epsilon$ is potentially much closer than both $\tau$ and $\norm{P_U-P_{\tilde U}}$. This step is carried out in an abstract result, \cref{lem:existenceminima}. 
	
	Combining these ingredients yields \cref{thm:maingeometric} and is done in \cref{proof:geometric2}.
	
	\subsection{Some formulas and inequalities}
	\label{sec:geometricformulas}
	
	Throughout the expository portions, we assume that \cref{assump:symmetry} holds. It is straightforward to verify that $\phi_\omega$ is infinitely differentiable in $\omega$. Due to the fourth moment condition in \cref{assump:symmetry}, we see that 
	$$
	\partial^\alpha_\omega \phi_\omega\in L^2_\nu, \forallspace |\alpha|\leq 2.
	$$
	For all $\alpha,\beta\in \N^d$ such that $\max\{|\alpha|,|\beta|\}\leq 2$, we have
	\begin{equation}
	\begin{split}
		\inner{\partial_{\omega'}^\alpha \phi_{\omega'}, \partial_\omega^\beta \phi_\omega}_{L^2_\nu}
		&= (-1)^{|\alpha|} i^{|\alpha|+|\beta|}  \partial^{\alpha+\beta} K(\omega-\omega'). 
	\end{split} \label{eq:Kgradient}
	\end{equation}
	
	Let us move onto basic properties of $q$. For any $|\alpha|\leq 2$ and bounded linear operator $B\colon L^2_\nu \to L^2_\nu$, since $\partial_\omega^\alpha \phi_\omega\in L^2_\nu$, we have 
	$$
	\partial_\omega^\alpha (B\phi_\omega) 
	= B (\partial_\omega^\alpha \phi_\omega). 
	$$
	Likewise, for any bounded linear functional $\ell\colon L^2_\nu\to \C$, we have
	$$
	\partial_\omega^\alpha (\ell (\phi_\omega)) 
	= \ell (\partial_\omega^\alpha \phi_\omega). 
	$$
	This allows us to switch the order of derivatives in $\omega\in\Omega$ and bounded operators on $L^2_\nu$, which we will use frequently without explicit mention, typically in the case where $B=P_U$, $B=P_{\tilde U}$, or $\ell(\cdot)=\inner{f,\cdot}_{L^2_\nu}$ for some $f\in L^2_\nu$. In particular, this observation implies that for any closed subspace $W\subset L^2_\nu$, the MUSIC function $q_W$ is at least twice differentiable. Some calculations yield
	\begin{align}
		q_W(\omega)
		&=1- \inner{\phi_\omega, P_W \phi_\omega}_{L^2_\nu}, \label{eq:q0} \\
		\partial_j q_W(\omega)
		&=-2\Re \inner{\partial_{\omega_j} \phi_\omega, P_{W} \phi_\omega}_{L^2_\nu}, \label{eq:q1} \\
		\partial_j \partial_k q_{W}(\omega)
		&=-2\Re \inner{\partial_{\omega_j} \partial_{\omega_k} \phi_\omega, P_{W} \phi_\omega}_{L^2_\nu} - 2 \Re \inner{\partial_{\omega_k} \phi_\omega, P_W \partial_{\omega_j}\phi_\omega}_{L^2_\nu}. \label{eq:q2}
	\end{align}
    
	\subsection{Geometric analysis of $q$}
	\label{sec:geometriclemmas}
	
	The first step is to control the eigenvalues of $\nabla^2q(\omega)$ for $\omega$ near $\vartheta$.
	
	\begin{lemma}
	\label{lem:approx2}
	Suppose \cref{assump:symmetry} holds. For any $\vartheta:=\{\theta_\ell\}_{\ell=1}^s$, $\tau\leq \Delta(\vartheta)/2$, and $\omega\in \Omega$ such that $\dist(\omega,\vartheta)\leq \tau$, we have
	\begin{align*}
		\|\nabla^2 q(\omega)-2 \Psi \|_F
		\leq \sqrt{2 \tau \|\nabla K\|_{L^\infty(B_\tau)} \Delta^2 K(0)} + \frac{1}{\lambda_{\min}(K_\vartheta)} \left(\|\nabla K\|_{L^\infty(B_\tau)}^2 + E_1(\vartheta) \right).
	\end{align*}
	\end{lemma}
	
	\begin{proof}
	Recall the shorthand notation $\Psi=-\nabla^2 K(0)$, which is a real symmetric matrix. Starting with formula \eqref{eq:q2} and doing some  further manipulations,
	\begin{equation*}
		\begin{split}
			\frac 12 \partial_j \partial_k q(\omega)
			&=-\Re \inner{\partial_{\omega_j} \partial_{\omega_k} \phi_\omega, \phi_\omega}_{L^2_\nu} + \Re \inner{\partial_{\omega_j} \partial_{\omega_k} \phi_\omega, (I-P_U) \phi_\omega}_{L^2_\nu} -\Re \inner{\partial_{\omega_j} \phi_\omega, P_U \partial_{\omega_k} \phi_\omega}_{L^2_\nu} \\ 
			&=\Psi_{j,k} + \Re \inner{\partial_{\omega_j} \partial_{\omega_k} \phi_\omega, (I-P_U) \phi_\omega}_{L^2_\nu} - \Re \inner{\partial_{\omega_j} \phi_\omega, P_U \partial_{\omega_k}\phi_\omega}_{L^2_\nu}.
		\end{split}
	\end{equation*}
	The first term is dominant, while the second and third terms are error quantities. Define the matrices $A(\omega)$ and $B(\omega)$ such that 
	\begin{align*}
		A_{j,k}(\omega)
		&=\Re \inner{\partial_{\omega_j} \partial_{\omega_k} \phi_\omega, (I-P_U) \phi_\omega}_{L^2_\nu} , \\
		B_{j,k}(\omega)
		&=- \Re \inner{\partial_{\omega_j} \phi_\omega, P_U \partial_{\omega_k}\phi_\omega}_{L^2_\nu}. 
	\end{align*}
	Hence, we have derived the formula
	\begin{equation}\label{eq:hessianq}
		\nabla^2 q (\omega) 
		= 2\Psi + 2A(\omega) + 2B(\omega). 
	\end{equation}		
	From here onward, fix any $\ell\in \{1,\dots,s\}$ and assume that $\omega\in B_\tau(\theta_\ell)$. 
	
	For the first matrix $A(\omega)$, notice that $(I-P_U)\phi_{\theta_\ell}=0$ because $P_U\phi_{\theta_\ell} = \phi_{\theta_\ell}$. For each $j,k\in \{1,\dots,d\}$,  
	\begin{align*}
		|A_{j,k}(\omega)|
		&=\left| \Re\inner{\partial_{\omega_j} \partial_{\omega_k} \phi_\omega,  (I-P_U) (\phi_\omega-\phi_{\theta_\ell})}_{L^2_\nu} \right| \\
		&\leq \norm{\partial_{\omega_j} \partial_{\omega_k} \phi_\omega}_{L^2_\nu}  \norm{\phi_\omega-\phi_{\theta_\ell}}_{L^2_\nu}.
	\end{align*}
	A calculation together with the mean value theorem and that $|\omega-\theta_\ell|\leq \tau$ shows that
	\begin{align*} 
		\norm{\phi_{\omega}-\phi_{\theta_\ell}}_{L^2_\nu}^2 
		&= 2K(0)-2K(\omega-\theta_\ell)
		\leq 2 \tau \|\nabla K\|_{L^\infty(B_\tau)}, \\
		\norm{\partial_{\omega_j} \partial_{\omega_k} \phi_\omega}_{L^2_\nu}^2
		&=\partial_j^2 \partial_k^2 K(0).  
	\end{align*}
	Using the above inequalities, we can conclude that 
	\begin{equation}
		\|A(\omega)\|_F^2
		\leq 2 \tau \|\nabla K\|_{L^\infty(B_\tau)} \sum_{j,k}  \partial_j^2 \partial_k^2 K(0)
		= 2 \tau \|\nabla K\|_{L^\infty(B_\tau)} \Delta^2 K(0).
		\label{eq:hessianq2}
	\end{equation}
	
	For the second matrix $B(\omega)$, recall the synthesis operator $T_\vartheta$ defined in \eqref{eq:synthesis1}. For convenience, let $T=T_\vartheta/\sqrt{\nu(\sampset)}$ be a normalized version and so
	$
	P_U=T (T^*T)^{-1} T^*$. A calculation yields
	\begin{align*}
		|B_{j,k}(\omega)|
		&\leq \left|\inner{\partial_{\omega_j} \phi_\omega, T (T^*T)^{-1} T^*\partial_{\omega_k}\phi_\omega}_{L^2_\nu}\right| \\
		&=\left|\inner{T^*\partial_{\omega_j} \phi_\omega, (T^*T)^{-1} (T^*\partial_{\omega_k} \phi_\omega)}_{\C^s}\right| \\
		&\leq \norm{ (T^*T)^{-1}} |T^*\partial_{\omega_j} \phi_\omega| |T^*\partial_{\omega_k} \phi_\omega|.
	\end{align*}
	Note that $T^*T = K_\vartheta$ and $T^* \partial_{\omega_j} \phi_\omega
	= \begin{bmatrix}
		\partial_j K(\omega-\theta_\ell)
	\end{bmatrix}_{\ell=1,\dots,s}
	$. Recall that $|\omega-\omega_\ell|\leq \tau$ and the definition of $E_1(\vartheta)$ defined in \eqref{eq:E1def}. Then,
	\begin{equation}
		\begin{split}
			\|B(\omega)\|_F
			&\leq \norm{ (T^*T)^{-1}} \left(\sum_{j=1}^d |T^*\partial_{\omega_j} \phi_\omega|^2 \sum_{k=1}^d |T^*\partial_{\omega_k} \phi_\omega|^2\right)^{1/2} \\
			&\leq \frac{1}{\lambda_{\min}(K_\vartheta)} \sum_{j=1}^s |\nabla K(\omega-\theta_j)|^2.
		\end{split}
		\label{eq:hessianq4}
	\end{equation}	
	We next control the summation in \eqref{eq:hessianq4} by breaking it into two terms. For all $j\ne \ell$, recalling that $|\omega-\theta_\ell|\leq \tau$, we have 
	$
	|\omega-\theta_j|
	\geq |\theta_j-\theta_\ell| - |\omega-\theta_\ell|
	\geq \Delta-\tau\geq \Delta/2. 
	$
	By definition of $E_1(\vartheta,\tau)$, this now implies 
	\begin{equation}
		\sum_{j=1}^s |\nabla K(\omega-\theta_j)|^2
		=|\nabla K(\omega-\theta_\ell)|^2 + \sum_{j\ne\ell} |\nabla K(\omega-\theta_j)|^2
		\leq \|\nabla K\|_{L^\infty(B_\tau)}^2 + E_1(\vartheta). \label{eq:hessianq5}
	\end{equation}
	Combining \eqref{eq:hessianq}, \eqref{eq:hessianq2}, \eqref{eq:hessianq4}, and \eqref{eq:hessianq5} completes the proof.
	\end{proof}
	
	While the previous lemma controlled the behavior of $q$ near $\vartheta$, we need to control $q$ on the complement set.

	\begin{lemma}\label{lem:approx}
	Suppose \cref{assump:symmetry} holds. For any $\vartheta=\{\theta_\ell\}_{\ell=1}^s$, $\tau\leq \Delta(\vartheta)/2$, and $\omega\in \domain$ such that $\dist(\omega,\vartheta)\geq \tau$, we have
	$$
	q(\omega) \geq 1 - \|K\|_{L^\infty(B_\tau^c)}^2 - E_0(\vartheta) - \lambda_{\max}(K_\vartheta) \max\left\{ \Big| \frac 1 {\lambda_{\max}(K_\vartheta)}-1\Big|,\, \Big|\frac 1 {\lambda_{\min}(K_\vartheta)} - 1\Big| \right\}.
	$$
	\end{lemma}
	
	\begin{proof}
	Recall the synthesis operator $T_\vartheta$ defined in \eqref{eq:synthesis1}, and for convenience, let $T=T_\vartheta/\sqrt{\nu(\sampset)}$ be a normalized version. Note that $P_U= T(T^*T)^{-1} T^*$ and $T^* \phi_\omega
	= \begin{bmatrix}
		K(\omega-\theta_\ell)
	\end{bmatrix}_{\ell=1,\dots,s}
	$. A calculation together with the definition of $E_0(\vartheta)$ in \eqref{eq:E0def} shows that
	\begin{equation*}
		\begin{split}
			q(\omega)&=1- \inner{\phi_\omega, P_U \phi_\omega}_{L^2_\nu} \\
			&=1- \inner{\phi_\omega, TT^* \phi_\omega}_{L^2_\nu} - \inner{\phi_\omega,  (P_U-TT^*) \phi_\omega}_{L^2_\nu} \\
			&=1 - \inner{T^*\phi_\omega, T^* \phi_\omega}_{\C^s} - \inner{T^* \phi_\omega, \left( (T^*T)^{-1}-I \right) T^*\phi_\omega}_{\C^s} \\
			&\geq 1-\sum_{\ell=1}^s |K(\omega-\theta_\ell)|^2 - \norm{I-(T^*T)^{-1}} | T^*\phi_\omega|^2.
		\end{split}
	\end{equation*}
	It remains to control the right hand side term. 
	
	For the second term, note that $T^*T=K_\vartheta$. Then
	\begin{align*}
		| T^*\phi_\omega|^2
		&\leq \norm{T^*}^2 \norm{\phi_\omega}_{L^2_\nu}^2
		\leq \lambda_{\max}(K_\vartheta), \\
		\norm{I-(T^*T)^{-1}} 
		&\leq \max\left\{ \Big| \frac 1 {\lambda_{\max}(K_\vartheta)}-1\Big|,\, \Big|\frac 1 {\lambda_{\min}(K_\vartheta)} - 1\Big| \right\}.
	\end{align*}
	
	For the summation, we first note that either there is a $\theta_j$ such that $|\omega-\theta_j|<\Delta/2$ or no such $\theta_j$ exists. Suppose there is a $\theta_j$. Then it is unique because for all $\ell\ne j$, we have $|\omega-\theta_\ell|\geq |\theta_j-\theta_\ell|- |\omega-\theta_j|\geq \Delta/2$. Using the definition of $E_0(\vartheta)$ and the assumption that $|\omega-\theta_j|\geq \tau$, we get
	$$
	\sum_{\ell=1}^s |K(\omega-\theta_\ell)|^2
	= |K(\omega-\theta_j)|^2 + \sum_{\ell\ne j} |K(\omega-\theta_\ell)|^2
	\leq \|K\|_{L^\infty(B_\tau^c)}^2 + E_0(\vartheta). 
	$$
	This settles the first case. In the second case where no $\theta_j$ exists, then $\dist(\omega,\vartheta)\geq \Delta/2$ and we immediately get
	$$
	\sum_{\ell=1}^s |K(\omega-\theta_\ell)|^2
	\leq E_0(\vartheta). 
	$$
	The upper bound for the first case is larger. Combining the above completes the proof. 
	\end{proof}
	
	\subsection{Geometric analysis of $\tilde q$}
	
	In the following, we develop Bernstein-type inequalities for the difference between two MUSIC functions and their derivatives.
	
	\begin{lemma}
		\label{lem:bernstein}
		Suppose Assumption \ref{assump:symmetry} holds. For any finite dimensional subspaces $U,\tilde U\subset L^2_\nu(\sampset)$ of equal dimension, their associated MUSIC functions $q=q_U$ and $\tilde q=q_{\tilde U}$ satisfy the following inequalities. For all $\omega\in \domain$,
		\begin{align}
			|q(\omega)-\tilde q(\omega)|
			&\leq \norm{P_U-P_{\tilde U}}, \label{eq:bernstein0} \\
			|\nabla q(\omega)-\nabla \tilde q(\omega)|
			&\leq \sqrt{\trace(\Psi)} \, \|P_{\tilde U} -P_U\|, \label{eq:bernstein1} \\
			\norm{\nabla^2 q(\omega)-\nabla^2 \tilde q(\omega)}_F
			&\leq 2 \sqrt{ \Delta^2 K(0) + (\Delta K(0))^2} \norm{P_U-P_{\tilde U}}. \label{eq:bernstein2} 
		\end{align}
	\end{lemma}
	
	\begin{proof}
		The first inequality of this lemma follows immediately from using \eqref{eq:q0} and 
		\begin{align*}
			|q(\omega)-\tilde q(\omega)|
			=\left| \inner{\phi_\omega, \left(P_U-P_{\tilde U}\right) \phi_\omega}_{L^2_\nu} \right|
			\leq \norm{P_U-P_{\tilde U}}.
		\end{align*}
		For the second inequality of this lemma, we use formulas \eqref{eq:q1} and \eqref{eq:Kgradient} to get 
		\begin{align*}
			|\nabla q(\omega)-\nabla \tilde q(\omega)|
			&= \left| \inner{\nabla \phi_\omega, (P_U-P_{\tilde U}) \phi_\omega}_{L^2_\nu} \right|
			\leq \|\nabla \phi_\omega\|_{L^2_\nu} \, \norm{P_{\tilde U} -P_U} 
			= \sqrt{\trace(\Psi)} \, \norm{P_{\tilde U} -P_U}.  
		\end{align*}
		For the final inequality, we use formulas \eqref{eq:q2} and \eqref{eq:Kgradient} to see that
		\begin{align*}
			|\partial_j\partial_k q(\omega) - \partial_j\partial_k \tilde q(\omega)|^2
			&\leq 4 \norm{\partial_{\omega_j} \partial_{\omega_k} \phi_\omega}_{L^2_\nu}^2 \norm{P_U-P_{\tilde U}}^2 + 4 \norm{\partial_{\omega_j} \phi_\omega}_{L^2_\nu}^2 \norm{\partial_{\omega_k} \phi_\omega}_{L^2_\nu}^2 \norm{P_U-P_{\tilde U}}^2 \\
			&= 4 \left( \partial_j^2 \partial_k^2 K(0) + \, \partial_j^2 K(0) \partial_k^2 K(0) \right) \norm{P_U-P_{\tilde U}}^2.
		\end{align*}
		This now yields
		\begin{align*}
			\norm{\nabla^2 q(\omega)-\nabla^2 \tilde q(\omega)}_F^2
			&\leq 4 \left( \Delta^2 K(0) + (\Delta K(0))^2 \right) \norm{P_U-P_{\tilde U}}^2.
		\end{align*}
		This completes the lemma.
	\end{proof}
	
	The following is an abstract lemma which shows that under an assumption on the Hessian of $q$ near $\{\theta_\ell\}_{\ell=1}^s$, the perturbed MUSIC function $\tilde q$ has critical points nearby the true ones. 
	
	\begin{lemma}\label{lem:existenceminima}
		Suppose \cref{assump:symmetry} holds and let $\vartheta=\{\theta_\ell\}_{\ell=1}^s \subset \Omega$. Assume that there are $\tilde U$ and $\epsilon >0$ such that 
		\begin{equation}\label{eq:eigcond}
			\inf_{\omega \colon \dist(\omega,\vartheta)\leq \epsilon} \, \lambda_d \left(\nabla^2 q(\omega)\right) 
			> \max\left\{ \frac {2}\epsilon \, \sqrt{\trace(\Psi)}, \, 2 \sqrt{ \Delta^2 K(0) + (\Delta K(0))^2} \right\} \norm{P_U-P_{\tilde U}}.
		\end{equation}
		Then there exist $\{\tilde \theta_\ell\}_{\ell=1}^s \subset \domain$ such that $\tilde \theta_\ell \in B_\epsilon(\theta_\ell)$, it is a local minimum of $\tilde q:=q_{\tilde U}$, and it is the only critical point of $\tilde q$ in $B_\epsilon(\theta_\ell)$.
	\end{lemma}
	
	\begin{proof}
		Let $q:=q_U$ be the MUSIC function corresponding to the true subspace associated with $\vartheta$. Fix $\ell$, consider the ball $B_\epsilon(\theta_\ell)$ where $\epsilon >0$ will be chosen later, and let $\omega\in \partial B(\theta_\ell,\epsilon)$. By the Taylor remainder formula, there is a $\zeta_\ell\in B(\theta_\ell,\epsilon)$ such that  
		\begin{align*}
			0 = q(\theta_\ell)=q(\omega)+ \nabla q(\omega) \cdot (\theta_\ell-\omega) + \frac 12 (\theta_\ell-\omega) \cdot \nabla^2 q(\zeta_\ell) \, (\theta_\ell-\omega). 
		\end{align*}
		Using this formula, we see that
		\begin{align*}
			\nabla \tilde q(\omega) \cdot (\omega-\theta_\ell) 
			&= \nabla q(\omega) \cdot (\omega-\theta_\ell) + (\nabla \tilde q(\omega)-\nabla q(\omega)) \cdot (\omega-\theta_\ell) \\
			&= q(\omega) + \frac 12 (\theta_\ell-\omega) \cdot \nabla^2 q(\zeta_\ell) \, (\theta_\ell-\omega) + (\nabla \tilde q(\omega)-\nabla q(\omega)) \cdot (\omega-\theta_\ell). 
		\end{align*}
		
		We proceed to show that the right hand side is nonnegative for all $\omega\in\partial B_\epsilon(\theta_\ell)$. To do this, we use that $q$ is nonnegative and \eqref{eq:bernstein1} to get
		\begin{align*}
			\nabla \tilde q(\omega) \cdot (\omega-\theta_\ell) 
			&\geq \frac 12 (\theta_\ell-\omega) \cdot \nabla^2 q(\zeta_\ell) \, (\theta_\ell-\omega) + (\nabla \tilde q(\omega)-\nabla q(\omega)) \cdot (\omega-\theta_\ell) \\
			&\geq \frac 12 \, \lambda_d \left(\nabla^2 q(\zeta_\ell)\right) |\theta_\ell-\omega|^2 - |\nabla \tilde q(\omega)-\nabla q(\omega)| |\omega-\theta_\ell| \\
			&\geq \frac 12 \, \lambda_d\left(\nabla^2 q(\zeta_\ell)\right) \, \epsilon^2 - \sqrt{\trace(\Psi)} \, \|P_{\tilde U} -P_U\| \, \epsilon. 
		\end{align*}
		Since $\zeta_\ell\in B_\epsilon(\theta_\ell)$, we use  assumption \eqref{eq:eigcond} and the previous displayed inequality to see that 
		$$
		\nabla \tilde q(\omega) \cdot (\omega-\theta_\ell) \geq 0.
		$$
		This inequality holds for all $\omega\in \partial B_\epsilon(\theta_\ell)$, so by \cite[Theorem 1]{morales2002bolzano}, there is a $\tilde \theta_\ell\in B_\epsilon(\theta_\ell)$ such that $\nabla \tilde q(\tilde \theta_\ell)=0$. This proves the existence of $\tilde \theta_\ell$, and it remains to prove the remaining properties. By Weyl's inequality, \eqref{eq:bernstein2}, and assumption \eqref{eq:eigcond}, we now see that 
		\begin{align*}
			\lambda_d \left(\nabla^2 \tilde q(\omega)\right)
			&\geq \lambda_d\left(\nabla^2 q(\omega)\right)-\norm{\nabla^2 q(\omega)-\nabla^2 \tilde q(\omega)} \\
			&\geq \lambda_d\left(\nabla^2 q(\omega)\right) - 2 \sqrt{ \Delta^2 K(0) + (\Delta K(0))^2} \, \norm{P_U-P_{\tilde U}}
			>0. 
		\end{align*}
		Thus, $\nabla^2\tilde q$ is strictly positive definite on $B_\epsilon(\theta_\ell)$. Since $\tilde \theta_\ell \in B_\epsilon(\theta_\ell)$ is a critical point of $\tilde q$, we now deduce that $\tilde \theta_\ell$ is a local minimum of $\tilde q$ and it is the only critical point of $\tilde q$ in $B_\epsilon(\theta_\ell)$.  
	\end{proof}

  	\subsection{Completion of proof of Theorem \ref{thm:maingeometric}}
		\label{proof:geometric2}
	
	Let $s\geq 1$ and $\vartheta=\{\theta_\ell\}_{\ell=1}^s$ be arbitrary. We first prove that $\tilde q:=q_{\tilde U}$ is an admissible optimization landscape for $\vartheta$ with parameters \eqref{eq:mainparameters}, which involves proving four statements. 
	
	For now, fix any $k\in \{1,\dots,s\}$ and $\omega \in B_\tau(\theta_k)$. We first show that $\tilde q$ is strictly convex in this ball. Note that \cref{lem:approx2} and assumption \eqref{eq:betacondition1} imply that
	\begin{equation}
		\|\nabla^2 q(\omega) - 2\Psi \| 
		\leq \delta_1 \lambda_d(\Psi). \label{eq:qhessianhelp2}
	\end{equation}
	Additionally, by \eqref{eq:bernstein2} and assumption \eqref{eq:noisecondition}, we have
	\begin{align}\label{eq:qhessianhelp1}
		\norm{\nabla^2 q(\omega)-\nabla^2 \tilde q(\omega)} 
		&\leq 2 \sqrt{ \Delta^2 K(0) + (\Delta K(0))^2} \, \norm{P_U-P_{\tilde U}}
		\leq \delta_2 \lambda_d(\Psi). 
	\end{align}
	By Weyl's inequality, \eqref{eq:qhessianhelp1}, and \eqref{eq:qhessianhelp2}, we see that
	\begin{align*}
		\lambda_d \left( \nabla^2 \tilde q(\omega)\right)
		\geq 2\lambda_d(\Psi) - \|\nabla^2 q(\omega) - 2\Psi \| - \norm{\nabla^2 q(\omega)-\nabla^2 \tilde q(\omega)}
		\geq (2-\delta_1-\delta_2) \lambda_d(\Psi).  
	\end{align*}
	This proves that $\tilde q$ is strictly convex in $B_\tau(\theta_k)$, and hence, the second entry in \eqref{eq:mainparameters}. 

    The next step is to prove the existence of a critical point $\tilde \theta_k \in B_\epsilon(\theta_k)$, where
	$$
	\epsilon := \frac {2\sqrt{\trace(\Psi)}} {\delta_2 \lambda_d(\Psi)} \, \norm{P_U-P_{\tilde U}}.
	$$
	We will verify that condition \eqref{eq:eigcond} of \cref{lem:existenceminima} holds for this choice of $\epsilon$. Using the definition of $\epsilon$, assumption \eqref{eq:noisecondition}, that $\delta_1+\delta_2<2$, and \eqref{eq:qhessianhelp2}, we get 
	\begin{align*}
		&\max\left\{ \frac 2 \epsilon \sqrt{\trace(\Psi)}, \, 2 \sqrt{ \Delta^2 K(0) + (\Delta K(0))^2} \right\} \norm{P_U-P_{\tilde U}} \\
		&\qquad \leq \delta_2 \lambda_d(\Psi)
		< (2-\delta_1) \lambda_d(\Psi)
		\leq \lambda_d\left( \nabla^2 q(\omega)\right). 
	\end{align*}
	This inequality holds for all $\omega\in B_\epsilon(\theta_k)$ and for all $k\in \{1,\dots,s\}$, which shows that \cref{lem:existenceminima} is fulfilled, so we conclude that $\tilde q$ has a local minimum at some $\tilde \theta_k\in B_\epsilon(\theta_k)$. Note that by \eqref{eq:noisecondition}, we have $\epsilon \leq \tau$. This establishes the first entry in \eqref{eq:mainparameters}. 

    We use that $I-P_{\tilde U}$ is a projection operator and $\phi_{\theta_k}=P_U \phi_{\theta_k}$, to see that
	\begin{align*}
		\tilde q(\theta_k)
		&= \left\|\left(I-P_{\tilde U}\right) \phi_{\theta_k} \right\|_{L^2_\nu}^2 \\
		&= \left\| \left(I-P_{\tilde U}\right) P_U \phi_{\theta_k} \right\|_{L^2_\nu}^2
		= \left\|\left(P_U-P_{\tilde U}\right) P_U \phi_{\theta_k} \right\|_{L^2_\nu}^2 
		\leq \left\| P_U-P_{\tilde U} \right\|^2. 
	\end{align*}
	Earlier in this proof, we showed that $\tilde \theta_k\in B_\epsilon(\theta_k)$ and it is the only local minimum of $\tilde q$ in this ball. Hence,
	$$
	\tilde q(\tilde \theta_k)
	\leq \tilde q(\theta_k)
	\leq \left\| P_U-P_{\tilde U} \right\|^2.
	$$
	This proves the third entry in \eqref{eq:mainparameters}. 
	
    Now, suppose instead $\omega\in \domain$ such that $\dist(\omega,\vartheta)\geq \tau$. To begin with, note that \cref{lem:approx} and assumption \eqref{eq:betacondition2} imply 
    \begin{align*}
    	q(\omega)
    	&\geq \frac 58 \left( 1-\|K\|_{L^\infty(B_\tau^c)}^2 \right). 
    \end{align*}
   	Thus, by \eqref{eq:bernstein0} and assumption \eqref{eq:noisecondition}, we get
	\begin{align*}
		\tilde q(\omega)
		&\geq q(\omega)- |q(\omega)-\tilde q(\omega)| 
		\geq q(\omega) - \|P_U-P_{\tilde U}\|
		\geq \left(\frac 58-\frac{\delta_2}{4d}\right) \left( 1-\|K\|_{L^\infty(B_\tau^c)}^2 \right). 
	\end{align*}
	Thus, we have established the fourth and final entry of \eqref{eq:mainparameters}. 
	
	We next prove \eqref{eq:qtilde1}. We use \eqref{eq:q1} and \eqref{eq:Kgradient} to see that
	\begin{align*}
		\|\nabla \tilde q\|_{L^\infty(\Omega)}
		&\leq 2\|\nabla \phi_\omega\|_{L^2_\nu}
		= 2\sqrt{-\Delta K(0)}
		= 2\sqrt{\trace(\Psi)},
	\end{align*}
	as desired. 
	
	For \eqref{eq:qtilde2}, we already proved the lower bound for $\lambda_d \left( \nabla^2 \tilde q(\omega)\right)$. For the remaining upper bound, we use Weyl's inequality, \eqref{eq:qhessianhelp1}, and \eqref{eq:qhessianhelp2}, to get
	\begin{equation*}
		\lambda_1 \left( \nabla^2 \tilde q(\omega)\right)
		\leq 2\lambda_1(\Psi) + \|\nabla^2 q(\omega) - 2\Psi\| + \norm{\nabla^2 q(\omega)-\nabla^2 \tilde q(x)}
		\leq (2+\delta_1+\delta_2)\lambda_1(\Psi). 
	\end{equation*}
	This completes the proof of \eqref{eq:qtilde2}.

    \section{Numerics}
    \label{sec:numerics} 
    
    In this simulation, we verify \cref{thm:maincube} part (b) for dimension $d=2$ and for three different examples of radially varying nonstationary Gaussian noise. Recall that $\eta = \sigma g$ where the entries of $g$ are $\calN(0,1)$ distributed while $\sigma$ is the nonstationary component of $\eta$. Suppose  
    $$
    \sigma(x)=\sigma (1+|x|)^r \quad r\in \R.
    $$
    \begin{enumerate}[1.]
    	\item 
    	(Constant $r=0$) Suppose $\sigma$ is constant, hence the values of $\eta$ on $Q_{2m}\cap\Z^d$ are i.i.d. $\calN(0,\sigma^2)$ random variables. In this case we have $\|\sigma\|_{\ell^2(Q_{2m}\cap\Z^d)}\asymp \sigma m$, and theorem asserts that
    	$$
    	\max_{\ell=1,\dots,s} \big|\theta_\ell - \theta_\ell^\sharp \big|
    	\lesssim_d  \frac{\sigma \sqrt{\log(m)}}{m^2}.
    	$$
    	\item 
    	(Increasing $r>0$) Suppose $\sigma(x)=\sigma (1+|x|)^{1/2}$ so that $r=1/2$. A calculation shows that $\|\sigma\|_{\ell^2(Q_{2m}\cap\Z^d)}\asymp \sigma m^{3/2}$ and theorem tells us that
    	$$
    	\max_{\ell=1,\dots,s} \big|\theta_\ell - \theta_\ell^\sharp \big|
    	\lesssim_d \frac{\sigma\sqrt{\log(m)}}{m^{3/2}}.
    	$$
    	\item 
    	(Decreasing $r<0$) Suppose $\sigma(x)=\sigma (1+|x|)^{-1/2}$ so that $r=-1/2$. Then $\|\sigma\|_{\ell^2(Q_{2m}\cap\Z^d)}\lesssim \sigma \sqrt m$ and the theorem concludes that
    	$$
    	\max_{\ell=1,\dots,s} \big|\theta_\ell - \theta_\ell^\sharp \big|
    	\lesssim_d  \frac{\sigma \sqrt{\log(m)}}{m^{5/2}}.
    	$$
    \end{enumerate}
    
    \begin{figure}[ht]
    	\centering
    	\includegraphics[width=0.45\textwidth]{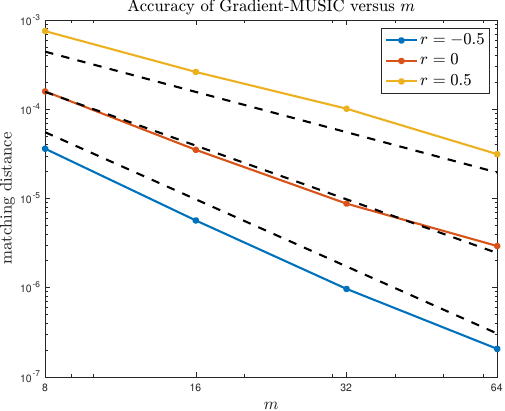}
    	\caption{Accuracy of Gradient-MUSIC versus $m$ for three types of nonstationary independent Gaussian noise with growth parameter $r\in \{-1/2,0,1/2\}$. The 90-th percentile error is recorded over 10 realizations of $\eta$ per parameter pair $m$ and $r$. The dashed black lines are $0.01 \cdot m^\alpha$ for $\alpha \in\{-3/2, -2, -5/2\}$. }
    	\label{fig:exper1}
    \end{figure}
    
    To set up the experiment, we fix an arbitrary $\vartheta=\{\theta_\ell\}_{\ell=1}^s$ with separation at least $1/8$, similar to Figure \ref{fig:example} and let $a_\ell=1$ for each $\ell$. For $m\in \{8,16,32,64\}$, and each trial, we generate $\eta=\sigma g$, where $\sigma$ is one of the three choices listed above. We compute the matching distance between the recovered and true signal frequencies. For each choice of $m$ and $r$, we repeat this scenario 10 times and report the 90-th percentile error. The results are shown in Figure \cref{fig:exper1}. The experiment indicates that the theoretical upper bounds in  \cref{thm:maincube} part (b) are sharp in $m$.
    
    \subsection{Computational complexity}
\label{sec:complexity1}

Concrete parameter choices for Gradient-MUSIC are given in \cref{thm:maincube,thm:mainball}. Among these, the most significant from a computational point of view are the mesh norm $\mesh(G)$ of the initialization grid and the number of gradient iterations $n$.

Suppose first that $G$ is a uniform grid satisfying
\[
\mesh(G)\asymp_d m^{-1}.
\]
Then $G$ has cardinality $O_d(m^d)$. Since the MUSIC function $\tilde q$ can be evaluated on such a grid using the fast Fourier transform (FFT), the cost of the initialization step is
\[
O_d\bigl(s\,m^d\log m\bigr).
\]
Once suitable representatives have been selected from the thresholded clusters, the local refinement stage is comparatively inexpensive. Indeed, gradient descent requires only $O(\log(1/\epsilon))$ iterations per cluster in order to achieve numerical accuracy $\epsilon/m$. Moreover, these $s$ local descents are independent and can therefore be performed in parallel. The resulting cost of the refinement stage is
\[
O_d\bigl(s\,m^d\log(1/\epsilon)\bigr).
\]

Thus the overall complexity of Gradient-MUSIC scales essentially like
\[
O_d\bigl(s\,m^d(\log m+\log(1/\epsilon))\bigr),
\]
up to constants depending only on the dimension.

It is instructive to compare this with the complexity of classical MUSIC. Since classical MUSIC proceeds by brute-force search over the parameter domain, achieving numerical accuracy $\epsilon/m$ requires evaluating $\tilde q$ on a grid $G$ with
\[
\mesh(G)\le \epsilon/m.
\]
If one again uses a uniform grid, then $G$ has cardinality
\[
O_d(m^d\epsilon^{-d}),
\]
and the corresponding FFT-based evaluation cost is
\[
O_d\bigl(m^d\epsilon^{-d}\log(m/\epsilon)\bigr).
\]
This is substantially worse than the cost of Gradient-MUSIC, especially in high dimension. The key point is that, for Gradient-MUSIC, the coarse initialization grid is tied only to the aperture scale $m^{-1}$ and not to the final numerical accuracy. The refinement to accuracy $\epsilon/m$ is then achieved by local descent rather than by global grid refinement. Consequently, the computational gap between Gradient-MUSIC and classical MUSIC widens rapidly as the ambient dimension $d$ increases.

\section{Conclusion and discussion}
\label{sec:conclusion}

The present work has been concerned with a basic question in super-resolution: how much of the measurement data is truly essential for stable and computable parameter recovery, and how should one formulate an optimization principle that uses precisely this information and no more. Our answer is that the decisive object is the signal subspace. In the noiseless setting, the data lie in the span of the steering vectors associated with the true configuration; in the noisy setting, this subspace is perturbed, but it remains the central geometric carrier of information. The MUSIC functional is therefore not introduced here merely as a convenient algorithmic device. Rather, it appears as the most economical nonconvex functional that directly encodes the subspace structure relevant to the inverse problem.

From this point of view, the main theorem, \cref{thm:maingeometric}, should be understood as a structural result about the geometry of reconstruction. It identifies, in an abstract and model-independent way, how perturbation of the signal subspace by deterministic, including adversarial, noise translates into deformation of the MUSIC landscape. At the same time, the theorem is constructive in a strong sense. Through the notion of an \emph{admissible landscape}, it yields a global geometric criterion under which a simple computational procedure---multiple initialization followed by gradient descent---provably succeeds. Thus the theorem does not merely assert that the desired parameters remain identifiable under noise; it explains why they remain computationally accessible.

In this respect, the present analysis sits at the intersection of several traditions. Classical resolution theory, beginning with optical and spectral considerations, is primarily concerned with the distinguishability of nearby objects and often formulates this question through aperture heuristics or peak-width criteria. The inverse problems literature emphasizes stability under perturbation, while the modern optimization literature focuses on the tractability of nonconvex objectives. One of the purposes of this paper has been to show that, in the super-resolution setting, these three viewpoints are not independent. The geometry of the subspace determines the geometry of the optimization landscape, and this geometry in turn governs both stability and algorithmic recovery.

A traditional heuristic expresses the resolution scale $\rho$ in the form
\begin{equation}
\label{eq:folklore}
\rho := \hbox{Resolution Scale} \sim \big(\hbox{Aperture Size}\big)^{-1}
\end{equation}
where the aperture size is determined by the geometry of the sampling set and is, roughly speaking, proportional to its diameter. 
This principle is certainly not without content, and indeed part of our analysis recovers it. In the framework of \cref{thm:maingeometric}, the relevant local scale is governed by the parameter $\tau$, and the reciprocal-aperture law corresponds essentially to the behavior of the local quantity $\|\nabla K\|_{L^\infty(B_\tau)}$, which typically grows in proportion to the diameter of the sampling set. In that sense, the classical heuristic captures the first-order local mechanism by which increased sampling extent improves resolution.

What it does not capture, however, is that super-resolution is not determined solely by local geometry. The admissibility conditions involve, in addition, the quantities $E_0(\vartheta)$ and $E_1(\vartheta)$ associated with the kernel matrix $K_\vartheta$, and these depend on the global arrangement of the object configuration. They measure, in effect, how the configuration interacts with itself through the sampling geometry, and hence they encode collective interference phenomena that are invisible to any purely local aperture principle. This distinction is particularly transparent in the case of sampling on a circle of radius $m$ in two dimensions, where
\[
K(\xi)=J_0(2\pi m|\xi|)
\]
and
\[
\nabla K(\xi)=-2\pi m J_1(2\pi m|\xi|)\frac{\xi}{|\xi|}
\sim m
\qquad \hbox{for sufficiently small } |\xi|.
\]
Here the local scaling agrees with the reciprocal-diameter picture. But since
\[
J_0(z),\,J_1(z)\sim z^{-1/2}, \qquad z\gg 1,
\]
neither function is square integrable on $\R^2$, and consequently the global quantities $E_0(\vartheta)$ and $E_1(\vartheta)$ remain sensitive to the overall spatial spread of the configuration. Widely separated objects continue to interact through the slow decay of the kernel. Thus sparse sampling on a curve exhibits a genuinely global interference effect that is absent from the standard heuristic. By contrast, when the sampling set is filled in, as for the full disk, these effects become controllable; see Example 2 and \cref{thm:mainball}. In this way, the present theory may be viewed as a refinement of the classical aperture principle: local resolution and global interference are distinct mechanisms, and both must be accounted for in a satisfactory theory.

The resulting picture is by now quite clear. For sampling sets given by a discrete cube or a continuous ball, \cref{thm:maincube} and \cref{thm:mainball} show informally that Gradient-MUSIC obeys the bound
\begin{equation}
\label{eq:optimal}
\hbox{(Location) Error} \lesssim \hbox{Resolution Scale}\times \hbox{Noise Level}
\end{equation}
under adversarial noise, where the noise level is measured in the sup norm over the sampling set. This is the natural worst-case law. Under Gaussian stationary noise, however, the estimate improves to
\begin{equation}
\label{eq:supres}
\hbox{(Location) Error} \lesssim \big(\hbox{Resolution Scale}\big)^{1+d/2}\times \hbox{Noise Level},
\end{equation}
up to a logarithmic factor.

Equally important is the fact that these results are uniform and nonasymptotic. They do not rely on a perturbative regime tied to a special configuration, nor do they require asymptotic separation assumptions depending on the particular arrangement of objects. Rather, they apply whenever the separation is at least a sufficiently large constant multiple of resolution scale, independently of the detailed geometry of the configuration. This uniformity is one of the main reasons to regard the theory as a candidate for a genuine minimax description of super-resolution. Indeed, the results strongly suggest that Gradient-MUSIC attains the optimal minimax scaling law in multiple dimensions, and this has recently been proved in the one-dimensional case \cite{fannjiang2025optimality}.

It is worth emphasizing that the present paper has not pursued quantitative sharpness of constants. This is a deliberate choice. In our earlier one-dimensional work \cite{fannjiang2025optimality}, a substantially more delicate geometric analysis yielded explicit constants, most notably the constant $4$ in the separation condition $\Delta(\vartheta)\ge 4/m$. Such refinements are important when one seeks near-optimal thresholds, but they require technical estimates whose role is primarily quantitative rather than structural. Here we have preferred to isolate the general mechanism: the passage from subspace perturbation to global landscape geometry, and from there to stability and computability. To pursue sharp constants in the full multidimensional setting would be an interesting problem in its own right, but it would belong to a different level of analysis.

There are several directions in which the present framework should develop further. One is to extend the class of admissible sampling geometries beyond the cube and the ball, and to understand more systematically which geometric features of the sampling domain suppress or amplify long-range interference. Another is to sharpen the connection with minimax lower bounds in higher dimensions, thereby completing the optimality theory suggested here. A third is to bring the present structural viewpoint into closer contact with related subspace methods, such as ESPRIT and matrix-pencil type algorithms, for which one may hope to formulate analogous landscape principles. More broadly, it would be of interest to understand whether the notion of admissibility has a counterpart in other inverse problems where the physically meaningful information is carried by a low-dimensional subspace but the natural reconstruction principles are nonconvex.
	 
	\section*{Acknowledgments}
	
	WL is partially supported
	by NSF-DMS Award \#2309602 and a PSC-CUNY Cycle 56 Grant. Both authors gratefully thank Wenjing Liao for insightful discussions.

\bibliographystyle{plain}
\bibliography{MUSICmultibib}

@book{born_wolf_1999, 
author    = {Born, Max and Wolf, Emil}, 
title     = {Principles of Optics}, 
publisher = {Cambridge University Press}, 
year      = {1999}, 
address   = {Cambridge}
}

@inproceedings{ding2024esprit,
  title={The {ESPRIT} algorithm under high noise: {Optimal} error scaling and noisy super-resolution},
  author={Ding, Zhiyan and Epperly, Ethan N. and Lin, Lin and Zhang, Ruizhe},
  booktitle={2024 IEEE 65th Annual Symposium on Foundations of Computer Science (FOCS)},
  pages={2344--2366},
  year={2024},
  organization={IEEE}
}

@article{wilber2022data,
	title={Data-driven algorithms for signal processing with trigonometric rational functions},
	author={Wilber, Heather and Damle, Anil and Townsend, Alex},
	journal={SIAM Journal on Scientific Computing},
	volume={44},
	number={3},
	pages={C185--C209},
	year={2022},
	publisher={SIAM}
}

@article{stoica2002maximum,
	title={Maximum likelihood estimation of the parameters of multiple sinusoids from noisy measurements},
	author={Stoica, Petre and Moses, Randolph L. and Friedlander, Benjamin and Soderstrom, Torsten},
	journal={IEEE Transactions on Acoustics, Speech, and Signal Processing},
	volume={37},
	number={3},
	pages={378--392},
	year={2002},
	publisher={IEEE}
}

@article{rao2002performance,
	title={Performance analysis of root-{MUSIC}},
	author={Rao, Bhaskar D. and Hari, K.V.S.},
	journal={IEEE Transactions on Acoustics, Speech, and Signal Processing},
	volume={37},
	number={12},
	pages={1939--1949},
	year={2002},
	publisher={IEEE}
}

@article{derevianko2025parameter,
	title={Parameter estimation for multivariate exponential sums via iterative rational approximation},
	author={Derevianko, Nadiia and H{\"u}bner, Lennart Aljoscha},
	journal={arXiv preprint arXiv:2504.19157},
	year={2025}
}

@article{traonmilin2020basins,
	title={The basins of attraction of the global minimizers of the non-convex sparse spike estimation problem},
	author={Traonmilin, Yann and Aujol, Jean-Fran{\c{c}}ois},
	journal={Inverse Problems},
	volume={36},
	number={4},
	pages={045003},
	year={2020},
	publisher={IOP Publishing}
}

@article{sauer2017prony,
	title={Prony’s method in several variables},
	author={Sauer, Tomas},
	journal={Numerische Mathematik},
	volume={136},
	number={2},
	pages={411--438},
	year={2017},
	publisher={Springer}
}

@article{kunis2016prony,
	title={A multivariate generalization of {Prony}'s method},
	author={Kunis, Stefan and Peter, Thomas and R{\"o}mer, Tim and von der Ohe, Ulrich},
	journal={Linear Algebra and its Applications},
	volume={490},
	pages={31--47},
	year={2016},
	publisher={Elsevier}
}

@article{li2025nonharmonic,
	title={Nonharmonic multivariate {Fourier} transforms and matrices: condition numbers and hyperplane geometry},
	author={Li, Weilin},
	journal={Applied and Computational Harmonic Analysis},
	year={2025}
}

@article{stoica1989music,
	title={{MUSIC}, maximum likelihood, and {Cramer-Rao} bound},
	author={Stoica, P. and Nehorai, A.},
	journal={IEEE Transactions on Acoustics, Speech, and Signal Processing},
	volume={37},
	number={5},
	pages={720--741},
	year={1989}
}

@article{mhaskar2025robust,
	title={Robust and tractable multidimensional exponential analysis},
	author={Mhaskar, Hrushikesh N. and Kitimoon, Sippanon and Raj, Raghu G.},
	journal={Numerical Algorithms},
	pages={1--24},
	year={2025},
	publisher={Springer}
}

@article{cuyt2020sparse,
	title={Sparse multidimensional exponential analysis with an application to radar imaging},
	author={Cuyt, Annie and Hou, Yuan and Knaepkens, Ferre and Lee, Wen-shin},
	journal={SIAM Journal on Scientific Computing},
	volume={42},
	number={3},
	pages={B675--B695},
	year={2020},
	publisher={SIAM}
}

@inproceedings{chen2021algorithmic,
	title={Algorithmic foundations for the diffraction limit},
	author={Chen, Sitan and Moitra, Ankur},
	booktitle={Proceedings of the 53rd Annual ACM SIGACT Symposium on Theory of Computing},
	pages={490--503},
	year={2021}
}

@inproceedings{jin2023super,
	title={Super-resolution and robust sparse continuous {Fourier} transform in any constant dimension: {Nearly} linear time and sample complexity},
	author={Jin, Yaonan and Liu, Daogao and Song, Zhao},
	booktitle={Proceedings of the 2023 Annual ACM-SIAM Symposium on Discrete Algorithms (SODA)},
	pages={4667--4767},
	year={2023},
	organization={SIAM}
}

@article{katz2024accuracy,
	title={On the accuracy of {Prony}'s method for recovery of exponential sums with closely spaced exponents},
	author={Katz, Rami and Diab, Nuha and Batenkov, Dmitry},
	journal={Applied and Computational Harmonic Analysis},
	volume={73},
	pages={101687},
	year={2024},
	publisher={Elsevier}
}

@article{poon2023geometry,
	title={The geometry of off-the-grid compressed sensing},
	author={Poon, Clarice and Keriven, Nicolas and Peyr{\'e}, Gabriel},
	journal={Foundations of Computational Mathematics},
	volume={23},
	number={1},
	pages={241--327},
	year={2023},
	publisher={Springer}
}

@article{sahnoun2017multidimensional,
	title={Multidimensional {ESPRIT} for damped and undamped signals: Algorithm, computations, and perturbation analysis},
	author={Sahnoun, Souleymen and Usevich, Konstantin and Comon, Pierre},
	journal={IEEE Transactions on Signal Processing},
	volume={65},
	number={22},
	pages={5897--5910},
	year={2017},
	publisher={IEEE}
}

@article{haardt2002simultaneous,
	title={Simultaneous {Schur} decomposition of several nonsymmetric matrices to achieve automatic pairing in multidimensional harmonic retrieval problems},
	author={Haardt, Martin and Nossek, Josef A.},
	journal={IEEE Transactions on Signal Processing},
	volume={46},
	number={1},
	pages={161--169},
	year={2002},
	publisher={IEEE}
}

@book{nesterov2018lectures,
	title={Lectures on Convex Optimization},
	author={Nesterov, Yurii},
	volume={137},
	year={2018},
	publisher={Springer}
}

@book{adler2007random,
	title={Random Fields and Geometry},
	author={Adler, Robert J. and Taylor, Jonathan E.},
	year={2007},
	publisher={Springer}
}

@article{fannjiang2025optimality,
	title={Optimality of {Gradient-MUSIC} for Spectral Estimation},
	author={Fannjiang, Albert and Li, Weilin and Liao, Wenjing},
	journal={arXiv preprint arXiv:2504.06842},
	year={2025}
}

@book{kato2013perturbation,
	title={Perturbation Theory for Linear Operators},
	author={Kato, Tosio},
	volume={132},
	year={1976},
	publisher={Springer Science \& Business Media}
}

@article{morales2002bolzano,
	title={A {Bolzano}'s theorem in the new millennium},
	author={Morales, Claudio H.},
	journal={Nonlinear Analysis: Theory, Methods \& Applications},
	volume={51},
	number={4},
	pages={679--691},
	year={2002},
	publisher={Elsevier}
}

@article{tropp2015introduction,
	title={An introduction to matrix concentration inequalities},
	author={Tropp, Joel A.},
	journal={Foundations and Trends{\textregistered} in Machine Learning},
	volume={8},
	number={1-2},
	pages={1--230},
	year={2015},
	publisher={Now Publishers, Inc.}
}

@article{gonccalves2018note,
	title={A note on band-limited minorants of an {Euclidean} ball},
	author={Gon{\c{c}}alves, Felipe},
	journal={Proceedings of the American Mathematical Society},
	volume={146},
	number={5},
	pages={2063--2068},
	year={2018}
}

@book{watson1922treatise,
	title={A Treatise on the Theory of Bessel functions},
	author={Watson, George Neville},
	volume={3},
	year={1922},
	publisher={The University Press}
}

@article{holt1996beurling,
	title={The {Beurling-Selberg} extremal functions for a ball in {Euclidean} space},
	author={Holt, Jeffrey J. and Vaaler, Jeffrey D.},
	journal={Duke Mathematical Journal},
	volume={83},
	number={1},
	pages={203--248},
	year={1996}
}

@article{helsen2022general,
	title={General framework for randomized benchmarking},
	author={Helsen, Jonas and Roth, Ingo and Onorati, Emilio and Werner, Albert H and Eisert, Jens},
	journal={PRX Quantum},
	volume={3},
	number={2},
	pages={020357},
	year={2022},
	publisher={APS}
}

@article{poon2019multidimensional,
	title={Multidimensional sparse super-resolution},
	author={Poon, Clarice and Peyr{\'e}, Gabriel},
	journal={SIAM Journal on Mathematical Analysis},
	volume={51},
	number={1},
	pages={1--44},
	year={2019},
	publisher={SIAM}
}

@article{stroeks2022spectral,
	title={Spectral estimation for {Hamiltonians}: a comparison between classical imaginary-time evolution and quantum real-time evolution},
	author={Stroeks, ME and Helsen, J and Terhal, BM},
	journal={New Journal of Physics},
	volume={24},
	number={10},
	pages={103024},
	year={2022},
	publisher={IOP Publishing}
}

@article{Somma2019QuantumEigenvalue,
  author  = {Rolando D. Somma},
  title   = {Quantum eigenvalue estimation via time series analysis},
  journal = {New Journal of Physics},
  volume  = {21},
  number  = {12},
  pages   = {123025},
  year    = {2019},
  doi     = {10.1088/1367-2630/ab5a8c}
}

@article{LinTong2022HeisenbergLimited,
  author  = {Lin Lin and Yu Tong},
  title   = {Heisenberg-limited ground-state energy estimation for early fault-tolerant quantum computers},
  journal = {PRX Quantum},
  volume  = {3},
  pages   = {010318},
  year    = {2022},
  doi     = {10.1103/PRXQuantum.3.010318}
}

@article{de2016exact,
	title={Exact solutions to super resolution on semi-algebraic domains in higher dimensions},
	author={De Castro, Yohann and Gamboa, Fabrice and Henrion, Didier and Lasserre, J-B},
	journal={IEEE Transactions on Information Theory},
	volume={63},
	number={1},
	pages={621--630},
	year={2016},
	publisher={IEEE}
}

@article{batenkov2021super,
	title={Super-resolution of near-colliding point sources},
	author={Batenkov, Dmitry and Goldman, Gil and Yomdin, Yosef},
	journal={Information and Inference: A Journal of the IMA},
	volume={10},
	number={2},
	pages={515--572},
	year={2021},
	publisher={Oxford University Press}
}

@article{li2021stable,
	title={Stable super-resolution limit and smallest singular value of restricted Fourier matrices},
	author={Li, Weilin and Liao, Wenjing},
	journal={Applied and Computational Harmonic Analysis},
	volume={51},
	pages={118--156},
	year={2021},
	publisher={Elsevier}
}

@article{haardt2008higher,
	title={Higher-order {SVD}-based subspace estimation to improve the parameter estimation accuracy in multidimensional harmonic retrieval problems},
	author={Haardt, Martin and Roemer, Florian and Del Galdo, Giovanni},
	journal={IEEE Transactions on Signal Processing},
	volume={56},
	number={7},
	pages={3198--3213},
	year={2008},
	publisher={IEEE}
}

@article{candes2014towards,
	title={Towards a Mathematical Theory of Super-resolution},
	author={Cand{\`e}s, Emmanuel J. and Fernandez-Granda, Carlos},
	journal={Communications on Pure and Applied Mathematics},
	volume={67},
	number={6},
	pages={906--956},
	year={2014},
	publisher={Wiley Online Library}
}

@article{vaaler1985some,
	title={Some extremal functions in {F}ourier analysis},
	author={Vaaler, Jeffrey D.},
	journal={Bulletin of the American Mathematical Society},
	volume={12},
	number={2},
	pages={183--216},
	year={1985}
}

@article{donoho1992superresolution,
	title={Superresolution via sparsity constraints},
	author={Donoho, David L.},
	journal={SIAM Journal on Mathematical Analysis},
	volume={23},
	number={5},
	pages={1309--1331},
	year={1992},
	publisher={SIAM}
}

@article{andersson2018esprit,
	title={{ESPRIT} for multidimensional general grids},
	author={Andersson, Fredrik and Carlsson, Marcus},
	journal={SIAM Journal on Matrix Analysis and Applications},
	volume={39},
	number={3},
	pages={1470--1488},
	year={2018},
	publisher={SIAM}
}

@inproceedings{abed1998least,
	title={A least-squares approach to joint {Schur} decomposition},
	author={Abed-Meraim, Karim and Hua, Yingbo},
	booktitle={Proceedings of the 1998 IEEE International Conference on Acoustics, Speech and Signal Processing, ICASSP'98 (Cat. No. 98CH36181)},
	volume={4},
	pages={2541--2544},
	year={1998},
	organization={IEEE}
}

@article{fannjiang2012coherence,
	title={Coherence-pattern guided compressive sensing with unresolved grids},
	author={Fannjiang, Albert C. and Liao, Wenjing},
	journal={SIAM Journal on Imaging Sciences},
	volume={5},
	number={1},
	pages={179--202},
	year={2012},
}

@article{liao2016music,
	title={{MUSIC} for single-snapshot spectral estimation: {S}tability and super-resolution},
	author={Liao, Wenjing and Fannjiang, Albert},
	journal={Applied and Computational Harmonic Analysis},
	volume={40},
	number={1},
	pages={33--67},
	year={2016},
	publisher={Elsevier}
}

@article{liao2015multi,
	title={ {MUSIC} for multidimensional spectral estimation: {S}tability and super-resolution},
	author={Liao, Wenjing},
	journal={IEEE Transactions on Signal Processing},
	volume={63},
	number={23},
	pages={6395-6406},
	year={2015}
}

@book{stewart1990matrix,
  title={Matrix Perturbation Theory},
  author={Gilbert W. and Sun, Ji-Guang},
  year={1990},
  publisher={Academic Press Boston}
}

@article{prony1795essai,
	title={Essai exp{\'e}rimentale et analytique},
	author={Prony, Par R.},
	journal={Journal de L'Ecole Polytechnique},
	volume={1},
	number={22},
	pages={24-76},
	year={1795}
}

@article{schmidt1986multiple,
	title={Multiple emitter location and signal parameter estimation},
	author={Schmidt, Ralph O.},
	journal={IEEE Transactions on Antennas and Propagation},
	volume={34},
	number={3},
	pages={276--280},
	year={1986},
	publisher={IEEE}
}

@article{kailath1989esprit,
	title={{ESPRIT}-estimation of signal parameters via rotational invariance techniques},
	author={Roy, Richard and Kailath, Thomas},
	journal={IEEE Transactions on Acoustics, Speech, and Signal Processing},
	volume={37},
	number={7},
	pages={984-995},
	year={1989},
	publisher={IEEE}
}

\appendix
    
    	\section{Proof of \cref{lem:abstractP}}
	\label{proof:abstractP}
	
	\begin{proof}
		Let $R$ denote the reflection map $Rf(x)=f(-x)$. In the special case of $y$ defined in \eqref{eq:y}, a computation shows that
		\begin{equation}
			\begin{split}
				H(y)f(x)
				&=\int_{X'} y(x+x') f(x') \, d\nu'(x')  \\
				&= \sum_{\ell=1}^s a_\ell \, e^{2\pi i \theta_\ell \cdot x} \int_{X'} Rf(x')e^{-2\pi i \theta_\ell \cdot x'}  \, d\nu(x'). 
			\end{split} \label{eq:hankelfactorization1} 
		\end{equation}
		From this formula, we already see that $H(y)$ has rank $s$. To quantify its singular values, we can write this in a more efficient way. Define the multiplication map $M_{\bfa}\colon \C^s \to \C^s$ by $(M_{\bfa} b)_\ell=a_\ell b_\ell$. The previous formula \eqref{eq:hankelfactorization1}  is equivalent to	
		\begin{equation}
			\label{eq:hankelfactorization2}
			H(y)
			= T_{\vartheta} M_{\bfa} (T_{\vartheta}')^* R. 
		\end{equation}
		Consequently, since $\min_\ell |a_\ell|=1$ (due to \cref{def:spectral}) and that $R$ is isometric,
		\begin{equation}
			\sigma_s(H(y))
			\geq \sigma_s(T_{\vartheta}) \sigma_s(T_{\vartheta}').
			\label{eq:sigslower} 
		\end{equation}
		
		Let us first deal with the simpler case where $H(\tilde y)=H(y)+H(\eta)$ is a matrix. By Weyl's inequality for singular values and the assumption $4\|H(\eta)\|\leq \sigma_s(H(\tilde y))$, we see that
		$$
		\sigma_s(H(\tilde y))
		\geq \sigma_s(H(y)) - \|H(\eta)\|
		\geq \frac 34 \sigma_s(H(y)). 
		$$
		Using this with Wedin's theorem 
		\begin{equation}
	\label{eq:wedin}
	\norm{P_U-P_{\tilde U}}
	\leq \frac{\|H(\eta)\|}{\sigma_s(H(\tilde y))}
\end{equation}	
 and \eqref{eq:sigslower} implies 
		$$
		\norm{P_U-P_{\tilde U}}
		\leq \frac{4\|H(\eta)\|}{3\sigma_s(H(y))}
		\leq \frac{4\|H(\eta)\|}{3\sigma_s(T_{\vartheta}) \sigma_s(T_{\vartheta}')}.
		$$
		This finishes proof of the matrix case. 
		
		The operator case is much more involved and we need to use powerful perturbation theory. Under the assumption $\tilde y(x+x')\in L^2_{\nu\times \nu'}(X\times X')$ the operator $H(\tilde y)$ is Hilbert-Schmidt and consequently has point spectrum. To control the locations of its singular values, we will reformulate this in terms of an associated operator and control its spectrum (set of eigenvalues).
		
		Consider the Hilbert space $L^2(B_m)\oplus L^2(B_m)$ where $\oplus$ is the direct sum of vector spaces and define $S$ on $L^2(B_m)\oplus L^2(B_m)$ via 
		$$
		S(f\oplus g):= (H(y) g) \oplus (H(y)^* f).
		$$
		A calculation shows that $S$ is Hermitian and $\pm \sigma_j(H(y))$ are eigenvalues of $S$ with corresponding eigenfunctions $\pm (u_j \oplus v_j)/\sqrt 2$, where $u_j$ (resp., $v_j$) is the $j$-left (resp., right) singular function of $H(y)$. We also define $\tilde S$ on $L^2(B_m)\oplus L^2(B_m)$ by 
		$$
		\tilde S(f\oplus g):= (H(\tilde y) g) \oplus (H(\tilde y)^* f).
		$$
		We have the analogous relationship between the spectral decompositions of $\tilde S$ and $H(\tilde y)$. 
		
		We already showed that $H(y)$ has rank $s$ and its nonzero singular values are contained in the interval $[\sigma_s(H(y)), \, \sigma_1(H(y))]$ contained in the complex plane $\C$. For convenience, define 
		$$
		\delta:=\frac 12 \sigma_s(H(y)).
		$$ 
		Let $\Gamma$ denote a rectangle in the complex plane with counterclockwise orientation and vertices 
		$$
		(\delta, \pm \delta) \andspace
		(\sigma_1(H(y))+\delta, \pm \delta).
		$$
		This curve encloses all nonzero singular values of $H(y)$ and positive eigenvalues of $S$. Following the convention in Kato \cite{kato2013perturbation}, the resolvent of $S$ is 
		$$
		R(z,S):=(S-z I)^{-1}.
		$$ 
		The resolvent of $S$ which is well defined whenever $z$ is not in the spectrum of $S$. The projection operator onto the $s$ eigenfunctions of $S$ corresponding to positive eigenvalues (precisely $\sigma_1(H(y)), \dots, \sigma_s(H(y))$) is given the Riesz projection
		\begin{equation*}
			P_S = - \frac{1}{2\pi i}\int_\Gamma R(z,S)\, dz.
		\end{equation*} 
		
		The first step is to control $\|R(z,S)\|$ uniformly in $z\in \Gamma$. Since the distance between $z\in \Gamma$ and the spectrum of $S$ is at least $\delta$, we have
		\begin{equation}
			\label{eq:resolvent1}
			\sup_{z\in \Gamma} \|R(z,S)\|
			= \sup_{z\in \Gamma} \  \sup_{\lambda \in \text{spec}(S)} \frac{1}{|\lambda-z|}
			\leq \frac 1 \delta. 
		\end{equation}
		
		The second step is to argue that $\Gamma$ encloses exactly $s$ eigenvalues of $\tilde S$. By \cite[Theorem 3.18 on page 214]{kato2013perturbation} (apply this for $\|S-\tilde S\|$ in place of $a$ in the reference and for $b=0$ in the reference), this is the case if
		\begin{equation}
			\label{eq:katocond1}
			\sup_{z\in \Gamma} \norm{R(z,S)} \norm{S-\tilde S}
			< 1.
		\end{equation}
		We first note the helpful inequality 
		\begin{equation}
			\label{eq:Shelp1}
			\big\|S-\tilde S\big\|
			\leq \|H(\eta)\|,
		\end{equation}
		which follows from 
		$$
		\big\|(S-\tilde S)(f\oplus g)\big\|^2
		=\| H(\eta) g\|_{L^2(B_m)}^2 + \| H(\eta)^* f\|_{L^2(B_m)}^2 
		\leq \|H(\eta)\|^2 \|f\oplus g\|_{L^2(B_m)\oplus L^2(B_m)}^2.
		$$
		By \eqref{eq:resolvent1}, \eqref{eq:Shelp1}, and the assumption $4\|H(\eta)\|\leq \sigma_s(H(y))$, we have
		\begin{equation}
			\label{eq:Shelp2}
			\sup_{z\in \Gamma} \big\|R(z,S)\|\norm{\tilde S-S}
			\leq \frac{\|H(\eta)\|}{\delta}
			\leq \frac 12.
		\end{equation}
		This establishes that condition \eqref{eq:katocond1} holds. From here onward, let $P_{\tilde S}$ be projection onto the space spanned by said $s$ eigenfunctions of $\tilde S$, which also satisfies the Riesz projection formula 
		\begin{equation*}
			P_{\tilde S} = - \frac{1}{2\pi i}\int_\Gamma R(z,\tilde S)\, dz.
		\end{equation*} 
		In particular, this proves that the projection operator $P_{\tilde U}$ is well defined. 
		
		The third step is to control the resolvent $R(z,\tilde S)$ uniformly in $z\in \Gamma$. We can write
		$$
		\tilde S - zI
		=(S-zI) (I-R(z,S) (\tilde S-S)).
		$$
		Using this formula, inequality \eqref{eq:Shelp2}, and Neumann series expansion of $(I-A)^{-1}$, the operator $I-R(z,S) (\tilde S-S)$ is invertible whenever $z\in \Gamma$. Also using \eqref{eq:resolvent1}, we get
		\begin{equation}
			\label{eq:resolvent2}
			\begin{split}
				\sup_{z\in \Gamma} \big\|R(z,\tilde S)\big\|
				&= \sup_{z\in \Gamma} \big\|(\tilde S - zI)^{-1} \big\| \\
				&\leq \sup_{z\in \Gamma}  \frac{\norm{R(z,S)}}{1-\big\|R(z,S)(\tilde S-S)\big\|} 
				\leq \sup_{z\in \Gamma}  2\norm{R(z,S)}
				\leq \frac 2 \delta.
			\end{split}
		\end{equation}
		This gives us suitable control over the operator norm of $R(z,\tilde S)$.
		
		The fourth step is to control $P_S-P_{\tilde S}$. By the Riesz projection formula,
		\begin{align*}
			P_S-P_{\tilde S}
			=-\frac 1 {2\pi i} \int_\Gamma \left(R(z,S)-R(z,\tilde S) \right)\, dz
			=-\frac 1 {2\pi i} \int_\Gamma \left(R(z,S)(S-\tilde S)R(z,\tilde S) \right)\, dz. 
		\end{align*}
		Note that the length of $\Gamma$ is $4\delta + 2 \sigma_1(H(y))$. By inequalities \eqref{eq:resolvent1}, \eqref{eq:Shelp1}, and \eqref{eq:resolvent2}, we obtain
		\begin{align*}
			\norm{P_S-P_{\tilde S}}
			&\leq \frac 1{2\pi} \text{length}(\Gamma) \sup_{z\in \Gamma} \norm{R(z,S)(S-\tilde S)R(z,\tilde S)} \\
			&\leq \frac 1{2\pi} \left(4\delta + 2 \sigma_1(H(y)) \right)  \frac{2 \|H(\eta)\|}{\delta^2} \\
			&\leq \frac {32}{\pi} \frac{\sigma_1(H(y))}{\sigma_s^2(H(y))} \, \|H(\eta)\|.
		\end{align*}
		
		Now we are ready to finish the proof and to transfer these results back to $P_U-P_{\tilde U}$. First, a calculation shows that
		$
		\norm{P_U-P_{\tilde U}}
		\leq 2\norm{P_S-P_{\tilde S}}, 
		$
		which can be deduced by using the relationship between $U$ and $\tilde U$ versus the eigenfunctions of $S$ and $\tilde S$. Recall \eqref{eq:sigslower} which lower bounds the smallest singular value of $\sigma_s(H(y))$. To upper bound $\sigma_1(H(y))$, we use factorization \eqref{eq:hankelfactorization2} again to see that
		$$
		\sigma_1(H(y))
		\leq |a|_\infty \sigma_1(T_\vartheta) \sigma_1(T_\vartheta').
		$$
		This completes the proof.
	\end{proof}

	\section{Completion of proof of \cref{thm:maincube}}
	
	\label{sec:prooflemmas1}
	\label{proof:maincube} 
The analysis below is the continuation of that in \cref{sec:example1}.
	
	We pick 
	$$
	\tau= \frac{1}{4\pi m d}.
	$$
	By \cref{lem:dirichletconstants,lem:dirichletenergy2} and formula \eqref{eq:hessian1},  
	$$
	\sqrt{2 \tau \|\nabla D_m\|_{L^\infty(B_\tau)} \Delta^2 D_m(0)} 
	\leq \frac{4\pi^2}3 m (m+1) \sqrt{\frac{72\pi^2 m^2 d^2 \tau^2}{5}} 
	\leq \sqrt{ \frac 9 {10}} \lambda_d(\Psi).
	$$
	For this choice of $\tau$, we have 
	$$
	\|\nabla D_m\|_{L^\infty(B_\tau)}^2
	\leq 16 \pi^4 m^4 \tau^2 
	= \frac{ \pi^2 m^2}{d^2}
	= \frac{3}{4d^2} \lambda_d(\Psi)
	\leq \frac{3}{16}\lambda_d(\Psi). 
	$$
	On the other hand, \cref{lem:lifourier2,lem:dirichletenergy2} tell us that $\lambda_{\min}(K_\vartheta)\to 1$ and $E_1(\vartheta)\to 0$ as $\beta\to\infty$. Hence, we pick $\beta$ large enough depending only on $d$, 
	$$
	\frac{1}{\lambda_{\min}(K_\vartheta)} \left(\|\nabla D_m\|_{L^\infty(B_\tau)}^2 + E_1(\vartheta) \right)
	\leq \left( \frac 3 2 - \sqrt{ \frac 9 {10}} \right) \lambda_d(\Psi).
	$$
	This verifies assumption \eqref{eq:betacondition1} where $\delta_1=3/2$. 
	
	For the remaining condition \eqref{eq:betacondition2}, we first argue that 
	\begin{equation}
		\label{eq:thresholdcube}
		\sup_{m\geq 1} \|D_m\|_{L^\infty(B_\tau^c)} 
		<1.
	\end{equation}
	Since $|d_m|$ is strictly decreasing away from zero on $[-1/m,1/m]$ and $D_m$ is a tensor product, we see that the max of $|D_m|$ on $Q_{1/m}\setminus B_\tau$ is precisely $|d_m(\tau)|$ or $|d_m(\tau/\sqrt d)|^d$. In either case, for a fixed $\alpha<1/2$, we use formula \eqref{eq:dirichlet} to see that $|d_m(\alpha/m)|$ is monotone increasing in $m$ and converges to $\sin(2\pi \alpha)/(2\pi \alpha)$. On the other hand, for any $|t|\geq 1/m$, using formula \eqref{eq:dirichlet} and some trigonometric inequalities, we can show that $|d_m(t)|\leq m/(2(2m+1))\leq 1/4$. This proves \eqref{eq:thresholdcube}. Note that \cref{lem:lifourier2,lem:dirichletenergy2} imply	
	$$
	\limsup_{\beta\to\infty} E_0(\vartheta)+\lambda_{\max}(K_\vartheta) \max\left\{ \Big| \frac 1 {\lambda_{\max}(K_\vartheta)}-1\Big|,\, \Big|\frac 1 {\lambda_{\min}(K_\vartheta)} - 1\Big| \right\} = 0.
	$$
	Hence, we just need to make $\beta$ large enough depending only on $d$ so that condition \eqref{eq:betacondition2} is fulfilled. The choice of $\beta$ does not depend on $m$ because \eqref{eq:thresholdcube} is uniform in $m$. Thus, \eqref{eq:betacondition2} also holds. From here onward, set $\delta_2=1/4$. 
	
	It remains to verify condition \eqref{eq:noisecondition}. Let $P_{\tilde U}$ denote projection onto the $s$-dimensional leading left singular subspace of $H(\tilde y)$ as described in \cref{sec:hankeldiscrete}. 	Due to inequality \eqref{eq:sigsH} and \cref{prop:lifourier}, we have
	\begin{equation}
		\label{eq:sigsH2}
		\sigma_s(H(y))
		\geq \sigma_s^2(T_\vartheta)
		\geq  \frac 12 |Q_m\cap \Z^d|
		\asymp_d  m^d.
	\end{equation}
	The next step is to control $\|H(\eta)\|$ for two cases of interest.
	\begin{enumerate}[(a)]
		\item 
		Starting with \eqref{eq:projdiscrete1} and under the assumption that $\|\eta\|_{\ell^p(\sampset_\star)}\leq c_d |a|_\infty^{-1} m^{d/p}$ for a sufficiently small $c_d>0$,
		\begin{equation*}
            \|H(\eta)\|
			\leq |\sampset_\star|^{1/p'} \|\eta\|_{\ell^p(\sampset_\star)}
			\lesssim_d m^{d/p'} \|\eta\|_{\ell^p(\sampset_\star)}
			=m^{d}m^{-d/p} \|\eta\|_{\ell^p(\sampset_\star)}
			\leq c_d |a|_\infty^{-1} m^d.
		\end{equation*}
		Combining this with \cref{lem:abstractP} and \eqref{eq:sigsH2}, for small enough $c_d$, we have
		\begin{equation}
			\label{eq:projhelp1}
			\norm{P_U-P_{\tilde U}}
			\leq \frac{\|H(\eta)\|}{\sigma_s(H(y))-\|H(\eta)\|}
			\lesssim_d 
			\frac{m^{d/p'}\|\eta\|_{\ell^p(\sampset_\star)}}{ m^d}
			=\frac{\|\eta\|_{\ell^p(\sampset_\star)}}{ m^{d/p}}.
		\end{equation}
		\item 
		By \cref{lem:projdiscrete2}, with probability at least $1-m^{-d}$, we have
		$$
		\|H(\eta)\|
		\leq \|\sigma\|_{\ell^2(\sampset_\star)} \sqrt{2d\log(m)}.
		$$ 
		Under the assumption that $\|\sigma\|_{\ell^2(\sampset_\star)}\sqrt{\log(m)}\leq c_d |a|_\infty^{-1}  \, m^d$ for a sufficiently small $c_d$ that depends only on $d$, by \cref{lem:abstractP} and \eqref{eq:sigsH2},
		\begin{equation}
			\label{eq:projhelp2}
			\norm{P_U-P_{\tilde U}}
			\leq \frac{\|H(\eta)\|}{\sigma_s(H(y))-\|H(\eta)\|}
			\lesssim_d \frac{\|\sigma\|_{\ell^2(\sampset_\star)} \sqrt{\log(m)}}{{ m^d}}.
		\end{equation}
	\end{enumerate}
	For both the deterministic and stochastic cases, by making $c_d$ a sufficiently small constant that only depends on $d$, in view of inequalities \eqref{eq:projhelp1} and \eqref{eq:projhelp2}, we have
	\begin{equation}
		\label{eq:projhelp3}
		\|H(\eta)\|\leq c_d |a|_\infty^{-1} m^d \andspace
        \norm{P_U-P_{\tilde U}}
		\leq c_d.
	\end{equation}
    Recalling \cref{rem:sparsitydetection}, the quantity $s$ can be determined by a thresholding procedure. We can treat both the deterministic and stochastic cases simultaneously since the rest of this proof only requires \eqref{eq:projhelp3} to hold. 
	
	Having dealt with the projection error, define the MUSIC function $\tilde q:=q_{\tilde U}$. We seek to use \cref{thm:maingeometric}. The first two assumptions \eqref{eq:betacondition1} and \eqref{eq:betacondition2} were verified in \cref{sec:verification1}. The final requirement \eqref{eq:noisecondition} amounts to $\|P_U-P_{\tilde U}\|_2\lesssim_d 1$, because $\delta_2=1/4$, $\partial_j^4 D_m(0)\asymp m^4$, and $\lambda_j(\Psi)\asymp m^2$ for all $j\in \{1,\dots,d\}$. Requirement \eqref{eq:noisecondition} is satisfied due to \eqref{eq:projhelp3} for sufficiently small $c_d$. 
	
	Applying \cref{thm:maingeometric} and the formula for $\Psi$ in \eqref{eq:hessian1} shows that $\tilde q$ is an admissible optimization landscape for gradient descent and $\vartheta$ with parameters
	$$
	\left( \frac{C\sqrt d}{m} \, \norm{P_U-P_{\tilde U}}, \
	\frac{1}{4\pi m d}, \
	\norm{P_U-P_{\tilde U}}^2, \ \alpha_1 \right),
	$$
	where importantly, $\alpha_1>0$ is an absolute constants due to \eqref{eq:thresholdcube}. Combining this with \eqref{eq:projhelp1} and \eqref{eq:projhelp2} proves \eqref{eq:cubecasea} and \eqref{eq:cubecaseb}, respectively. 
	
	Next, we need to show that Gradient-MUSIC with appropriate parameters produces iterates that converge linearly to the desired local minima. To this end, we use inequalities  \eqref{eq:qtilde1} and formula for $\Psi$ in \eqref{eq:hessian1} to see that $\|\nabla \tilde q\|_{L^\infty(\Omega)}
	\leq 4\pi \sqrt{dm(m+1)/3}\lesssim \sqrt d m$. Also since $\delta_1+\delta_2=7/4$, inequality \eqref{eq:qtilde2} tells us that
	\begin{equation} \label{eq:convexitycube}
		\frac{\pi^2}{3} m(m+1)
		=\frac 1 4 \lambda_d(\Psi)
		\leq \lambda_d \round{\nabla^2 \tilde q(\omega)}
		\leq \lambda_1 \round{\nabla^2 \tilde q(\omega)} 
		\leq \frac{15}{4} \lambda_d(\Psi)
		= 5\pi^2 m(m+1). 
	\end{equation}
	Pick any finite $G\subset \domain$ with 
	\begin{equation}\label{eq:meshcube}
		\mesh(G)
		\asymp \frac 1{md}.
	\end{equation}
	By \cref{lem:threshold}, thresholding $\tilde q$ on $G$ by $c$ finds for each $\ell\in \{1,\dots,s\}$, an element $\theta_{\ell,0}\in G\cap B_\tau(\theta_\ell)$. The MUSIC function $\tilde q$ is at least twice differentiable, so by \cite[Theorem 2.1.15]{nesterov2018lectures} (with $\mu=\lambda_d(\Psi)/4$ and $L=15\lambda_d(\Psi)/4$), gradient descent with initial point $\theta_{\ell,0}$ and step size
	\begin{equation}
		\label{eq:stepsizecube}
		h
		\leq \frac 2{\mu+L}
		= \frac{3}{8\pi^2 m(m+1)},
	\end{equation} 
	produces iterates which converge at a linear rate to local minimum $\tilde\theta_\ell$. Specifically, if $h$ is equal to the right hand side of \eqref{eq:stepsizecube} and $\theta^\sharp_\ell$ is the $n$-th iterate with any initial point $\theta_{\ell,0}\in G\cap B_\tau(\theta_\ell)$, then by \cite[Theorem 2.1.15]{nesterov2018lectures} and that $|\theta_{\ell,0}-\tilde \theta_\ell|\leq 2\tau$, we see that 
	\begin{equation}
		\label{eq:numericalerrorcube}
		\big|\theta^\sharp_\ell-\tilde \theta_\ell\big|
		\leq \left(\frac{L/\mu-1}{L/\mu+1}\right)^n |\theta_{\ell,0}-\tilde \theta_\ell| 
		\leq 2 \tau \left(\frac {14}{15}\right)^n
		\lesssim_d \frac 1 m\left(\frac {14}{15} \right)^n. 
	\end{equation}
	To make this error $\epsilon/m$, where $\epsilon/m$ is any upper bound for the perturbation error $|\theta_\ell-\tilde\theta_\ell|$, we pick $n\asymp_d \log(1/\epsilon)$. This completes the proof.

	\subsection{Proof of \cref{lem:dirichletconstants}}
	\label{proof:dirichletconstants}
	
		Fix any $j\in \{1,\dots,d\}$. We use the tensor product formula for $D_m$ in \eqref{eq:squareDirichlet}, $d_m'(0)=0$, and $\|d_m\|_{L^\infty(\T)}\leq 1$ to see that
		$$
		|\partial_j D_m(\xi)|
		= |d_m'(\xi_j)-d_m'(0)|\prod_{k\ne j} |d_m(\xi_k)|
		\leq |d_m'(\xi_j)-d_m'(0)|.
		$$
		By the mean value theorem and that $\|d_m''\|_{L^\infty(\T)}\leq 4\pi^2 m^2 \|d_m\|_{L^\infty}\leq 4\pi^2 m^2$ by Bernstein's inequality for trigonometric polynomials, 
		$$
		|d_m'(\xi_j)-d_m'(0)|
		\leq \|d_m''\|_{L^\infty(\T)} |\xi_j|
		\leq 4\pi^2 m^2 |\xi_j|.
		$$
		Combining the above, we see that
		$$
		|\nabla D_m(\xi)|^2
		=\sum_{j=1}^d |\partial_j D_m(\xi)|^2
		\leq 16\pi^4 m^4 |\xi|^2. 
		$$
		This proves the first inequality.
		
		For the second part, due to \eqref{eq:squareDirichlet} and the Fourier series representation \eqref{eq:dirichlet}, if $j\ne k$, 
		$$
		\partial_j^2 \partial_k^2 D_m(0)
		= \frac{1}{(2m+1)^2}\left( \sum_{j=-m}^m 4\pi^2 j^2 \right)\left( \sum_{j=-m}^m 4\pi^2 j^2 \right)
		=\frac{16\pi^4}{9}m^2 (m+1)^2. 
		$$
		If $j=k$, then 
		$$
		\partial_j^4 D_m(0)
		= \frac{1}{(2m+1)}\sum_{j=-m}^m 16\pi^4 j^4
		=\frac{16\pi^4}{15} m (m+1)(3m^2+3m-1)
		\leq \frac{16\pi^4}{5} m^2 (m+1)^2. 
		$$
		Then we see that
		$$
		\Delta^2 K(0)
		=\sum_{j,k} \partial_j^2 \partial_k^2 D_m(0)
		\leq 16\pi^4 m^2 (m+1)^2 \left( \frac{d}{5} + \frac{d(d-1)}{9} \right)
		\leq \frac{16\pi^4 d^2}5 m^2 (m+1)^2
		$$
		This completes the proof. 
	
	\subsection{Preparatory lemmas used in proof of \cref{lem:dirichletenergy2}}
	\label{proof:dirichletenergy2}

	\begin{figure}
		\centering
		\begin{tikzpicture}[scale=0.8]
			\fill[color=gray!30] (-5,-1) rectangle (-1,1);
			\fill[color=gray!30] (1,-1) rectangle (5,1);
			\fill[color=gray!30] (-1,1) rectangle (1,5);
			\fill[color=gray!30] (-1,-1) rectangle (1,-5);
			\fill[color=gray!10] (-5,-1) rectangle (-1,-5);
			\fill[color=gray!10] (1,-1) rectangle (5,-5);
			\fill[color=gray!10] (1,1) rectangle (5,5);
			\fill[color=gray!10] (-1,5) rectangle (-5,1);
			\fill[color=gray!50] (-1,-1) rectangle (1,1);
			\draw[<->] (-5.2,0) -- (5.2,0) node[right] {$\xi_1$};
			\draw[<->] (0,-5.2) -- (0,5.2) node[above] {$\xi_2$};
			\draw[-,dashed,] (-5,1) -- (5,1);
			\draw[-,dashed,] (-5,-1) -- (5,-1);
			\draw[-,dashed,] (-1,-5) -- (-1,5);
			\draw[-,dashed,] (1,-5) -- (1,5);
			\draw (.5,.5) node {$A_{0,0}$};
			\draw (3,.5) node {$A_{1,0}$};
			\draw (-3,.5) node {$A_{1,0}$};
			\draw (0.5,3) node {$A_{0,1}$};
			\draw (0.5,-3) node {$A_{0,1}$};
			\draw (3,3) node {$A_{1,1}$};
			\draw (-3,3) node {$A_{1,1}$};
			\draw (-3,-3) node {$A_{1,1}$};
			\draw (3,-3) node {$A_{1,1}$};
		\end{tikzpicture}
		\caption{Decomposition of $\T^2$ into $A_{0,0}$, $A_{1,0}$, $A_{0,1}$, and $A_{1,1}$.}
		\label{fig:decomposition}
	\end{figure}
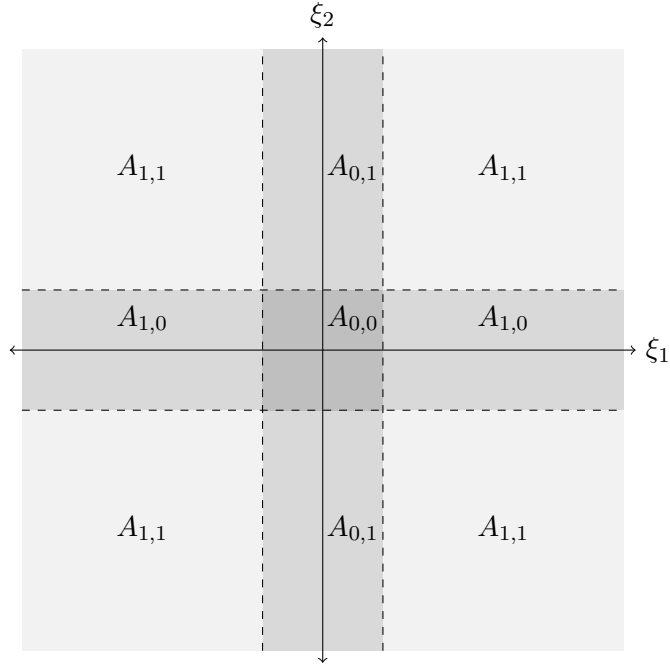
	
	\begin{lemma}
		\label{lem:dirichletenergy}
		Let $m\geq 1$, $d\geq 2$, and $\beta\geq 2\pi \sqrt d$. For any finite set $\Gamma\subset\T^d$ such that $\Delta_\infty(\Gamma)\geq \beta/m$ and $\Gamma\cap (Q_{\beta/m})^\circ=\emptyset$, we have 
		\begin{align*}
			\sum_{\gamma\in \Gamma} |D_m(\gamma)|^2
			&\leq \frac{4\pi^2 d}{\beta^2}, \andspace
			\sum_{\gamma\in \Gamma} |\nabla D_m(\gamma)|^2
			\leq \frac{16\pi^4 d^2m^2}{\beta^2}. 
		\end{align*}
	\end{lemma}
	
	\begin{proof}
		Note that $\|d_m\|_{L^\infty(\T)}\leq 1$. On the other hand, for $|t|\leq \pi$, since $|\sin(t/2)|\geq |t|/\pi$, we see that $|d_m(t)|\leq \pi/(2m|t|)$. This yields the pointwise upper bound,
		\begin{equation}
			\label{eq:Dupper}
			|D_m(\xi)|^2
			\leq \prod_{j=1}^d |d_m(\xi_j)|^2
			\leq \prod_{j=1}^d \min \left\{1, \, \frac{\pi^2}{4m^2 \xi_j^2} \right\}.
		\end{equation} 
		
		Let $\sigma\in \{0,1\}^d$, and note there are $2^d$ many distinct $\sigma$. Define the region $A_{\sigma}\subset\R^d$ to be the set of all $\xi\in \T^d$ such that $|\xi_j|<\beta/m$ if $\sigma_j=0$ and $|\xi_j|\geq \beta/m$ if $\sigma_j=1$. See Figure \ref{fig:decomposition} for an illustration. For convenience, let $I_{\sigma}:=\{j\colon \sigma_j=1\}$ be support of  $\sigma$. Note that $(Q_{\beta/m})^\circ=A_{(0,\dots,0)}$ and recall the assumption $\Gamma\cap (Q_{\beta/m})^\circ=\emptyset$. This implies the formula, 
		\begin{equation}
			\sum_{\gamma\in \Gamma} |D_m(\gamma)|^2
			=\sum_{r=1}^d \sum_{|\sigma|=r} \sum_{\gamma\in \Gamma\cap A_\sigma} |D_m(\gamma)|^2.  \label{eq:slab}		
		\end{equation}
		
		For now, we fix $\sigma\ne (0,\dots,0)$ and let $r=|\sigma|\geq 1$. Using \eqref{eq:Dupper}, we get 
		\begin{equation}
			\begin{split}
				\sum_{\gamma\in \Gamma\cap A_\sigma} |D_m(\gamma)|^2
				&\leq \sum_{\gamma\in \Gamma \cap A_\sigma}  \prod_{j=1}^d \min \left\{1, \, \frac{\pi^2}{4m^2 \gamma_j^2} \right\} 
				\leq \sum_{\gamma\in \Gamma \cap A_\sigma}  \prod_{j\in I_{\sigma}} \frac{\pi^2}{4m^2 \gamma_j^2}. 
			\end{split} \label{eq:prodhelp}
		\end{equation}
		To control the product term on the right hand side of \eqref{eq:prodhelp}, we proceed by a dimension reduction argument. We claim that 
		\begin{equation}
			\label{eq:columbenergy}
			\sum_{\gamma\in \Gamma \cap A_\sigma} \prod_{j\in I_{\sigma}} \frac{1}{\gamma_j^2}
			\leq  \frac{4^rm^{2r}}{\beta^{2r}}. 
		\end{equation}
		The left hand side product only depends on $\gamma_j$ for $j\in I_{\sigma}$, so without loss of generality, assume that $\sigma=(1,\dots,1,0,\dots,0)$. By definition of $I_{\sigma}$, we have $|\gamma_j|\geq \beta/m$ for $j\leq r$ and $|\gamma_j|<\beta/m$ for $j>r$. The key observation is that the projection of distinct $\gamma,\gamma'\in \Gamma\cap A_\sigma$ onto their first $r$ entries are separated by at least $\beta/m$ with respect to the $\ell^\infty$ norm on $\R^r$ -- because if not, then since $|\gamma_j|<\beta/m$ and $|\gamma_j'|<\beta/m$ for all $j>r$, that would imply that $|\gamma-\gamma'|_\infty<\beta/m$ which would contradict the assumption that $\Delta_\infty(\Gamma)\geq \beta/m$. Thus, we assume without loss of generality that the last $d-r$ coordinates of each $\gamma\in \Gamma\cap A_\sigma$ are zero and view $\Gamma\cap A_\sigma$ as a subset of $\R^r$. We are now in position to apply \cref{lem:boxpacking} (which is stated and proved after the proof of this lemma), which yields \eqref{eq:columbenergy}. 
		
		Continuing, we combine \eqref{eq:slab}, \eqref{eq:prodhelp} and \eqref{eq:columbenergy} to see that
		\begin{align*}
			\sum_{\gamma\in \Gamma} |D_m(\gamma)|^2
			\leq \sum_{r=1}^d \binom{d}{r} \frac{2^r \pi^{2r}}{\beta^{2r}}
			\leq \sum_{r=1}^\infty \frac{2^r \pi^{2r} d^r}{\beta^{2r}}
			= e^{2\pi^2 d/\beta^2}-1
			\leq \frac{4\pi^2 d}{\beta^2},  
		\end{align*}
		where the final inequality used that $e^t-1\leq 2t$ for all $t\in [0,1]$ and $\beta\geq 2\pi \sqrt d$ by assumption. This proves the claimed inequality.
		
		To prove the second inequality of this lemma, we use a similar argument. By Bernstein's inequality, 
		$\|d_m'\|_{L^\infty(\T)}\leq 2\pi m\|d_m\|_{L^\infty(\T)}=2\pi m.
		$
		On the other hand, a direct calculation and using that $2|t|/\pi \leq |\sin(t)|\leq |t|$ for $|t|\leq \pi/2$ shows that
		\begin{align*}
			|d_m'(t)|
			&=\left| \frac{ \pi \cos((2m+1)\pi t)}{\sin(\pi t)} - \frac{ \pi \sin((2m+1)\pi t)\cos(\pi t)}{(2m+1)\sin^2(\pi t)}\right| 
			\leq \frac{\pi}{2|t|}+\frac{\pi^2}{4|t|}
			< \frac{\pi^2}{|t|}. 
		\end{align*}
		This ultimately yields the inequality, $|d_m'(t)|
			\leq 2 \pi m \min\left\{1, {\pi}/(2m|t|) \right\}$. 
		Using that $D_m$ is a tensor product, we have 
		\begin{equation}
			\label{eq:Dupper2}
			|\nabla D_m(\xi)|^2
			=\sum_{j=1}^d |d_m'(\xi_j)|^2 \prod_{k\ne j} |d_m(\xi_k)|^2
			\leq 4 \pi^2 m^2 d \prod_{j=1}^d \min \left\{1, \, \frac{\pi^2}{4m^2 \xi_j^2} \right\}.
		\end{equation}
		The upper bounds for $|D_m|^2$ and $|\nabla D_m|^2$ in \eqref{eq:Dupper} and \eqref{eq:Dupper2} are only different up to the constant factor $4\pi^2 m^2 d $. Hence, the claimed inequality for $|\nabla D_m|^2$ follows by repeating the same argument.
	\end{proof}
	
	\begin{lemma}
		\label{lem:boxpacking}
		Let $r\geq 1$, $\delta>0$, and $\Gamma\subset\R^r$ be a finite set such that $\Delta_\infty(\Gamma)\geq \delta$ and $|\gamma_j|\geq \delta$ for all $\gamma\in \Gamma$ and $j\in \{1,\dots,r\}$. Then 
		$$
		\sum_{\gamma\in \Gamma} \prod_{j=1}^r \frac{1}{\gamma_j^2} 
		\leq \frac{4^r}{\delta^{2r}}. 
		$$
	\end{lemma}
	
	\begin{proof}
		We define the function $g\colon \R^r\to [0,\infty]$ by
		$$
		g(\gamma):=\prod_{j=1}^r \frac{1}{\gamma_j^2}. 
		$$
		Note that $g$ is a tensor product of nonnegative single variable functions that are each convex on any interval that does not contain zero. In view of the assumption that $|\gamma_j|\geq \delta$ for each $j\in \{1,\dots,r\}$, the interval $[\gamma_j-\frac 12 \delta, \gamma_j+ \frac 12 \delta]$ does not contain zero, and $t\mapsto 1/t^2$ is convex on this interval, which has length $\delta$ and midpoint $\gamma_j$. By the Hermite–Hadamard inequality,
		\begin{equation}
			\label{eq:ghelp1}
			g(\gamma)
			= \prod_{j=1}^r \frac 1{\gamma_j^2}
			\leq \prod_{j=1}^r \frac 1 {\delta} \int_{\gamma_j-\frac 12 \delta}^{\gamma_j+\frac 12 \delta} \frac 1 {t^2}\, dt
			= \frac{1}{\delta^r} \int_{Q_{\delta/2}(\gamma)} g. 
		\end{equation}
		Since $\Delta_\infty(\Gamma)\geq \delta$, for distinct $\gamma,\gamma'\in \Gamma$, the closed cubes $Q_{\delta/2}(\gamma)$ and $Q_{\delta/2}(\gamma')$ only overlap on a set of measure zero. We again recall that $|\gamma_j|\geq \delta$ for each $\gamma\in\Gamma$ and $j\in \{1,\dots,r\}$ in order to see that $\bigcup_{\gamma\in \Gamma} (Q_{\delta/2}(\xi))\subset ([-\frac 12 \delta,\frac 12\delta]^c)^d$. Using this with \eqref{eq:ghelp1}, we see that
		\begin{align*}
			\sum_{\gamma\in \Gamma} g(\gamma)
			\leq \sum_{\gamma\in \Gamma} \frac{1}{\delta^r}  \int_{Q_{\delta/2}(\gamma)} g
			= \frac{1}{\delta^r} \int_{\bigcup_{\gamma\in \Gamma} Q_{\delta/2}(\gamma)} g
			\leq \frac{1}{\delta^r} \left( \int_{[ - \frac 12 \delta, \frac 12 \delta]^c} \frac 1 {t^2}\, dt \right)^r
			= \frac{4^r}{\delta^{2r}}. 
		\end{align*}
		This completes the proof. 
	\end{proof}
	
	\subsection{Proof of \cref{lem:projdiscrete2}}
	\label{proof:projdiscrete2}
	
		For convenience, we let $\sampset_\star = Q_{2m}\cap \Z^d$ and $\sampset= Q_{m}\cap \Z^d$. The second part of this lemma is a straightforward consequence of the first. If $\sigma(x)=\sigma$ for all $x\in Q_{2m}\cap \Z^d$, then $\|\sigma\|_{\ell^2(Q_{2m}\cap \Z^d)}=\sigma \sqrt{|Q_{2m}\cap \Z^d|}\lesssim \sigma m^{d/2}$.
		
		It remains to prove the second part of the lemma. Let $e_x\colon \sampset_\star\to \C$ denote the function $e_x(x')=1$ if $x=x'$ and $e_x(x')=0$ otherwise. Since the map $\eta\mapsto H(\eta)$ is linear, $\{e_x\}_{x\in \sampset_\star}$ is a basis for functions defined on $\sampset_\star$, and $\eta(x)=\sigma(x)g(x)$, we have
		$$
		H(\eta)
		=\sum_{x\in \sampset_\star} \eta(x) H(e_x)
		=\sum_{x\in \sampset_\star} g(x) A_x.
		$$ 
		
		We seek to apply \cite[Theorem 4.1.1]{tropp2015introduction} which controls the spectral norm of a sum of fixed matrices multiplied by normal random variables. To this end, we need to control the matrix variance statistic,
		$$
		\var(H(\eta)):= \max\left\{ \Big\| \sum_{x\in \sampset_\star} A_x A_x^* \Big\|, \Big\| \sum_{x\in \sampset_\star} A_x^* A_x \Big\| \right\}.
		$$
		To do this, we first note that $A_x A_x^*$ is a diagonal matrix due to the calculation
		\begin{align*}
			(A_x A_x^*)_{j,k}
			&= (\sigma(x))^2 (H(e_x)H(e_x)^*)_{j,k}
			= (\sigma(x))^2 \sum_{\ell=1}^{|\sampset_\star|} e_x\left(x^{(j)}+x^{(\ell)}\right) e_x\left(x^{(\ell)}+x^{(k)}\right).
		\end{align*}
		For the last sum, the only nonzero terms are $\ell$ for  which $x=x^{(j)}+x^{(\ell)}=x^{(\ell)}+x^{(k)}$, which implies $x^{(j)}=x^{(k)}$. This shows that $A_x A_x^*$ is diagonal. For fixed $j$, either there is a unique $x^{(\ell)}$ such that $x^{(j)}+x^{(\ell)}=x$ or no such $x^{(\ell)}$, so only nonzero diagonal entries of $A_x A_x^*$ are $(\sigma(x))^2$ and so 
		$$
		\|A_x A_x^*\|\leq (\sigma(x))^2.
		$$
		The same argument shows that
		$$
		\|A_x^* A_x\|\leq (\sigma(x))^2.
		$$
		
		Thus, we have the inequality 
		$$
		\var(H(\eta))
		\leq \sum_{x\in \sampset_\star} (\sigma(x))^2
		=\|\sigma\|_{\ell^2(\sampset_\star)}^2. 
		$$
		Applying \cite[Theorem 4.1.1]{tropp2015introduction} yields, for all $t\geq 0$,
		\begin{equation*}
			\P\left(\|H(\eta)\|\geq t\right)
			\leq 2 |\sampset| \exp\left(-\frac{t^2}{2\|\sigma\|_{\ell^2(\sampset_\star)}^2}\right). 
		\end{equation*}
		Specializing this to $t=\|\sigma\|_{\ell^2(\sampset_\star)} \sqrt{2 \log(|\sampset|/\delta)}$ completes the proof.

	\section{Completion of proof of \cref{thm:mainball}}
	\label{sec:prooflemmas2}	
	\label{proof:mainball}
	
	The analysis below is the continuation of that in \cref{sec:example2}.

		
	We pick 
		$$
		\tau= \frac{1}{4\pi m \sqrt d}.
		$$
		For this choice of $\tau$, by \cref{lem:besselconstants} and \eqref{eq:hessian2}, we have 
		$$
		\|\nabla W_m\|_{L^\infty(B_\tau)}^2
		\leq \frac{16\pi^4 m^4 \tau^2}{(d+2)^2}
		\leq \frac{\pi^2 m^2}{d(d+2)^2}
		= \frac{1}{4 d(d+2)}\lambda_d(\Psi) .
		$$ 
		By \cref{lem:besselenergy,lem:besselconstants} and formula \eqref{eq:hessian1},  
		$$
		\sqrt{2 \tau \|\nabla W_m\|_{L^\infty(B_\tau)} \Delta^2 W_m(0)} 
		\leq \frac{4\pi^2 m^2}{d+2} \sqrt{8\pi^2 m^2 d\tau^2 \frac{d+2}{d+4}} 
		\leq \frac 1{\sqrt2} \lambda_d(\Psi).
		$$
		Employing \cref{lem:orthogonalityball,lem:besselenergy}, for all sufficiently large $\beta$ depending only on $d$, we have 
		$$
		\frac{1}{\lambda_{\min}(K_\vartheta)} \left(\|\nabla W_m\|_{L^\infty(B_\tau)}^2 + E_1(\vartheta) \right)
		\leq \left(1-\frac 1{\sqrt2}\right) \lambda_d(\Psi).
		$$
		This verifies assumption \eqref{eq:betacondition1} where $\delta_1=1$. 
		
		For the second condition \eqref{eq:betacondition2}, we first argue that 
		\begin{equation}
			\label{eq:thresholdball}
			\sup_{m >0} \|W_m\|_{L^\infty(B_\tau^c)}
			<1.
		\end{equation}
		The only step that requires justification is the first equality. Since $W_m$ is radial, we treat it as a function defined on $[0,\infty)$. Let $j_{\alpha,k}$ denote the $k$-th positive zero of $J_{\alpha}$. It is known that the sign of $J_\alpha$ alternates on intervals whose endpoints are its positive roots. Since $J_{d/2}$ is positive and strictly decreasing on the interval $[0,j_{d/2,1}/(2\pi m)]$, the sup of $|W_m(r)|$ on this interval is attained at $r=\tau$. Using the formula $(t^{-\alpha}J_\alpha(t))'=-t^{-\alpha}J_{\alpha+1}(t)$, we see that the local extreme values of $W_m$ occur exactly when $J_{d/2+1}(2\pi m r)=0$, or by definition, at $r= j_{d/2+1,k}/(2\pi m)$. Moreover, the local max of $|J_\alpha(t)|$ are decreasing in $t$. Putting together these observations, 
		\begin{align*}
			\sup_{r\geq \tau} |W_m(r)|
			=\max \left\{ \left|W_m \left( \frac{1}{4\pi m \sqrt d} \right)\right|, \, \left|W_m\left( \frac{j_{d/2+1,k}}{2\pi m}\right) \right| \right\}. 
		\end{align*}
		The right hand side does not depend on $m$ and both terms are strictly upper bounded by $1$, which proves \eqref{eq:thresholdball}. Note that \cref{lem:orthogonalityball,lem:besselenergy} imply	
		$$
		\limsup_{\beta\to\infty} E_0(\vartheta)+\lambda_{\max}(K_\vartheta) \max\left\{ \Big| \frac 1 {\lambda_{\max}(K_\vartheta)}-1\Big|,\, \Big|\frac 1 {\lambda_{\min}(K_\vartheta)} - 1\Big| \right\} = 0.
		$$
		Hence, we just need to make $\beta$ large enough depending only on $d$ so that condition \eqref{eq:betacondition2} is fulfilled. The choice of $\beta$ does not depend on $m$ because \eqref{eq:thresholdball} is uniform in $m$.
		
		We need to verify the third condition \eqref{eq:noisecondition}, for which we pick $\delta_2=1/2$. Let $P_{\tilde U}$ be projection onto the $s$-dimensional leading left singular functions of the operator $H(\tilde y)$ as described in \cref{sec:hankelcontinuous}. By the factorization \eqref{eq:hankelfactorization2}, that  $\min_\ell |a_\ell|= 1$ (see \cref{def:spectral}), and \cref{lem:orthogonalityball}, we have 
		\begin{align}
			\sigma_s(H(y))
			&\geq \sigma_s^2 (T_{\vartheta})
			\geq \frac 12 |B_m|, 	\label{eq:sigshankel1} \\
			\sigma_1(H(y))
			&\leq |a|_\infty \,  \sigma_{\min}^2(T_\vartheta)
			\leq \frac 3 2 |a|_\infty \, |B_m|.\label{eq:sigshankel2}
		\end{align}
		The next step is to control $\|H(\eta)\|$ and $\|P_U-P_{\tilde U}\|$ for the two cases of $\eta$.
		\begin{enumerate}[(a)]
			\item 
			Under the assumption that $\|\eta\|_{L^p(B_{2m})}\leq c_d\,  |a|_\infty^{-1} \, m^{-d/p}$, by inequality \eqref{eq:hankelnoise2}, we have
			\begin{equation*}
			\|H(\eta)\|
			\leq |B_{2m}|^{1/p'} \|\eta\|_{L^p(B_{2m})}
			\lesssim_d m^{d}m^{-d/p} \|\eta\|_{L^p(B_{2m})}
			\leq c_d |a|_\infty^{-1} m^d.
			\end{equation*}
			By making $c_d$ small enough and recalling the lower bound in \eqref{eq:sigshankel2}, we use \cref{lem:abstractP}, \eqref{eq:sigshankel1}, and \eqref{eq:sigshankel2} to obtain
			\begin{equation}
				\label{eq:projhelp4}
				\norm{P_U-P_{\tilde U}}
				\lesssim \frac{|a|_\infty}{|B_m|} \|H(\eta)\|
				\lesssim_d |a|_\infty m^{-d/p} \|\eta\|_{L^p(B_{2m})}.
			\end{equation}
        \item 
		By \cref{lem:contnoisestochastic}, with probability at least $1-m^{-d}$, we have
        $$
        \|H(\eta)\|
        \lesssim_{d,\gamma} \|\sigma\|_{L^2(B_{2m})} \sqrt{\log(m)}.
        $$
        From here onward, we assume that this event holds. By \cref{lem:abstractP}, \eqref{eq:sigshankel1}, and \eqref{eq:sigshankel2},
		\begin{equation}
			\label{eq:projhelp5}
			\norm{P_U-P_{\tilde U}}
			\lesssim \frac{\sigma_1(H(y))}{\sigma_s^2(H(y))} \, \|H(\eta)\|
			\lesssim_{d,\gamma} \frac{|a|_\infty \|\sigma\|_{L^2(B_{2m})} \sqrt{\log(m)}}{{ m^d}}.
		\end{equation}
		\end{enumerate}
		For both the deterministic and stochastic cases, by making $c_d$ sufficiently small depending only on $d$, in view of inequalities \eqref{eq:projhelp4} and \eqref{eq:projhelp5}, we have
		\begin{equation}
			\label{eq:projhelp6}
            \|H(\eta)\|
            \leq c_d |a|_\infty^{-1} m^d \andspace 
			\norm{P_U-P_{\tilde U}}
			\leq c_d.
		\end{equation}
		By \cref{rem:sparsitydection2}, the quantity $s$ can be determined by a thresholding procedure. We can treat both the deterministic and stochastic cases simultaneously since the rest only requires \eqref{eq:projhelp6} to hold. 

        Having dealt with the projection error, define the MUSIC function $\tilde q:=q_{\tilde U}$. We seek to use \cref{thm:maingeometric}. The first two assumptions \cref{eq:betacondition1} and \eqref{eq:betacondition2} where verified in \cref{sec:verification2}. The final requirement \eqref{eq:noisecondition} amounts to $\|P_U-P_{\tilde U}\|\lesssim_d 1$, because $\delta_2=1/2$,  $\partial_j^4 W_m(0)\asymp_d m^4$ and $\lambda_j(\Psi)\asymp_d m^2$ for all $j\in \{1,\dots,d\}$. Requirement \eqref{eq:noisecondition} is satisfied due to \eqref{eq:projhelp6} for sufficiently small $c_d$. 
	
		Applying \cref{thm:maingeometric} and the formula for $\Psi$ in \eqref{eq:hessian2} shows that $\tilde q$ is an admissible optimization landscape for gradient descent and $\vartheta$ with parameters
		$$
		\left( \frac{Cd}{ m} \norm{P_U-P_{\tilde U}}, \
		\frac{1}{4\pi m \sqrt d}, \
		\norm{P_U-P_{\tilde U}}^2, \ \alpha_1 \right),
		$$
		where $\alpha_1>0$ is an absolute constant due to  \eqref{eq:thresholdball}. Combining this with \eqref{eq:projhelp4} and \eqref{eq:projhelp5} proves \eqref{eq:ballcasea} and \eqref{eq:ballcaseb}, respectively. 
	
	Next, we need to show that Gradient-MUSIC with appropriate parameters produces iterates that converge linearly to the desired local minima. To this end, we use inequality \eqref{eq:qtilde1} and formula for $\Psi$ in \eqref{eq:hessian2} to see that $\|\nabla \tilde q\|_{L^\infty(\Omega)}\leq 4\pi m$. Employing \eqref{eq:qtilde2} and recalling that $\delta_1+\delta_2=3/2$, we have
	\begin{equation} \label{eq:convexityball}
		\frac{2\pi^2m^2}{d+2} 
		= \frac 12 \lambda_d(\Psi)
		\leq \lambda_d \round{\nabla^2 \tilde q(\omega)}
		\leq \lambda_1 \round{\nabla^2 \tilde q(\omega)} 
		\leq \frac 72 \lambda_d(\Psi) 
		= \frac{14\pi^2 m^2}{d+2}. 
	\end{equation}
	\cref{lem:threshold} is fulfilled by picking any finite $G\subset \domain$ with 
	\begin{equation}\label{eq:meshball}
		\mesh(G)
		\asymp \frac 1{ \sqrt d m}.
	\end{equation}
	Then for each $\theta_\ell$, we can find an element $\theta_{\ell,0}\in G\cap B_\tau(\theta_\ell)$ by thresholding $\tilde q$ on $G$ with parameter $\alpha_1$. The MUSIC function $\tilde q$ is at least twice differentiable, so by \cite[Theorem 2.1.15]{nesterov2018lectures} (with $\mu=\lambda_d(\Psi)/2$ and $L=7\lambda_d(\Psi)/2$), gradient descent with initial point $\theta_{\ell,0}$ and step size
	\begin{equation}
		\label{eq:stepsizeball}
		h
		\leq \frac 2{\mu+L}
		= \frac{d+2}{16\pi^2 m^2},
	\end{equation} 
	produces iterates which converge at a linear rate to local minimum $\tilde\theta_\ell$. More specifically, if $h$ is equal to the right hand side of \eqref{eq:stepsizecube} and $\theta_\ell^\sharp$ is the $n$-th iterate with any initial point $\theta_{\ell,0}\in G\cap B_\tau(\theta_\ell)$, then by \cite[Theorem 2.1.15]{nesterov2018lectures} and that $|\theta_{\ell,0}-\tilde \theta_\ell|\leq 2\tau$,
	\begin{equation}
		\label{eq:numericalerrorball}
		\big|\theta_{\ell,n}-\tilde \theta_\ell\big|
		\leq \left(\frac{L/\mu-1}{L/\mu+1}\right)^n |\theta_{\ell,0}-\tilde \theta_\ell| 
		\leq \left(\frac{13}{14}\right)^n 2\tau
		\lesssim_d \frac 1 m \left(\frac{13}{14}\right)^n.
	\end{equation}
	To make this error $\epsilon/m$, where $\epsilon/m$ is any upper bound for the perturbation error $|\theta_\ell-\tilde\theta_\ell|$, we pick $n\asymp_d \log(1/\epsilon)$. This completes the proof.

	\subsection{Proof of \cref{lem:besselconstants}}
	\label{proof:besselconstants}
	
		We first calculate $\nabla W_m(\xi)$. We use the identity $\frac d {dt}(t^{-\alpha}J_\alpha(t))=-t^{-\alpha} J_{\alpha+1}(t)$, see \cite[equation (6), page 46]{watson1922treatise}
		and formula \eqref{eq:Besselkernel} to see that
		\begin{equation}
			\label{eq:nablaW}
			\nabla W_m(\xi)
			= \frac {dW_m}{dr}(|\xi|) \frac \xi {|\xi|}	
			= - \frac{2\pi m}{|B_1|} \, \frac{J_{d/2+1}(2\pi m|\xi|)}{(m |\xi|)^{d/2}} \frac \xi {|\xi|}. 
		\end{equation}
		By the power series definition of $J_\alpha(t)$, whenever $0\leq t\leq 1$,
		$$
		|J_\alpha(t)|
		\leq \left(\frac t 2 \right)^\alpha \sum_{k=0}^\infty \frac 1 {k! \Gamma(k+\alpha+1)} \left(\frac t 2 \right)^{2k}
		\leq \frac{t^\alpha}{2^\alpha\Gamma(\alpha+1)}\sum_{k=0}^\infty \frac 1 {(k!)^24^k}
		\leq \frac{e^{1/4}}{2^\alpha \Gamma(\alpha+1)} t^\alpha.
		$$
		Combining the above inequalities, whenever $2\pi m |\xi|\leq 1$, we see that
		$$
		|\nabla W_m(\xi)|
		\leq \frac{2\pi m}{|B_1|} \, \frac{e^{1/4}(2\pi m|\xi|)^{d/2+1}}{2^{d/2+1}\Gamma(d/2+2)(m |\xi|)^{d/2}}
		\leq \frac{4\pi^2 m^2 |\xi|}{d+2}. 
		$$ 
		This proves the first part of this lemma. For the second part of this lemma, notice that
		$$
		\Delta^2 W_m (0)
		= \frac{1}{|B_m|}\int_{B_m} 16 \pi^4 |\xi|^4 \, d\xi
		= \frac{|\S^{d-1}|}{|B_m|} \frac{16 \pi^4 m^{d+4}}{d+4}
		= \frac{16 \pi^4 m^4 d}{d+4},
		$$
		where we used that $|\S^{d-1}|=d |B_1|$.  
	
	\subsection{Proof of \cref{lem:orthogonalityball}}
	\label{proof:orthogonalityball}
		The proof hinges on a connection between the extreme singular values of $T_\vartheta$ and special functions. Suppose for some $\beta>0$ and all $m>0$, there exist functions $\varphi$ and $\psi$ that are continuous on $\R^d$, zero outside $(B_{\beta/m})^\circ$, their Fourier transforms belong to $L^1(\R^d)$, and $\hat\psi(\omega)\leq \mathbbm{1}_{B_m}(\omega)\leq \hat \varphi(\omega)$ for almost all $\omega\in\R^d$. The functions $\hat\psi$ and $\hat\varphi$ are called minorants and majorants of $\mathbbm{1}_{B_m}$, respectively. 
		
		The existence of appropriate minorants and majorants were established, such as \cite{vaaler1985some} for $d=1$ and \cite[Theorem 3]{holt1996beurling} for $d\geq 2$ (with parameter choices $\nu=d/2-1$, $r=m$ and $\delta=\beta d/m$ in the reference's notation). There is an absolute constant $C>0$ and function $\epsilon_0$ such that $\epsilon_0(\beta)=O(1/\beta)$
		and if $\beta\geq Cd$, then there are radial $\psi$ and $\varphi$ such that $\hat\psi(\omega)\leq \mathbbm{1}_{B_m}(\omega)\leq \hat \varphi(\omega)$ for almost all $\omega\in\R^d$ and   		
		$$
		(1-\epsilon_0(\beta))|B_m|
		\leq \psi(0)
		\leq \varphi(0)
		\leq (1+\epsilon_0(\beta))|B_m|.
		$$
		While \cite{vaaler1985some} provides a complicated explicit formula for $\epsilon_0$, it may be more convenient to use the asymptotic bound $\epsilon_0(\beta)=O(1/\beta)$ as $\beta\to\infty$. Reference \cite[Lemma 2.1]{li2025nonharmonic} shows that if $\vartheta\subset\R^d$ such that $\Delta(\vartheta)\geq \beta/m$ and $\beta\geq Cd$, then
		$$
		\sqrt{\psi(0)}
		\leq \sigma_{\min}(T_\vartheta)
		\leq \sigma_{\max}(T_\vartheta)
		\leq \sqrt{\varphi(0)}. 
		$$ 
		To complete the proof, we use that $\lambda_j(K_\vartheta) = \sigma_j^2(T_\vartheta)/|B_m|$ for each $j$.

	\subsection{Preparatory lemma used in proof of \cref{lem:besselenergy}}
	\label{proof:besselenergy}
	
	\begin{lemma} \label{lem:besselenergy2}
		For any $m>0$, $\gamma\geq 1$ and $\Gamma\subset\R^d$ such that $\Delta(\Gamma)\geq \beta/m$ and $\Gamma\cap (B_{\beta/m})^\circ = \emptyset$, we have
		$$
		\sum_{\gamma\in \Gamma} |W_m(\gamma)|^2
		\leq \frac{d\, 2^{d+1}}{\pi^2 |B_1|^2 \beta^{d+1}} \andspace
		\sum_{\gamma\in \Gamma} |\nabla W_m(\gamma)|^2
		\leq \frac{d \, 2^{d+1} 4\pi^2 m^2 }{|B_1|^2 \beta^{d+1}}.
		$$
	\end{lemma}
	\begin{proof}		
		We first concentrate on the estimate for $|W_m(\gamma)|^2$. Recall the pointwise inequality, $|J_\alpha(t)|\leq \sqrt{2/(\pi t)}$ for all $t>0$ and $\alpha>0$. Together with the explicit formula for $W_m$ in \eqref{eq:Besselkernel}, we get the pointwise bound
		$$
		|W_m(\xi)|^2
		\leq \frac{1}{\pi^2 |B_1|^2}  \frac{1}{(m|\xi|)^{d+1}}. 
		$$
		For convenience, let $C_d:=1/(\pi^2 |B_1|^2)$. 
		A calculation shows that for any $\xi\ne0$, 
		$$
		\Delta_\xi \left(\frac{1}{|\xi|^{d+1}}\right)
		= \frac{(d+1)(d+2)}{|\xi|^{d+3}} - \frac 2 {|\xi|} \frac{d+1}{|\xi|^{d+2}} 
		= \frac{d(d+1)}{|\xi|^{d+3}}
		>0. 
		$$
		This proves that $\xi\mapsto 1/|\xi|^{d+1}$ is sub-harmonic on any ball that does not intersect the origin. In view of the assumptions that $\Delta(\Gamma)\geq \beta/m$ and $\Gamma\cap (B_{\beta/m})^\circ=\emptyset$, for each $\gamma\in \Gamma$, the ball $B_{\beta/(2m)}(\gamma)$ does not contain the origin. This, along with recalling the mean value property of sub-harmonic functions, we get
		\begin{align*}
			\sum_{\gamma\in \Gamma} |W_m(\gamma)|^2
			&\leq \frac{C_d}{m^{d+1}} \sum_{\gamma\in \Gamma} \frac 1{|\gamma|^{d+1}}
			\leq \frac{C_d}{m^{d+1}} \sum_{\gamma\in \Gamma} \frac 1{|B_{\beta/(2m)}|} \int_{B_{\beta/(2m)}(\gamma)} \frac{1}{|\xi|^{d+1}} \, d\xi.
		\end{align*}
		Since $\Delta(\Gamma)\geq \beta/m$, we see that $\bigcup_{\gamma\in \Gamma} B_{\beta/(2m)}(\gamma)$ is a union of balls that overlap on a set of measure zero and is contained in the closure of $(B_{\beta/(2m)})^c$. Then continuing from the prior calculation,
		\begin{align*}
			\sum_{\gamma\in \Gamma} |W_m(\gamma)|^2
			&\leq \frac{C_d}{m^{d+1}} \frac 1{|B_{\beta/(2m)}|} \int_{(B_{\beta/(2m)})^c} \frac{1}{|\xi|^{d+1}} \, d\xi \\
			&= \frac{C}{m^{d+1}}  \frac{|\S^{d-1}|}{|B_{\beta/(2m)}|} \int_{\beta/(2m)}^\infty \frac{1}{r^2} \, dr 
			= \frac{C_d d }{m^{d+1}} \left(\frac{2m}{\beta}\right)^{d+1} 
			= \frac{d\, 2^{d+1}}{\pi^2 |B_1|^2 \beta^{d+1}}. 
		\end{align*}
		This completes the first part of the proof. 
		
		For the second estimate for $|\nabla W_m|^2$, recall the explicit formula for $\nabla W_m$ in \eqref{eq:nablaW}. Using the pointwise inequality, $|J_\alpha(t)|\leq \sqrt{2/(\pi t)}$, we obtain the upper bound
		$$
		|\nabla W_m(\xi)|^2
		\leq \frac{4m^2}{|B_1|^2} \frac{1}{(m|\xi|)^{d+1}}. 
		$$
		Since these upper bounds for $|W_m|^2$ and $|\nabla W_m|^2$ are identical up a constant factor, repeating the same argument completes the proof.	
	\end{proof}

\subsection{Proof of \cref{lem:contnoisestochastic}}
\label{proof:contnoisestochastic}

	\begin{proof}
		For convenience, define $\eta_m:=\bbone_{B_{2m}} \eta\in L^2(Q_{2m})$, which we treat as a periodic function. For any $f\in L^2(B_m)$, we let $f$ also denote its zero extension to $Q_{2m}$ and then periodically extended to $\R^d$. Recall the reflection map $R$. Letting $*$ be the convolution on $\R^d/Q_{2m}$, for each $x\in B_m$,
		$$
		(H(\eta)f)(x)
		= \int_{B_m} \eta(x +u) f(u)\, du 
		= (\eta_m* Rf)(x).
		$$
		By various algebraic properties of the Fourier transform, Plancherel's, and H\"older's, 
		\begin{equation}
			\label{eq:etahelp1}
			\|H(\eta)f\|_{L^2(B_m)}
			= \frac{1}{\sqrt{|Q_{2m}|}} \, \big\| \hat{\eta_m} \,  \hat{Rf} \, \big\|_{\ell^2(\Z^d)} 
			\leq \|\hat{\eta_m}\|_{\ell^\infty(\Z^d)} \|f\|_{L^2(B_m)}.
		\end{equation}
		Thus, it suffices to bound $\|\hat{\eta_m}\|_{\ell^\infty(\Z^d)}$. 
		
		To accomplish this, we will control the variance of $\hat{\eta_m}(k)$ and use a union bound. For convenience, define $\sigma_m:=\sigma \chi_m$, where recall that $\chi_m$ is any function such that $\chi_m=1$ on $B_{2m}$. Recall \cref{def:etastochastic} and a calculation yields
		$$
		\hat{\eta_m}(k)
		=\int_{\R^d} \eta_m(x) e^{-2\pi ik\cdot x/(4m)} \, dx
		=\int_{\R^d} \gamma(t) \, \hat {\sigma_m}\left(\tfrac k{4m}-t \right)  \, dW(t).
		$$
		The switch of integrals is justified since $\gamma \in L^2(\R^d)$ and its support is contained in a compact set implies $\gamma\in L^1(\R^d)$, while $\sigma\in L^2(\R^d)$ implies $\sigma_m\in L^1(\R^d)$. This shows that $\hat{\eta_m}(k)$ for any $k\in\Z^d$ is a mean zero normal random variable and its variance is, by Ito isometry,
		\begin{equation}
			\label{eq:etahelp2}
			\var(\hat{\eta_m}(k))
			= \E |\hat{\eta_m}(k)|^2
			= \int_{\R^d} \left| \gamma(t)  \, \hat {\sigma_m}\left(\tfrac k{4m}-t\right) \right|^2 \, dt.
		\end{equation}
		Pick a sufficiently large constant $C_d$ that depends only on dimension such that $\sum_{k\in\Z^d} C_d(1+|k|)^{-(d+1)}\leq 1$. Let $\delta\in (0,1)$ be arbitrary and define
		$$
		v_k^2:= \var(\hat{\eta_m}(k)) \log\left( \frac{(1+|k|)^{d+1}}{C_d \, \delta} \right). 
		$$
		By the union bound, definition of $v_k$, and tail bound for Gaussian random variables, we have
		\begin{align*}
			\P\left( \exists \, k\in\Z^d \text{ s.t. } |\hat{\eta_m}(k)| \geq v_k \right)
			&\leq \sum_{k\in \Z^d} \P\left(|\hat{\eta_m}(k)|\geq v_k\right) 
			\leq \sum_{k\in \Z^d} \exp\left(- \frac{v_k^2}{\var(\hat{\eta_m}(k))} \right)
			\leq \delta.
		\end{align*}
		This implies with probability at least $1-\delta$, we have $|\hat{\eta_m}(k)|\leq v_k$ for all $k\in \Z^d$ and so 
		\begin{equation*}
			\big\| \hat{\eta_m} \big\|_{\ell^\infty(\Z^d)}^2
			\leq \sup_{k\in \Z^d} v_k^2
			\lesssim_d \sup_{k\in \Z^d}   \var(\hat{\eta_m}(k)) \big(\log(1/\delta)+\log(1+|k|)\big). 
		\end{equation*}
		In view of this inequality, to finish the proof of the first statement of this lemma, it suffices to prove that
		\begin{equation}
			\label{eq:etahelp3}
			\sup_{k\in \Z^d} \var(\hat{\eta_m}(k)) \log(1+|k|)
			\lesssim_{d,\gamma} \|\sigma_m \|_{L^2(\R^d)}^2 \log(m). 
		\end{equation}
		Since $\gamma$ is assumed to be compactly supported, let $R\geq 1$ be big enough so that the support of $\gamma$ is contained in $B_R$. If $|k|\leq CRm$ for any absolute $C$, then using \eqref{eq:etahelp2}, we have
		$$
		\var(\hat{\eta_m}(k)) \log(1+|k|)
		\leq \|\gamma\|_{L^\infty(\R^d)}^2 \|\hat{\sigma_m}\|_{L^2(\R^d)}^2 \log(1+CRm)
		\lesssim_{\gamma} \|\sigma_m \|_{L^2(\R^d)}^2 \log(m). 
		$$ 
		From here onward, assume that $|k|\geq CRm$. Then for all $t\in B_R$, we have $m|\frac k{4m}-t|\geq \frac 1 4 |k|-mR\gtrsim_R |k|$. Using that the support of $\gamma$ is contained in $B_R$, \eqref{eq:etahelp2} and assumption \eqref{eq:sigmacond0}, we see that 
		\begin{equation}
			\begin{split}
			&\var(\hat{\eta_m}(k)) \log(1+|k|) \\
			&\quad \leq \|\gamma\|_{L^\infty(\R^d)}^2 \int_{\R^d} \left| \hat {\sigma_m}\left(\tfrac k{4m}-t \right) \right|^2 \log(1+m|\tfrac k{4m}-t|) \, dt  \\
			&\quad \lesssim_R \|\gamma\|_{L^\infty(\R^d)}^2 \int_{\R^d} \left| \hat{\sigma_m} \round{\tfrac{k}{4m}-t}\right|^2 \log(m) \, dt + \int_{\R^d} \left| \hat {\sigma_m}\left(\tfrac k{4m}-t \right) \right|^2  \log\round{1+\left|\tfrac{k}{4m}-t\right|} \, dt \\
			&\quad = \|\gamma\|_{L^\infty(\R^d)}^2 \int_{\R^d} \left| \hat{\sigma_m} (t)\right|^2 \log(m) \, dt + \int_{\R^d} \left| \hat {\sigma_m}(t) \right|^2  \log\round{1+|t|} \, dt \\
			&\quad \lesssim_d \|\gamma\|_{L^\infty(\R^d)}^2 \|\sigma_m \|_{L^2(\R^d)}^2 \log(m).
		\end{split} \label{eq:etahelp4}
		\end{equation}
		This proves \eqref{eq:etahelp3}.
		
		For the second part of this lemma, we verify that the constant function $\sigma(x)=\sigma$ satisfies condition \eqref{eq:sigmacond0}. From here onward, pick $\chi_m$ to be the indicator function of $B_{2m}$ and note that $\|\sigma\|_{L^2(B_m)}^2=\sigma^2 |B_m|\asymp_d \sigma^2 m^d$. Recalling that $J_\alpha$ denotes the Bessel function with parameter $\alpha$,
		$$
		|\hat{\sigma\chi_m}(\xi)|^2
		= \sigma^2 |\hat{\chi_m}(\xi)|^2 
		= \sigma^2 (2m)^d \, \frac{|J_{d/2}(4\pi m|\xi|)|^2}{|\xi|^d}.
		$$
		To verify inequality \eqref{eq:sigmacond0}, we split the integral into two cases. For the region where $|\xi|\geq 1$, using the point-wise inequality $|J_\alpha(t)|\leq \sqrt{2/(\pi t)}$, 
		$$
		\int_{|\xi|\geq 1} |\hat{\sigma \chi_m}(\xi)|^2 \log(1+|\xi|) \, d\xi 
		\lesssim_d \sigma^2 m^{d-1} \int_{|\xi|\geq 1} \frac{\log(1+|\xi|)}{|\xi|^{d+1}} \, d\xi
		\lesssim_d \sigma^2 m^{d-1}.
		$$
		On the other hand, if $|\xi|\leq 1$, then we use that $\log(1+|\xi|)\lesssim 1$ and Parseval to get 
		$$
		\int_{|\xi|\leq 1} |\hat{\sigma \chi_m}(\xi)|^2 \log(1+|\xi|) \, d\xi 
		\lesssim \sigma^2 \int_{\R^d} |\hat{\chi_m}(\xi)|^2 \, d\xi
		\asymp_d \sigma^2 m^d.
		$$
		This completes the proof.
	\end{proof}

\end{document}